\RequirePackage[2022-06-01]{latexrelease} 
\documentclass{amsart}

\usepackage{silence}

\usepackage{longtable}
\usepackage{array}

\usepackage[T1]{fontenc}
\usepackage[utf8]{inputenc}
\usepackage{newunicodechar}
\makeatletter
\def\newunicodechar#1#2{%
  \@tempswafalse
  \edef\nuc@tempa{\detokenize{#1}}%
  \if\relax\nuc@tempa\relax
  \nuc@emptyargerr
  \else
  \edef\@tempb{\expandafter\@car\nuc@tempa\@nil}%
  \nuc@check
  \if@tempswa
  \@namedef{u8:\nuc@tempa}{#2}%
  \fi
  \fi
}
\makeatother

\usepackage{amssymb}
\usepackage{mathtools}
\usepackage[mathcal]{euscript}
\usepackage{stmaryrd}
\usepackage{physics}

\usepackage{tikz}
\usetikzlibrary{cd,arrows.meta,decorations.pathreplacing,calligraphy,decorations.markings,positioning,calc}
\usepackage{tikz-cd}

\usepackage{pgfplots}
\pgfplotsset{width=7cm,compat=newest}
\pgfdeclarelayer{edgelayer}
\pgfdeclarelayer{nodelayer}
\pgfsetlayers{edgelayer,nodelayer,main}
\usepackage{mathrsfs}

\usepackage{xr}[=v5]

\usepackage[linktoc=all]{hyperref} 


\usepackage[shortlabels]{enumitem} 

\usepackage[capitalize,nosort]{cleveref} 
\crefformat{equation}{(#2#1#3)} 
\crefmultiformat{equation}{(#2#1#3)}{ and~(#2#1#3)}{, (#2#1#3)}{,
and~(#2#1#3)} 
\crefname{figure}{Figure}{Figures}
\Crefname{enumi}{}{}

\Crefname{subdefinition}{Definition}{Definitions}
\Crefname{subenumi}{Definition}{Definitions}
\AtBeginEnvironment{subdefinition}{%
  \crefalias{enumi}{subenumi}%
  \setlist[enumerate,1]{ref={\thedefinition.(\arabic*)}}%
}

\Crefname{conditionenumi}{Condition}{Conditions}

\Crefname{partenumi}{Part}{Parts}

\usepackage{etoolbox}
\Crefname{definition}{Definition}{Definitions}
\Crefname{definitionenumi}{Definition}{Definitions}
\AtBeginEnvironment{definition}{%
  \crefalias{enumi}{definitionenumi}%
  \setlist[enumerate,1]{
    label={{(\roman*)}},
    ref={\thedefinition.(\roman*)}
  }%
}

\usepackage{booktabs}
\usepackage{float}
\usepackage{wrapfig}
\usepackage{caption}
\usepackage{verbatim}

\makeatletter
\DeclareRobustCommand*\cal{\@fontswitch\relax\mathcal}
\DeclareRobustCommand*\mit{\@fontswitch\relax\mathnormal}
\DeclareRobustCommand*\frak{\@fontswitch\relax\mathfrak}
\makeatother

\usepackage{xcolor}
\usepackage[margin=2cm]{geometry}

\usepackage{graphicx}
\usepackage{import}
\usepackage{xifthen}
\usepackage{pdfpages}
\usepackage{transparent}

\mathcode`\:="603A 
\mathchardef\colon="303A 

\newcommand{\cat}[1]{\mathsf{#1}}

\newcommand{\Set}{\cat{Set}}
\newcommand{\sset}{\cat{sSet}}

\newcommand{\Top}{\cat{Top}}

\newcommand{\nerve}{\mathcal{N}}

\newcommand{\we}{\simeq}

\newcommand{\id}{\mathrm{id}}
\newcommand{\dom}{\textnormal{dom}}

\DeclareMathOperator*{\colim}{colim}
\DeclareMathOperator*{\holim}{holim}
\DeclareMathOperator*{\hocolim}{hocolim}

\newcommand{\proj}{\textnormal{proj}}

\newcommand{\inj}{\textnormal{inj}}

\newcommand{\Cech}{\textnormal{\v{C}ech}}

\newcommand{\Open}{\textnormal{Open}}
\newcommand{\pr}{\textnormal{pr}}

\newcommand{\Mor}{\textnormal{Mor}}

\DeclareMathOperator{\Sing}{Sing}
\newcommand{\unit}{\mathbf{1}}

\newcommand{\CSet}{\textnormal{C}^{\infty}\Set}
\newcommand{\CsSet}{\textnormal{C}^{\infty}\sset}

\newcommand{\field}[1]{\mathfrak{Field}_{#1}}
\newcommand{\fieldnoniso}[1]{\cat{Field}_{#1}}

\newcommand{\Cart}{\cat{C}^{\infty}\cat{Cart}}
\newcommand{\Man}{\cat{C}^{\infty}\cat{Man}}

\newcommand{\fembnoniso}[1]{\cat{C}^{\infty}\cat{Emb}_{#1}}
\newcommand{\fembcartnoniso}[1]{\cat{C}^{\infty}\cat{EmbCart}_{#1}}

\newcommand{\femb}[1]{\mathfrak{C}^{\infty}\mathfrak{Emb}_{#1}}
\newcommand{\fembcart}[1]{\mathfrak{C}^{\infty}\mathfrak{Emb}\cat{Cart}_{#1}}
\newcommand{\gemb}[1]{\mathfrak{GEmb}_{#1}}

\newcommand{\Bord}{\mathfrak{Bord}}

\newcommand{\FFT}{\mathfrak{FFT}}

\newcommand{\GCat}{\mathfrak{G}\mathfrak{Cat}}

\newcommand{\sheaf}[1]{\mathcal{#1}}

\renewcommand{\ev}{\textnormal{\textbf{ev}}}

\newcommand{\Fun}{\mathrm{Fun}}
\newcommand{\Map}{\mathrm{Map}}

\newcommand{\psh}{\mathrm{PSh}}
\newcommand{\spsh}{\mathrm{PSh}_{\Delta}}

\newcommand{\ra}{\rightarrow}

\newcommand{\RM}{\sheaf{R}_{M}}
\newcommand{\CM}{{\sheaf{C}}_{M}}

\newcommand{\Cc}{\mathfrak{C}}

\newcommand{\GL}{\mathrm{GL}}
\newcommand{\RR}{{\mathbb{R}}}
\newcommand{\NN}{{\mathbb{N}}}
\newcommand{\ZZ}{{\mathbb{Z}}}

\newcommand{\rpathgermy}{{\mathrm{RPath}}^{\textnormal{\textbf{g}}}_{M}}

\renewcommand{\op}{\mathrm{op}}
\newcommand{\Yo}{\mathscr{Y}}

\newcommand{\pullback}{\ar@{}[dr]|{\mathrm{pb}}}
\newcommand{\pushout}{\ar@{}[dr]|{\mathrm{po}}}

\makeatletter
\@addtoreset{equation}{section}
\@addtoreset{equation}{subsection}
\let\proof@qed\displaymath@qed
\let\c@subsubsection\c@equation

\makeatother

\theoremstyle{definition}
\newtheorem{theorem}[equation]{Theorem}
\newtheorem{example}[equation]{Example}
\newtheorem{definition}[equation]{Definition}

\newtheorem{lemma}[equation]{Lemma}
\newtheorem{corollary}[equation]{Corollary}
\newtheorem{proposition}[equation]{Proposition}

\newtheorem{remark}[equation]{Remark}

\newtheorem{notation}[equation]{Notation}

\newtheorem*{maintheorem}{Theorem}

\numberwithin{equation}{subsection}
\numberwithin{subsubsection}{subsection}

\makeatletter
\patchcmd{\@settitle}{\uppercasenonmath\@title}{}{}{}
\patchcmd{\@setauthors}{\MakeUppercase{\authors}}{\authors}{}{}
\makeatother

\author{Jacek Kenig, Dmitri Pavlov}
\title{Moduli spaces of geometric functorial field theories}

\begin{document}
\raggedbottom

\begin{abstract}
  We develop tools to compute moduli spaces of geometric functorial field
  theories as mapping spaces of equivariant simplicial presheaves. Given a
  \(d\)-dimensional geometric structure \(F\), presented as a presheaf on the
  site of smooth families of \(d\)-manifolds, we define its Cartesian
  realization, which is an \(\mathrm{O}(d)\)-equivariant simplicial presheaf on
  the site of Cartesian spaces. We use Cartesian realizations to present the
  moduli space of functorial field theories with geometric structure \(F\) as a
  mapping space between \(\mathrm{O}(d)\)-equivariant simplicial presheaves. In
  a companion paper, we use this result to compute the moduli space of smooth
  one-dimensional oriented Riemannian functorial field theories valued in an
  arbitrary smooth symmetric monoidal infinity-category.
\end{abstract}

\maketitle

\tableofcontents

\section{Introduction}

Among the many approaches to formalizing quantum field theory, two broad
approaches have been especially influential. They may be viewed as
field-theoretic analogues of the Heisenberg and Schr\"{o}dinger pictures of
ordinary quantum mechanics. In algebraic quantum field theory, which is closer
to the Heisenberg picture, one assigns an algebra of observables to each
spacetime region, with inclusions of regions inducing morphisms between the
corresponding algebras. In the functorial approach, which is closer to the
Schr\"{o}dinger picture, one instead assigns objects playing the role of state
spaces to spatial boundaries and morphisms playing the role of evolution
operators to the spacetime regions between them. The present paper adopts the
functorial viewpoint.

Informally, a \(d\)-dimensional functorial field theory assigns an object
\(Z(Y)\) of a target category to every closed \((d-1)\)-manifold \(Y\), and a
morphism
\[
  Z(X)\colon Z(Y_0)\longrightarrow Z(Y_1)
\]
to every \(d\)-dimensional bordism \(X\colon Y_0\rightsquigarrow Y_1\). Gluing
bordisms corresponds to composing morphisms, while disjoint union corresponds to
the monoidal product. In a fully extended theory, this assignment continues
through all dimensions: points are sent to objects, \(1\)-dimensional bordisms
to \(1\)-morphisms, and, more generally, \(k\)-dimensional bordisms with corners
to \(k\)-morphisms. These assignments are encoded by a symmetric monoidal
functor from a symmetric monoidal \((\infty,d)\)-category of bordisms to a
suitable target symmetric monoidal \((\infty,d)\)-category. The bordism category
appearing as the source of this functor is itself part of the definition. Its
objects and morphisms may be manifolds and bordisms without additional
structure, or manifolds and bordisms equipped with additional data that the
field theory is required to respect.

For topological field theories, the bordism category retains only topological
information. This is appropriate when the theory is invariant
under all deformations of the underlying manifolds, but it discards precisely
the information needed in many nontopological examples arising in physics. In
dimension \(1\),
for instance, all compact intervals are isomorphic as unoriented topological
bordisms, so a topological field theory cannot distinguish intervals of
different lengths. Quantum-mechanical time evolution, by contrast, depends on
elapsed time, which is naturally represented by the Riemannian length of the
worldline. Moreover, it is not enough merely to distinguish individual
lengths: a smoothly parametrized family of intervals should be sent to a
smoothly parametrized family of morphisms in the target. Thus, a smooth
variation of the elapsed time, or of any other external parameter, gives a
corresponding smooth variation of the evolution operator, in the sense
encoded by the target category.

A geometric functorial field theory retains this smoothly parametrized geometric
information. Bordisms may carry Riemannian metrics, maps to a target manifold,
principal bundles with connection, or other geometric structures. Both the
bordisms and their geometric structures are organized into smooth families,
and the theory assigns to them smooth families of morphisms in the target,
naturally with respect to reparametrizations. These data must also
satisfy descent with respect to open covers and remain compatible with
isotopies. The main advantage of the Grady--Pavlov framework
\cite{GP22,GP23,GP26} is that it incorporates these requirements into the
bordism and target categories themselves. It therefore preserves the local
and smoothly varying geometric information that is invisible to an ordinary
topological bordism category.

\subsection{The Grady--Pavlov framework}

Grady--Pavlov~\cite{GP26} introduce a general notion of a field on a manifold
that includes Riemannian metrics, principal bundles with connection,
symplectic and complex structures, and many more geometric structures relevant
to physics.
They construct symmetric monoidal \((\infty,d)\)-categories of bordisms
equipped with such fields. The use of smooth families and isotopies is
essential here: it ensures that the construction records both the local
variation of the fields and the higher morphisms needed for a fully extended
bordism category.

In~\cite{GP23}, Grady--Pavlov prove that extended functorial field theories
with geometric data satisfy locality. In particular, the bordism
construction is compatible with descent in the geometric structure. They
also construct classifying spaces for concordance classes of geometric
functorial field theories. This locality theorem is what allows a geometric
structure to be computed locally and then reconstructed from its values on
sufficiently simple families of manifolds.

The geometric cobordism hypothesis of~\cite{GP22} combines this locality with
the framed cobordism hypothesis. It expresses spaces of functorial field
theories with arbitrary geometric structures in terms of derived mapping
objects. This does not make the resulting mapping space automatically easy to
compute: the geometric structure appearing as its source may itself have a
complicated homotopy type. It does, however, separate the general bordism
theory from the structure-specific computation and places the latter in a
form accessible to the methods of homotopy theory.

\subsection{The main result}

The contribution of the present paper is to develop further the theory of
geometric structures in the framework of Grady--Pavlov and use the above results
to provide a simpler form of the moduli space of geometric functorial field
theories whose study is more amenable to the usual computational methods of
homotopy theory.

\begin{maintheorem}[\Cref{thm:moduli.space.reduction}]
  Let \(\cat{V}\) be a fibrant geometric symmetric monoidal
  \((\infty,d)\)-category with isotopies. Let \(F\) be a \(d\)-dimensional
  geometric structure without isotopies, and let \(\mathfrak{I}_{d}F\) be its
  isotopification (\Cref{def:isotopification.functor}). Let
  \(\FFT_{d,\cat{V}}^{\mathfrak{I}_{d}F}\) denote the smooth simplicial set of
  \(d\)-dimensional functorial field theories with geometric structure \(F\), as
  defined in \Cref{def:functorial.field.theory}. There is a natural weak
  equivalence
  \[
    \FFT_{d,\cat{V}}^{\mathfrak{I}_{d}F}\we
    \CsSet\bigl(\lambda^{\infty}_{\mathrm{O}(d)}\Cc_{d}q^{*}(F),\rho^{*}_{\infty,d}\mathfrak{q}^{*}\FFT_{d}(\cat{V})\bigr).
  \]
\end{maintheorem}

The right-hand side is formed in the \(\CsSet\)-enriched category of
\(\mathrm{O}(d)\)-equivariant smooth simplicial presheaves on \(\Cart\). The
notation \(\CsSet(-,-)\) denotes the smooth simplicial mapping object.

The source of this mapping object is constructed from \(F\) in three steps.
First, \(q^{*}\) restricts \(F\) from all \(d\)-dimensional families to the
Cartesian ones, namely those isomorphic over \(U\) to the projection
\(\RR^{d}\times U\to U\)
(\Cref{def:Cartesian.subsite.restriction.functors}). Second, the Cartesian
realization functor \(\Cc_d\) represents the values of this restriction as an
\(\mathrm{O}(d)\)-equivariant simplicial presheaf on \(\Cart\)
(\Cref{def:Cartesian.realization.functor}). Finally,
\(\lambda^{\infty}_{\mathrm{O}(d)}\) regards each of its simplicial-set values
as a smooth simplicial set that is constant in an additional smooth direction
(\Cref{def:equivariant.lambda}).

The target is constructed from the field stack \(\FFT_d(\cat{V})\)
(\Cref{def:field.theory.field.stack}). Its value on a \(d\)-dimensional family
\(p\) is the smooth simplicial set of \(\cat{V}\)-valued functorial field
theories on bordisms carrying the representable geometric structure determined
by \(p\). The functor \(\mathfrak{q}^{*}\) restricts this field stack to
Cartesian families while retaining isotopies and their higher homotopies
(\Cref{def:Cartesian.subsite.restriction.functors}). The functor
\(\rho^{*}_{\infty,d}\) then evaluates the restricted stack on the standard
families \(\RR^{d}\times U\to U\) and expresses the result as an
\(\mathrm{O}(d)\)-equivariant smooth simplicial presheaf on \(\Cart\)
(\Cref{def:rho.equivalence}).

If \(\cat{V}\) is assumed to have duals, the geometric framed cobordism
hypothesis makes the target easier to compute: after forgetting the
\(\mathrm{O}(d)\)-action, it is equivalent to the core \(\cat{V}^{\times}\),
consisting of the objects of \(\cat{V}\), the equivalences between them, and
their smooth variation in families. We remark that if \(\cat{V}\) does not have
duals, then the target is weakly equivalent to the core of the fully dualizable
part of \(\cat{V}\).

As an example, we give a further computation in dimension \(1\) and the
Riemannian geometric structure \(\RM\) (\Cref{ex:riemannian.structure}):

\begin{maintheorem}[\Cref{cor:one-dimensional.Riemannian.classification}; more
  details to appear in~{\cite[Theorem~5.0.1]{KP26}}]
  Let \(M\) be a smooth manifold, and let \(\RM\) be the Riemannian geometric
  structure of \Cref{ex:riemannian.structure}. Let \(\cat{V}\) be a fibrant
  geometric symmetric monoidal \((\infty,1)\)-category with isotopies and duals.
  There is a natural weak equivalence
  \[
    \FFT_{1,\cat{V}}^{\mathfrak{I}_{1}\RM}
    \we
    \CsSet\bigl(
      \lambda^{\infty}\nerve(\rpathgermy),
      \cat{V}^{\times}
    \bigr),
  \]
  where \(\rpathgermy\) is the Riemannian path category constructed in the
  companion paper~\cite{KP26}, and the mapping object is computed between smooth
  simplicial presheaves on Cartesian spaces.
\end{maintheorem}

For \(M=*\), the map to \(M\) carries no information, and a morphism in the
Riemannian path category is determined by its nonnegative length. Concatenation
adds lengths, so one initially obtains the additive semigroup \((\RR_{\geq
0},+)\). The Riemannian path category itself is not made into a groupoid.
Rather, isotopies supply the homotopies expressing the bordism duality
identities~\cite[Examples~5.5.5 and~5.5.7]{GP26}. Consequently, the
nonnegative-length evolution maps act by automorphisms, and the resulting action
factors through the group completion \((\RR_{\geq 0},+)^{\mathrm{gp}} \cong
(\RR,+) \). Thus, a theory for \(M=*\) determines an object of \(\cat{V}\)
equipped with a smooth one-parameter group of automorphisms, rather than merely
a one-parameter semigroup. More generally, for arbitrary \(M\), the same
argument gives the homotopy groupoid completion of the Riemannian path category.

\subsection{Targets of functorial field theories}

Regardless of the definition of the bordism category one chooses to use and
whether or not the manifolds are merely topological or equipped with geometric
structures, the crucial part of the study of functorial field theories is
defining suitable target categories. \Cref{thm:moduli.space.reduction} applies
to targets presented as geometric symmetric monoidal \((\infty,d)\)-categories
with isotopies and duals.

In physical applications, the most immediate candidate for $d=1$ is some version
of the category of Hilbert spaces and operators between them. The ordinary
category \(\cat{Hilb}\), however, does not by itself determine a target of the
required kind. One must specify an appropriate smooth enrichment and address the
required dualizability properties. In particular, the dualizable objects in the
ordinary tensor category of Hilbert spaces are finite-dimensional, whereas the
state spaces arising in quantum mechanics, such as \(L^2(M)\), are generally
infinite-dimensional. Ramzi's~\cite{Ram26} free rigid construction produces a
rigid symmetric monoidal \((\infty,1)\)-category from a given
\((\infty,1)\)-category. It does not, however, supply the smooth enrichment
required of a geometric target. For \(d>1\), one possible extension is to regard
an \((\infty,d)\)-category as an \((\infty,d-1)\)-category enriched in
\((\infty,1)\)-categories and apply Ramzi's construction to its enriched
hom-objects. However, making this construction compatible with composition and
verifying the required higher dualizability properties has yet to be worked out.

The calculation for \(M=*\) also exposes a separate analytic difficulty in
defining a suitable target for functorial field theories. For a typical positive
unbounded Hamiltonian \(H\), the Euclidean evolution operator \(\exp(-tH)\) is
bounded for \(t\geq 0\), whereas its inverse \(\exp(tH)\) is generally
unbounded, so does not define a morphism in the category of Hilbert spaces and
bounded maps. One possible solution is to replace the Hilbert space
\(\mathcal{H}\) by a suitable dense space of test vectors together with its
continuous dual, forming a
\emph{rigged Hilbert space}
\[
  \Phi\subset\mathcal{H}\subset\Phi'.
\]
The space \(\Phi\) may be chosen so that forward and inverse evolution
both preserve \(\Phi\).
Such a construction addresses the analytic problem of invertible
time evolution, while the free rigid construction addresses the categorical
problem of dualizability. A satisfactory target for quantum-mechanical
applications must incorporate the required smooth and descent properties.
We leave the construction and study of such targets to subsequent work.

The choice of target involves more than analytic considerations: it must also
retain the homological and higher-categorical information carried by the theory.
For example, a suitable smooth model of the homotopy theory of chain complexes
treats quasi-isomorphisms as equivalences and retains chain homotopies and their
higher coherences, rather than only complexes and chain maps in an ordinary
category. Nontruncated targets are also natural in the description of
\emph{twisted} or \emph{relative} field theories. In the model of
Stolz--Teichner~\cite{ST11}, a twist \(T\) assigns algebras to codimension-one
manifolds and bimodules to bordisms, while a \(T\)-twisted field theory is a
natural transformation from the trivial twist to \(T\). For an invertible twist,
the partition function on a closed bordism \(\Sigma\) takes values in the dual
of the line \(T(\Sigma)\); in families, it is a section of the resulting line
bundle rather than a scalar-valued function. This requires a higher-categorical
target in which twists, transformations between them, and the coherence data
governing gluing are retained. The Batalin–Vilkovisky formalism provides another
motivation for homological targets: fields and observables are organized into
complexes whose differentials encode gauge symmetry and the equations of motion.
These examples motivate the homotopy-coherent and derived nature of the
formalism used here.

These observations motivate three distinct continuations of the present work.
The first is to construct analytically appropriate geometric target categories,
including suitable versions of \(\cat{Hilb}\), and verify the hypotheses needed
to apply the theorem to them. The second is to compute the derived mapping
spaces presented by the main theorem and thereby classify functorial field
theories for specific targets. The third is to develop quantization procedures
compatible with the functorial framework, including integration over fields, and
compare the resulting functorial field theories with quantum field theories
already present in the physics literature. The present paper supplies the
geometric and homotopical input common to these later problems.

\subsection*{Previous work}

Smooth geometric functorial field theories in the nonextended setting were
introduced by Stolz--Teichner~\cite{ST04,ST11}.

In the fully extended topological setting, the cobordism hypothesis, as
formulated by Lurie~\cite{Lur09}, identifies the space of framed field theories
with the core of the fully dualizable part of the target category; versions with
tangential structures are described by the corresponding equivariant mapping
spaces. The fully extended geometric setting used here was developed by
Grady--Pavlov~\cite{GP22,GP23,GP26}.

In dimension \(1\), Berwick-Evans--Pavlov~\cite{BerwickEvansPavlov23} construct
a smooth \(1\)-dimensional bordism category and prove that smooth
\(1\)-dimensional topological field theories over a manifold \(M\) are
equivalent to vector bundles with connection on \(M\).
Ludewig--Stoffel~\cite{LS21} develop a general nonextended framework for
geometric field theories based on smooth families of bordisms. Their framework
is formulated for particular vector-space-valued targets, whereas the present
work describes fully extended field theories with arbitrary target smooth
\((\infty, d)\)-categories.

For a broader account of previous work on functorial field theories and the
cobordism hypothesis, see~\cite[Sections~1.1 and~1.2]{GP22}.

\subsection*{Organization of the paper and prerequisites}

We assume basic familiarity with enriched category theory, model categories,
simplicial homotopy theory, and differential geometry. The more specialized
material on
weighted colimits, descent, and smooth simplicial presheaves in model
categorical language is reviewed in
\Cref{sec:background}.

In \Cref{sec:geometric.structures}, we introduce geometric structures as field
stacks on the sites of fiberwise embeddings, both with and without isotopies. We
construct the isotopification functor and show that the relevant homotopy
theories can be computed by restricting to Cartesian families.

\Cref{sec:geom.fft} introduces the \emph{Cartesian realization functor}
\(\Cc_d\). We then use its properties together with the results of
\Cref{sec:geometric.structures} and Grady--Pavlov \cite{GP23,GP26} to obtain the
desired presentation of the moduli space of \(d\)-dimensional functorial field
theories with arbitrary geometric structure.

\subsection*{Acknowledgments}

The first author was partially supported by NCN OPUS-24 2022/47/B/ST2/03313
``Quantum Geometry and BPS states''.

\section*{Table of notation}

\begingroup
\small
\renewcommand{\arraystretch}{1.03}
\setlength{\tabcolsep}{4pt}

\begin{center}
  \begin{tabular}{@{}
      >{\raggedright\arraybackslash}p{0.24\textwidth}
      >{\raggedright\arraybackslash}p{0.22\textwidth}
      >{\raggedright\arraybackslash}p{0.48\textwidth}
    @{}}
    \toprule
    \textbf{Notation} & \textbf{Defined in} & \textbf{Meaning} \\
    \midrule
    \(\Set,\ \sset\)
    & \Cref{def:simplicial.presheaves}
    & sets and simplicial sets \\
    \(\psh(\cat{C}),\ \spsh(\cat{C})\)
    & \Cref{def:simplicial.presheaves}
    & presheaves and simplicial presheaves on \(\cat{C}\) \\
    \(\psh(\cat{C},\cat{V})\)
    & \Cref{def:simplicial.presheaves}
    & \(\cat{V}\)-valued \(\cat{V}\)-enriched presheaves \\
    \(\Yo_c\)
    & \Cref{def:simplicial.presheaves}
    & contravariant representable \(\cat{C}(-,c)\) \\
    \(\Yo^c\)
    & \Cref{def:weighted.colimit}
    & covariant representable \(\cat{C}(c,-)\) \\
    \(\CSet,\ \CsSet\)
    & \Cref{def:simplicial.presheaves}
    & smooth sets and smooth simplicial sets \\
    \(\Man,\ \Cart\)
    & \Cref{def:Man.Cart.sites}
    & smooth manifolds and Cartesian spaces \\
    \(\Sing^\infty\)
    & \Cref{def:smooth.singular.complex}
    & smooth singular complex \\
    \(\lambda^\infty,\ \ev_{\RR^0}\)
    & \Cref{def:constant.smooth.direction}
    & constant-in-the-smooth-direction functor and evaluation at~\(\RR^0\) \\
    \addlinespace[2pt]
    \(\fembnoniso{d},\ \flat\)
    & \Cref{def:fembnoniso}
    & ordinary fiberwise-embedding site and its base-space functor \\
    \(\femb{d}\)
    & \Cref{def:femb}
    & enriched fiberwise-embedding site \\
    \(\fieldnoniso{d},\ \field{d}\)
    & \Cref{def:geometric.structures.non.isotopies,def:geometric.structures}
    & geometric structures without and with isotopies \\
    \(\RM\)
    & \Cref{ex:riemannian.structure}
    & oriented Riemannian \(\sigma\)-model structure with target \(M\) \\
    \(\iota,\ \iota^\infty_!\)
    & \Cref{def:useful.functors}
    & inclusion into the site with isotopies and the induced enriched left Kan
    extension \\
    \(\lambda^\infty_{\mathrm{O}(d)}\)
    & \Cref{def:equivariant.lambda}
    & equivariant constant-in-the-smooth-direction functor \\
    \(\mathfrak{I}_d\)
    & \Cref{def:isotopification.functor}
    & isotopification functor \\
    \(\fembcartnoniso{d},\ \fembcart{d}\)
    & \Cref{def:fembcart}
    & ordinary and enriched Cartesian subsites \\
    \(q^{*},\ \mathfrak{q}^{*}\)
    & \Cref{def:Cartesian.subsite.restriction.functors}
    & functors induced by restrictions to the Cartesian subsites \\
    \addlinespace[2pt]
    \(\mathcal{K}_d\)
    & \Cref{def:Cartesian.coefficient}
    & \(\mathrm{O}(d)\)-equivariant Cartesian coefficient \\
    \(\Cc_d\)
    & \Cref{def:Cartesian.realization.functor}
    & (derived) Cartesian realization functor \\
    \(\mathfrak{R}_d\)
    & \Cref{prop:Cartesian.realization.Quillen.adjunction}
    & right adjoint to \(\Cc_d\) \\
    \(\rho_d,\ \rho_{\infty,d}^{*}\)
    & \Cref{def:rho.equivalence}
    & inclusion functor and the induced restriction equivalence \\
    \(\FFT_{d,\cat{V}}^{F}\)
    & \Cref{def:functorial.field.theory}
    & smooth simplicial set of \(F\)-structured field theories valued in
    \(\cat{V}\) \\
    \(\cat{V}^\times\)
    & \Cref{def:core.of.V}
    & invertible part, or core, of \(\cat{V}\) \\
    \(\FFT_d(\cat{V})\)
    & \Cref{def:field.theory.field.stack}
    & field stack of \(\cat{V}\)-valued field theories \\
    \bottomrule
  \end{tabular}
\end{center}
\endgroup

\section{Background and conventions}\label{sec:background}

This section collects the categorical and model-categorical background used
throughout the paper and fixes our conventions. Although most of the individual
ingredients are available in the literature, their combination in the precise
enriched model-categorical form required here is not, to our knowledge,
standard. We therefore record only the formulations needed in the subsequent
sections.

\subsection{Enrichment, tensors, and weighted colimits}\label{subsec:enrichment}
We fix our conventions for enrichment, tensors, cotensors, and ordinary and
enriched weighted colimits.

\begin{notation}\label{def:simplicial.presheaves}
  We use the following notation, following
  \textnormal{\cite[Notation~2.2.1-2]{GP26}}. Let \( \cat{C} \)
  be a small category. We write
  \[
    \psh(\cat{C})\coloneqq \Fun(\cat{C}^{\op},\Set),
    \qquad
    \spsh(\cat{C})\coloneqq \Fun(\cat{C}^{\op},\sset)
  \]
  for the categories of presheaves and simplicial presheaves on
  \( \cat{C} \), respectively. More generally, if \( \cat{V} \) is a
  category, we write
  \(\psh(\cat{C},\cat{V})\coloneqq\Fun(\cat{C}^{\op},\cat{V})\). If \(\cat{C}\)
  is instead a small \(\cat{V}\)-enriched category, then
  \(\psh(\cat{C},\cat{V})\) denotes the category of \(\cat{V}\)-enriched
  presheaves \(\Fun_{\cat{V}}(\cat{C}^{\op},\cat{V})\). Thus the same notation
  is used for ordinary and enriched presheaves; the
  intended meaning is determined by whether \(\cat{C}\) is ordinary or
  \(\cat{V}\)-enriched. If \( c\) is an object in \(\cat{C} \) (enriched or
  unenriched), we write
  \[
    \Yo_c\coloneqq \cat{C}(-,c)
  \]
  for the corresponding representable presheaf, regarded as a discrete
  simplicial presheaf when necessary.

  In the special case \( \cat{C}=\Cart \)
  (see \Cref{def:Man.Cart.sites}), we write
  \[
    \CSet\coloneqq \psh(\Cart),
    \qquad
    \CsSet\coloneqq \spsh(\Cart),
  \]
  and refer to these categories as \emph{smooth sets} and \emph{smooth
  simplicial sets}, respectively.
  For \(F\in \CSet\) and \(L\in \Cart\), we write
  \(F_L\coloneqq F(L)\).

  If \(\cat{C}\) is a small \(\CsSet\)-enriched category, we also
  write \( \psh(\cat{C},\CsSet) \) for its category of \(\CsSet\)-enriched
  presheaves.

  If \( \cat{D} \) is a \( \CSet \)-enriched category and
  \(L\in \Cart\), then an element of \( \cat{D}(x,y)_L \) is called an
  \emph{\(L\)-family of morphisms} from \(x\) to \(y\). We write
  \( \cat{D}_L \) for the ordinary category with the same objects as
  \( \cat{D} \) and with morphism sets
  \[
    \cat{D}_L(x,y)\coloneqq \cat{D}(x,y)_L.
  \]
  The category \( \cat{D}_{\RR^0} \) is called the \emph{underlying category
  of \( \cat{D} \)}, and an \( \RR^0 \)-family of morphisms is simply
  called a morphism.

  Finally, if \( f:\cat{C}\to\cat{D} \) is a \( \CSet \)-enriched
  functor between \( \CSet \)-enriched categories, we denote by
  \( f_L:\cat{C}_L\to\cat{D}_L \) the functor obtained by evaluating the
  enriched hom-objects on \(L\).
\end{notation}

\begin{remark}\label{rem:presheaves.complete.cocomplete}
  If \( \cat{C} \) is a small category and \( \cat{V} \) is a complete and
  cocomplete category, then the category \( \psh(\cat{C},\cat{V}) \) is complete
  and cocomplete,
  with limits and colimits computed objectwise. In particular,
  \( \spsh(\cat{C}) \) is complete and cocomplete, with limits and colimits
  computed objectwise.

  More generally, let \( \cat{V} \) be complete and cocomplete closed
  symmetric monoidal category, and let \( \cat{C} \) be a small \( \cat{V}
  \)-enriched
  category. Then the underlying category of the $\cat{V}$-enriched category
  \[
    \psh(\cat{C},\cat{V})
    =
    \Fun_{\cat{V}}(\cat{C}^{\op},\cat{V})
  \]
  admits small limits and colimits, computed objectwise.
\end{remark}

\begin{definition}\label{def:tensored.cotensored}
  \textnormal{\cite[Section~3.7]{Kel82}},
  \textnormal{\cite[Definitions~3.7.2--3]{Rie14}}.
  Let \( \cat{V} \) be a closed symmetric monoidal category, with internal hom
  functor \( \underline{\cat{V}}(-,-) \), viewed as enriched over itself.

  A \( \cat{V} \)-enriched category \( \cat{M} \) is \emph{tensored}
  or \emph{copowered} over \( \cat{V} \) if it is equipped with a functor
  \[
    -\otimes- : \cat{V}\times \cat{M}\to \cat{M}
  \]
  and an isomorphism
  \[
    \cat{M}(v\otimes x,y)
    \cong
    \underline{\cat{V}}\bigl(v,\cat{M}(x,y)\bigr),
  \]
  natural in \( v\in\cat{V} \) and \( x,y\in\cat{M} \).

  A \( \cat{V} \)-enriched category \( \cat{M} \) is \emph{cotensored}
  or \emph{powered} over \( \cat{V} \) if it is equipped with a functor
  \[
    (-)^{(-)} : \cat{V}^{\op}\times \cat{M}\to \cat{M}
  \]
  and an isomorphism
  \[
    \cat{M}(x,y^{v})
    \cong
    \underline{\cat{V}}\bigl(v,\cat{M}(x,y)\bigr),
  \]
  natural in \( v\in\cat{V} \) and \( x,y\in\cat{M} \).
\end{definition}

\begin{definition}\label{def:tensor.cotensor.presheaves}
  \textnormal{\cite[Section~3.7]{Kel82}},
  \textnormal{\cite[Definitions~3.7.2--3]{Rie14}}.
  Let \( \cat{V} \) be a complete closed symmetric monoidal
  category, and let \( \cat{C} \) be a small \( \cat{V} \)-enriched category.
  For \( K\in\cat{V} \) and \( X\in \psh(\cat{C},\cat{V}) \), define
  \( K\otimes X \) and \( X^{K} \) objectwise by
  \[
    (K\otimes X)(c)=K\otimes X(c),
    \qquad
    (X^{K})(c)=\underline{\cat{V}}\bigl(K,X(c)\bigr),
  \]
  for all \( c\in\cat{C} \). The enriched presheaf structures are the
  ones induced from the enriched presheaf structure on \( X \), using the
  symmetry and closed structure of \( \cat{V} \). The enrichment is
  given by the end
  \[
    \psh(\cat{C},\cat{V})(X,Y)
    \coloneqq
    \int_{c\in\cat{C}}
    \underline{\cat{V}}\bigl(X(c),Y(c)\bigr),
  \]
  as in \cite[Section~2.2]{Kel82}; this end exists by completeness of
  \(\cat{V}\). Thus
  \( \psh(\cat{C},\cat{V}) \) is enriched, tensored, and cotensored over \(
  \cat{V} \).
  Equivalently, there are natural isomorphisms in
  \( \cat{V} \)
  \[
    \psh(\cat{C},\cat{V})(K\otimes X,Y)
    \cong
    \underline{\cat{V}}\bigl(
      K,\psh(\cat{C},\cat{V})(X,Y)
    \bigr)
  \]
  and
  \[
    \psh(\cat{C},\cat{V})(X,Y^{K})
    \cong
    \underline{\cat{V}}\bigl(
      K,\psh(\cat{C},\cat{V})(X,Y)
    \bigr).
  \]
  In particular, for \( \cat{V}=\sset \) this gives
  \[
    (K\otimes X)(c)=K\times X(c),
    \qquad
    (X^{K})(c)=\Map_{\sset}\bigl(K,X(c)\bigr).
  \]
\end{definition}

\begin{definition}\label{def:weighted.colimit}
  \textnormal{\cite[Section~3.10]{Kel82}},
  \textnormal{\cite[Definition~7.4.1]{Rie14}}.
  Let \( \cat{V} \) be a closed symmetric monoidal category, let
  \( \cat{C} \) be a small \( \cat{V} \)-enriched category, and let
  \( \cat{M} \) be a cocomplete \( \cat{V} \)-enriched category tensored over
  \( \cat{V} \). For enriched functors
  \[
    F:\cat{C}^{\op}\to\cat{V},
    \qquad
    G:\cat{C}\to \cat{M},
  \]
  the \emph{weighted colimit} of \(G\) with weight \(F\) is the object
  \(F\otimes_{\cat{C}}G\in\cat{M}\) characterized by the natural isomorphism
  \[
    \cat{M}\bigl(F\otimes_{\cat{C}}G,m\bigr)
    \cong
    \psh(\cat{C},\cat{V})
    \bigl(F,\cat{M}(G(-),m)\bigr),
  \]
  for \(m\in\cat{M}\). It is computed by the enriched coend
  \[
    F\otimes_{\cat{C}}G
    \cong
    \int^{c\in \cat{C}} F(c)\otimes G(c).
  \]
  We will use the enriched co-Yoneda
  isomorphisms
  \[
    F\otimes_{\cat{C}}\Yo^{c}
    \cong
    F(c)
  \]
  and, for \(A\in\cat{M}\),
  \[
    F\otimes_{\cat{C}}\bigl(\Yo^{c}\otimes A\bigr)
    \cong
    F(c)\otimes A,
  \]
  where \(\Yo^{c}=\cat{C}(c,-)\) and
  \((\Yo^{c}\otimes A)(d)=\cat{C}(c,d)\otimes A\).
\end{definition}

\subsection{Model structures on enriched presheaves and their
localizations}\label{subsec:model.categories}

Model categories provide the technical language in which one can
reliably perform computations with \((\infty,d)\)-categories. For this reason,
Grady--Pavlov~\cite{GP23,GP22,GP26} make systematic use of model
categorical presentations, and we follow the same strategy here. Using model
categories for our computations is not strictly necessary, but constructions
with enriched presheaves, weighted
colimits, derived mapping objects, and Bousfield localizations retain formulas
close to those of ordinary category theory. A quasicategorical treatment would
instead require systematic use of (co)cartesian and left/right fibrations.

\subsubsection{Enriched model categories}
Enriched model categories allow us to perform homotopical constructions while
retaining mapping objects in the enriching category, which in our applications
record smooth families.

\begin{definition}\label{def:adjectives.for.model.categories}
  Let \(\cat{M}\) be a model category.
  \begin{itemize}
    \item \textnormal{\cite[Definition~13.1.1]{Hir03}}.
      We say that \(\cat{M}\) is \emph{left proper} if pushouts of weak
      equivalences
      along cofibrations are weak equivalences, i.e., for every pushout square
      \[
        \begin{tikzcd}
          A \ar[r,rightarrowtail] \ar[d, "\sim"'] & B \ar[d] \\
          A' \ar[r,rightarrowtail] & B'
        \end{tikzcd}
      \]
      in which the map \(A\to B\) is a cofibration and \(A\to A'\) is a weak
      equivalence, the
      induced map \(B\to B'\) is a weak equivalence.

    \item \textnormal{\cite[Definition~A.2.6.1]{Lur17b}}.
      We say that \(\cat{M}\) is \emph{combinatorial} if it is cofibrantly
      generated
      and its underlying category is locally presentable.

    \item \textnormal{\cite[Definition~1.21]{Bar10}}.
      We say that \(\cat{M}\) is \emph{tractable} if it is combinatorial and
      admits
      sets \(I\) and \(J\) of generating cofibrations and generating acyclic
      cofibrations, respectively, whose domains are cofibrant.

  \end{itemize}

  If \(\cat{V}\) is a symmetric monoidal model category and \(\cat{M}\) is a
  \(\cat{V}\)-model category (\Cref{def:v.model.category}), then we say that
  \(\cat{M}\) is left proper, combinatorial, or tractable as a \(\cat{V}\)-model
  category if its underlying ordinary model category has the corresponding
  property above.
\end{definition}

\begin{definition}\label{def:left.quillen.bifunctor}
  \textnormal{\cite[Definition~4.2.1]{Hov99}},
  \textnormal{\cite[Definition~1.27.1]{Bar10}}.
  Let \(\cat{M}\), \(\cat{N}\), and \(\cat{P}\) be model categories. An
  \emph{adjunction of two variables} consists of functors
  \[
    -\otimes-:\cat{M}\times\cat{N}\to\cat{P},
    \qquad
    \Mor_{\mathrm{r}}:\cat{N}^{\op}\times\cat{P}\to\cat{M},
    \qquad
    \Mor_{\mathrm{\ell}}:\cat{M}^{\op}\times\cat{P}\to\cat{N},
  \]
  together with natural isomorphisms
  \[
    \cat{P}(X\otimes Y,Z)
    \cong
    \cat{M}(X,\Mor_{\mathrm{r}}(Y,Z))
    \cong
    \cat{N}(Y,\Mor_{\mathrm{\ell}}(X,Z)).
  \]
  Thus, for each \(Y\in\cat{N}\), the functor
  \(-\otimes Y:\cat{M}\to\cat{P}\) is left adjoint to
  \(\Mor_{\mathrm{r}}(Y,-):\cat{P}\to\cat{M}\), while for each
  \(X\in\cat{M}\), the functor \(X\otimes-:\cat{N}\to\cat{P}\) is left
  adjoint to \(\Mor_{\mathrm{\ell}}(X,-):\cat{P}\to\cat{N}\). The
  subscripts record whether the right- or left-hand tensor variable is held
  fixed.
  It is a \emph{left Quillen bifunctor}, or, equivalently, a \emph{Quillen
    adjunction
  of two variables}, if for every cofibration \(i:X\to X'\) in \(\cat{M}\)
  and every cofibration \(j:Y\to Y'\) in \(\cat{N}\), the pushout-product
  \[
    i\square j:
    (X'\otimes Y)\sqcup_{X\otimes Y}(X\otimes Y')
    \longrightarrow
    X'\otimes Y'
  \]
  is a cofibration in \(\cat{P}\), and is a weak equivalence whenever
  \(i\) or \(j\) is a weak equivalence.
\end{definition}

\begin{notation}\label{def:cofibrant.fibrant.replacement}
  \textnormal{\cite[Paragraph before Lemma 1.1.9]{Hov99}}.
  Let \(\cat{M}\) be a model category. We write
  \[
    QX \to X,
    \qquad
    Y \to RY
  \]
  for a cofibrant replacement of an object \(X\in \cat{M}\) and a
  fibrant replacement of an object \(Y\in \cat{M}\), respectively.
\end{notation}

\begin{definition}\label{def:symmetric.monoidal.model.category}
  \textnormal{\cite[Definition~4.2.6]{Hov99}},
  \textnormal{\cite[Definition~1.27.2]{Bar10}}.
  A symmetric monoidal model category is a closed symmetric monoidal category
  \((\cat{V},\otimes,\unit_{\cat{V}},\underline{\cat{V}}(-,-))\) equipped
  with a model structure such that
  \begin{enumerate}
    \item the tensor product and internal homs form a Quillen adjunction
      of two variables
      \[
        \otimes:\cat{V}\times\cat{V}\to\cat{V},
        \qquad
        \underline{\cat{V}}(-,-):\cat{V}^{\op}\times\cat{V}\to\cat{V};
      \]
    \item for some cofibrant replacement \(Q\unit_{\cat{V}}\to\unit_{\cat{V}}\)
      as in \Cref{def:cofibrant.fibrant.replacement}, and every object
      \(A\in\cat{V}\), the composite
      \[
        Q\unit_{\cat{V}}\otimes A
        \longrightarrow
        \unit_{\cat{V}}\otimes A
        \longrightarrow
        A
      \]
      is a weak equivalence.
  \end{enumerate}
\end{definition}

\begin{definition}\label{def:Cartesian.model.category}
  \textnormal{\cite[Definition~1.27.3]{Bar10}}.
  A Cartesian model category is a symmetric monoidal model category whose
  underlying symmetric monoidal structure is the cartesian
  monoidal structure.
\end{definition}

\begin{definition}\label{def:derived.mapping.space}
  \textnormal{\cite[Theorem 5.4.9]{Hov99}},
  \textnormal{\cite[Section 17.4]{Hir03}},
  \textnormal{\cite[Definition 1.35]{Bar10}}.
  Let \(\cat{M}\) be a \(\cat{V}\)-model category. We write
  \[
    \Map_{\cat{M}}^{\cat{V}}(X,Y)
  \]
  for its enriched mapping object and
  \[
    \mathbb{R}\Map_{\cat{M}}^{\cat{V}}(X,Y)
    \simeq
    \Map_{\cat{M}}^{\cat{V}}(QX,RY)
  \]
  for its derived mapping object, computed using cofibrant and fibrant
  replacements. In this paper, the enriching categories are \(\sset\) and
  \(\CsSet\). Since \(\CsSet\) is Quillen equivalent to \(\sset\) by
  \Cref{prop:Cartesian.model.str.smooth.simp.sets}, we use the terms
  \emph{derived mapping object} and \emph{derived mapping space} interchangeably
  for both simplicial and smooth simplicial enrichments. When the enrichment is
  clear, we write
  \[
    \Map_{\cat{M}}(X,Y),
    \qquad
    \mathbb{R}\Map_{\cat{M}}(X,Y).
  \]
  The underlying simplicial homotopy type of a derived
  \(\CsSet\)-enriched mapping object is obtained by applying
  \(\Sing^\infty\).
\end{definition}

\begin{definition}\label{def:v.model.category}
  \textnormal{\cite[Definition~4.2.18]{Hov99}},
  \textnormal{\cite[Definition~1.27.4]{Bar10}}.
  Let \(\cat{V}\) be a symmetric monoidal model category. A
  \emph{\(\cat{V}\)-model category} is a complete and cocomplete
  \(\cat{V}\)-enriched category \(\cat{M}\), tensored and cotensored over
  \(\cat{V}\), equipped with a model structure on its underlying ordinary
  category, such that the following conditions hold.
  Write \(\cat{V}_{0}\) and \(\cat{M}_{0}\) for the underlying ordinary
  categories.
  \begin{enumerate}
    \item The underlying tensor functor
      \[
        -\otimes-:
        \cat{V}_{0}\times\cat{M}_{0}
        \longrightarrow
        \cat{M}_{0}
      \]
      is a left Quillen bifunctor in the sense of
      \Cref{def:left.quillen.bifunctor}. In the notation of that definition, the
      three model categories \(\cat{M},\cat{N},\cat{P}\) there are here
      \(\cat{V}_{0},\cat{M}_{0},\cat{M}_{0}\), respectively, and the two right
      adjoints are
      \[
        \Mor_{\mathrm{r}}(X,Y)
        =
        \Map_{\cat{M}}^{\cat{V}}(X,Y),
        \qquad
        \Mor_{\mathrm{\ell}}(K,Y)
        =
        Y^{K}.
      \]
      Equivalently, for every cofibration \(i:K\to L\) in \(\cat{V}\)
      and every cofibration \(j:X\to Y\) in \(\cat{M}\), the
      pushout-product
      \[
        i\square j:
        (L\otimes X)\sqcup_{K\otimes X}(K\otimes Y)
        \longrightarrow
        L\otimes Y
      \]
      is a cofibration in \(\cat{M}\), and is a weak equivalence whenever
      \(i\) or \(j\) is a weak equivalence.
    \item For some cofibrant replacement
      \(Q\unit_{\cat{V}}\to\unit_{\cat{V}}\), as in
      \Cref{def:cofibrant.fibrant.replacement}, and every object
      \(X\in\cat{M}\), the composite
      \[
        Q\unit_{\cat{V}}\otimes X
        \longrightarrow
        \unit_{\cat{V}}\otimes X
        \longrightarrow
        X
      \]
      is a weak equivalence.
  \end{enumerate}
\end{definition}

\begin{definition}\label{def:v.quillen.adjunction}
  \textnormal{\cite[Definition~1.27.6]{Bar10}}.
  Let \(\cat{M}\) and \(\cat{N}\) be \(\cat{V}\)-model categories. A
  \(\cat{V}\)-enriched adjunction
  \[
    F:\cat{M}\rightleftarrows\cat{N}:G
  \]
  is a \emph{\(\cat{V}\)-enriched Quillen adjunction} if its underlying
  ordinary adjunction is a Quillen adjunction. In this case, the functor \(F\)
  is called
  a \emph{left Quillen \(\cat{V}\)-functor} and \(G\) is called a \emph{right
    Quillen
  \(\cat{V}\)-functor}.
\end{definition}

\begin{definition}\label{def:quillen.equivalence.criterion}
  \textnormal{\cite[Chapter~I, Section~4]{Qui67}},
  \textnormal{\cite[Definition 1.3.12]{Hov99}},
  \textnormal{\cite[Definition~8.5.20]{Hir03}}.
  Let
  \[
    F:\cat{M}\rightleftarrows \cat{N}:G
  \]
  be a Quillen adjunction. We say that it is a \emph{Quillen
  equivalence} if for every cofibrant object \(X\in\cat{M}\) and every
  fibrant object \(Y\in\cat{N}\), a morphism
  \[
    FX \to Y
  \]
  is a weak equivalence in \(\cat{N}\) if and only if its adjoint
  \[
    X \to GY
  \]
  is a weak equivalence in \(\cat{M}\). For an enriched Quillen
  adjunction, this definition is applied to the underlying ordinary Quillen
  adjunction.
\end{definition}

\begin{definition}\label{def:v.quillen.equivalence}
  \textnormal{\cite[Definition~1.27.6]{Bar10}}.
  A \(\cat{V}\)-enriched Quillen adjunction
  \[
    F:\cat{M}\rightleftarrows\cat{N}:G
  \]
  is a \emph{\(\cat{V}\)-enriched Quillen equivalence} if its underlying
  Quillen adjunction is a Quillen equivalence.
\end{definition}

\begin{definition}\label{def:derived.functors}
  \textnormal{\cite[Definition~8.5.11]{Hir03}}.
  Let
  \[
    F:\cat{M}\rightleftarrows\cat{N}:G
  \]
  be a Quillen adjunction. Choose a functorial cofibrant replacement
  \(Q\) in \(\cat{M}\) and a functorial fibrant replacement \(R\) in
  \(\cat{N}\). The composite
  \[
    \mathbb{L}F\coloneqq FQ:\cat{M}\longrightarrow\cat{N}
  \]
  is the \emph{left derived functor} of \(F\). Thus, for a
  morphism \(f:X\to Y\) in \(\cat{M}\), the left derived morphism is
  \[
    F(Qf):F(QX)\longrightarrow F(QY)
  \]
  in \(\cat{N}\). Dually, the composite
  \[
    \mathbb{R}G\coloneqq GR:\cat{N}\longrightarrow\cat{M}
  \]
  is the \emph{right derived functor} of \(G\). Thus, for a
  morphism \(g:X\to Y\) in \(\cat{N}\), the right derived morphism is
  \[
    G(Rg):G(RX)\longrightarrow G(RY)
  \]
  in \(\cat{M}\).
\end{definition}

\subsubsection{Enriched left Bousfield localization}
Left Bousfield localization is the model-categorical analogue of reflective
localization; for combinatorial model categories, its theory parallels that
of reflective localizations of locally presentable categories developed by
Adámek--Rosický~\cite{AdamekRosicky94}.

\begin{definition}\label{def:enriched.left.bousfield.localization}
  \textnormal{\cite[Definition~4.42]{Bar10}}.
  Let \(\cat{V}\) be a symmetric monoidal model category, let
  \(\cat{M}\) be a \(\cat{V}\)-model category, and let \(S\) be a class of
  morphisms in \(\cat{M}\). A \emph{\(\cat{V}\)-enriched left Bousfield
  localization} of \(\cat{M}\) with respect to \(S\) is a
  \(\cat{V}\)-model category \(L_S\cat{M}\), together with a
  \(\cat{V}\)-enriched left Quillen functor
  \[
    j:\cat{M}\longrightarrow L_S\cat{M},
  \]
  which is initial among \(\cat{V}\)-enriched left Quillen functors
  \(\varphi:\cat{M}\to\cat{N}\) to \(\cat{V}\)-model categories \(\cat{N}\)
  whose left derived functors send every morphism in \(S\) to a weak
  equivalence in \(\cat{N}\).
\end{definition}

\begin{definition}\label{def:enriched.s.local.objects.equivalences}
  \textnormal{\cite[Definition~4.45]{Bar10}}.
  Let \(\cat{V}\) be a symmetric monoidal model category, let \(\cat{M}\) be a
  \(\cat{V}\)-model category, and let \(S\) be a class of morphisms in
  \(\cat{M}\). An object \(Z\in\cat{M}\) is \emph{enriched \(S\)-local} if,
  for every morphism \(s:A\to B\) in \(S\), the induced map
  \[
    s^{*}:
    \mathbb{R}\Map_{\cat{M}}^{\cat{V}}(B,Z)
    \longrightarrow
    \mathbb{R}\Map_{\cat{M}}^{\cat{V}}(A,Z)
  \]
  is a weak equivalence in \(\cat{V}\).

  A morphism \(f:X\to Y\) in \(\cat{M}\) is an
  \emph{enriched \(S\)-local equivalence} if, for every enriched
  \(S\)-local object \(Z\), the induced map
  \[
    f^{*}:
    \mathbb{R}\Map_{\cat{M}}^{\cat{V}}(Y,Z)
    \longrightarrow
    \mathbb{R}\Map_{\cat{M}}^{\cat{V}}(X,Z)
  \]
  is a weak equivalence in \(\cat{V}\).
\end{definition}

\begin{theorem}\label{thm:enriched.left.bousfield.localization.concrete}
  \textnormal{\cite[Theorem~4.46 and Proposition~4.47]{Bar10}},
  \textnormal{\cite[Proposition~2.3.10]{GP26}}.
  For the unenriched statement, see
  \cite[Theorem~4.7]{Bar10} and \cite[Theorem~4.1.1]{Hir03}.
  Let \(\cat{V}\) be a tractable symmetric monoidal model category, let
  \(\cat{M}\) be a left proper combinatorial \(\cat{V}\)-model category, and
  let \(S\) be a set of morphisms in \(\cat{M}\). Then the
  \(\cat{V}\)-enriched left Bousfield localization of \(\cat{M}\) at \(S\)
  exists and is a left proper combinatorial \(\cat{V}\)-model category.
  Moreover:
  \begin{enumerate}
    \item its underlying \(\cat{V}\)-enriched category is \(\cat{M}\);
    \item its weak equivalences are precisely the enriched \(S\)-local
      equivalences;
    \item its cofibrations and acyclic fibrations coincide with those of
      \(\cat{M}\);
    \item its fibrant objects are precisely the fibrant enriched \(S\)-local
      objects of \(\cat{M}\).
  \end{enumerate}
  If \(\cat{M}\) is tractable, then so is its localization.
\end{theorem}

\subsubsection{Model structures on enriched diagram categories}
Injective and projective model structures allow us to derive constructions on
enriched presheaves while retaining objectwise control of weak equivalences.

\begin{definition}\label{def:injective.model.structure}
  \textnormal{\cite[Theorem~2.3]{Jar87}},
  \textnormal{\cite[Section~2]{DHI04}},
  \textnormal{\cite[Theorem~2.30]{Bar10}},
  \textnormal{\cite[Proposition A.2.8.2]{Lur17b}}.
  Let \(\cat{V}\) be a combinatorial model category, and let
  \(\cat{C}\) be a small category. The combinatorial injective model structure
  on
  \[
    \psh(\cat{C},\cat{V})
    =
    \Fun(\cat{C}^{\op},\cat{V})
  \]
  is the model structure whose weak equivalences and cofibrations are defined
  objectwise. Thus a morphism \(X\to Y\) is a weak equivalence if, for every
  \(c\in\cat{C}\), the map
  \[
    X(c)\to Y(c)
  \]
  is a weak equivalence in \(\cat{V}\), and it is a cofibration if, for every
  \(c\in\cat{C}\), the map
  \[
    X(c)\to Y(c)
  \]
  is a cofibration in \(\cat{V}\). The fibrations are the morphisms having the
  right lifting property with respect to all morphisms that are both
  cofibrations and weak equivalences in this model structure.

  Under the same assumptions as above there exists (also by
  \cite[Proposition~A.2.8.2]{Lur17b}) a combinatorial projective
  model structure on \(\psh(\cat{C}, \cat{V})\), where weak equivalences and
  fibrations are defined objectwise.
\end{definition}

\begin{theorem}\label{thm:enriched.injective.projective.model.structures}
  \textnormal{\cite[Theorem~4.4]{Mos18}; the enriched assertion follows from
  \cite[Theorem~5.4]{Mos18}.}
  Let \(\cat{V}\) be a combinatorial symmetric monoidal model category, let
  \(\cat{M}\) be a combinatorial \(\cat{V}\)-model category, and let
  \(\cat{C}\) be a small \(\cat{V}\)-enriched category. For \(c,c'\in\cat{C}\),
  write
  \[
    -\otimes \cat{C}(c,c'):\cat{M}\longrightarrow \cat{M}
  \]
  for tensoring in \(\cat{M}\) by the hom-object
  \(\cat{C}(c,c')\in\cat{V}\).

  If, for all \(c,c'\in\cat{C}\), the functor
  \[
    -\otimes \cat{C}(c,c'):\cat{M}\longrightarrow \cat{M}
  \]
  preserves cofibrations, then
  \[
    \Fun_{\cat{V}}(\cat{C},\cat{M})
  \]
  admits the injective model structure. Its weak equivalences and cofibrations
  are defined objectwise, and its fibrations are determined by the corresponding
  lifting property.

  If, for all \(c,c'\in\cat{C}\), the functor
  \[
    -\otimes \cat{C}(c,c'):\cat{M}\longrightarrow \cat{M}
  \]
  preserves trivial cofibrations, then
  \[
    \Fun_{\cat{V}}(\cat{C},\cat{M})
  \]
  admits the projective model structure. Its weak equivalences and fibrations
  are defined objectwise, and its cofibrations are determined by the
  corresponding lifting property.

  Moreover, whenever one of these model structures exists, it is again a
  \(\cat{V}\)-model category. Its \(\cat{V}\)-enrichment is given by
  \[
    \Fun_{\cat{V}}(\cat{C},\cat{M})(F,G)
    \coloneqq
    \int_{c\in\cat{C}}
    \Map_{\cat{M}}^{\cat{V}}\bigl(F(c),G(c)\bigr),
  \]
  by the same end construction as in
  \Cref{def:tensor.cotensor.presheaves}, and its tensors and cotensors are
  computed objectwise.

  In particular, applying this theorem to \(\cat{C}^{\mathrm{op}}\) and
  \(\cat{M}=\cat{V}\) gives the corresponding injective and projective model
  structures on
  \[
    \psh(\cat{C},\cat{V})
    =
    \Fun_{\cat{V}}(\cat{C}^{\mathrm{op}},\cat{V}).
  \]
\end{theorem}

\begin{remark}
  The hypotheses of
  \Cref{thm:enriched.injective.projective.model.structures} hold for
  \(\cat{V}=\sset\) and \(\cat{V}=\CsSet\). Both are combinatorial Cartesian
  model categories in which every object is cofibrant; hence tensoring by any
  enriched hom-object preserves cofibrations and trivial cofibrations. Thus
  the enriched injective and projective diagram model structures used below
  exist; see~\cite[Remark~5.5]{Mos18}.
\end{remark}

\begin{remark}\label{rem:injective.presheaves.cofibrant}
  In the cases \(\cat{V}=\sset\) and \(\cat{V}=\CsSet\), cofibrations are
  monomorphisms. Hence every object of \(\psh(\cat{C},\cat{V})\) is
  cofibrant in the injective model structure: indeed, for every
  \(X\in\psh(\cat{C},\cat{V})\), the map from the initial presheaf
  \[
    \varnothing\to X
  \]
  is objectwise the monomorphism \(\varnothing\to X(c)\).

  This is one reason why the injective model structure is the natural choice
  for the arguments below. In functorial field theory one is typically
  interested in derived mapping objects out of bordism categories or presheaves
  encoding geometric structures.
  Since all objects are cofibrant in the cases used below, these sources
  require no cofibrant replacement.
\end{remark}

\subsubsection{Smooth simplicial sets}
Smooth simplicial sets allow homotopy types to vary smoothly with a Cartesian
parameter.

\begin{definition}\label{def:smooth.singular.complex}
  \textnormal{\cite[Definition~2.2.3]{GP26}}.
  For \(n\geq 0\), let
  \[
    \Delta^n_{\mathrm{e}}
    \coloneqq
    \left\{(t_0,\ldots,t_n)\in\RR^{n+1}
      \middle|
      \sum_{i=0}^{n}t_i=1
    \right\}
  \]
  be the extended smooth \(n\)-simplex, regarded as an object of \(\Cart\)
  (\Cref{def:Man.Cart.sites}). The \emph{smooth singular complex} of
  \(X\in\CsSet\) is the simplicial set
  \[
    \Sing^\infty(X)_n \coloneqq X(\Delta^n_{\mathrm{e}})_n.
  \]
  This is the diagonal of a bisimplicial set: its simplicial operators are
  induced jointly by the cosimplicial structure maps of the extended smooth
  simplices and by the simplicial structure maps of \(X\).
\end{definition}

\begin{proposition}\label{prop:Cartesian.model.str.smooth.simp.sets}
  \textnormal{\cite[Proposition~2.2.4--2.2.5]{GP26}},
  \textnormal{\cite[Theorem~12.7]{PavlovSmoothOka}}.
  The category \( \CsSet \) admits a left proper combinatorial Cartesian
  model structure in which:
  \begin{enumerate}
    \item cofibrations are monomorphisms of simplicial presheaves;
    \item weak equivalences are the maps \(f:X\to Y\) such that
      \[
        \Sing^\infty(f):\Sing^\infty(X)\longrightarrow \Sing^\infty(Y)
      \]
      is a weak equivalence of simplicial sets.
  \end{enumerate}
  We use this model structure on \( \CsSet \) throughout the paper.

  Moreover, the smooth realization--singular complex adjunction
  \[
    \sset
    \rightleftarrows
    \CsSet
  \]
  is a Quillen equivalence for this model structure.
\end{proposition}

\begin{proposition}\label{prop:smooth.model.as.localization}
  Let
  \[
    I_{\RR}
    \coloneqq
    \left\{
      \Yo_L\otimes\RR
      \xrightarrow{\id_{\Yo_L}\otimes {!}}
      \Yo_L\otimes\RR^0
      \cong\Yo_L
      \ \middle|\ L\in\Cart
    \right\},
  \]
  where \({!}:\RR\ra\RR^0\) denotes the morphism of representable smooth
  simplicial sets induced by the collapse map \(\RR \ra \RR^0\) of Cartesian
  spaces.
  The \(\RR\)-local injective model structure on \(\CsSet\) from
  \Cref{prop:Cartesian.model.str.smooth.simp.sets} is the left
  Bousfield localization of the injective model structure at
  \(I_{\RR}\).
\end{proposition}

\begin{proof}
  The localization at \(I_{\RR}\) is the \(\RR\)-local model structure
  of \cite[Definition~2.5]{Bun22}. By
  \cite[Theorem~5.7]{Bun22}, its weak equivalences are precisely the
  morphisms sent to weak equivalences of simplicial sets by the smooth
  singular complex. Since left Bousfield localization does not change
  the cofibrations, this is the model structure of
  \Cref{prop:Cartesian.model.str.smooth.simp.sets}.
\end{proof}

\begin{definition}\label{def:constant.smooth.direction}
  Let
  \[
    \lambda^\infty:\sset\longrightarrow \CsSet
  \]
  denote the functor that regards a simplicial set as a smooth simplicial
  set constant in the smooth direction: \( \lambda^\infty(K)_L=K \), for
  \(L\in\Cart\).
  The evaluation at the terminal Cartesian space \( \RR^0 \) defines a functor
  \[
    \ev_{\RR^0}:\CsSet\longrightarrow\sset,
    \qquad
    X\longmapsto X(\RR^0).
  \]
  The functor \( \lambda^\infty \) is left adjoint to
  \( \ev_{\RR^0} \). Moreover, this is a Quillen equivalence by
  Bunk~\cite[Lemma~2.17]{Bun22}.

  More generally, for any small category \( \cat{C} \), we use the same
  notation for the objectwise extension
  \[
    \lambda^\infty:
    \spsh(\cat{C})
    \longrightarrow
    \psh(\cat{C},\CsSet),
    \qquad
    (\lambda^\infty X)(c)_L=X(c).
  \]
  Its right adjoint is the objectwise evaluation functor
  \[
    \ev_{\RR^0}:
    \psh(\cat{C},\CsSet)
    \longrightarrow
    \spsh(\cat{C}),
    \qquad
    (\ev_{\RR^0}Y)(c)=Y(c)(\RR^0).
  \]
\end{definition}

\subsection{Homotopy colimits and limits}\label{subsec:homotopy.colimits}
Homotopy weighted colimits allow us to extend enriched diagrams from
representable presheaves to arbitrary presheaves in a homotopy-coherent way.

\begin{proposition}\label{prop:weighted.colimit.quillen.bifunctor}
  \textnormal{\cite[Theorem 21.1]{Shu06}},
  \textnormal{\cite[Proposition~A.2.9.26~and~Remark~A.2.9.27]{Lur17b}},
  \textnormal{\cite[Theorem~3.2]{Gam10}}.
  Let \(\cat{V}\) be a combinatorial symmetric monoidal model category. Let
  \(\cat{C}\) be a small \(\cat{V}\)-enriched category, and let
  \(\cat{M}\) be a combinatorial \(\cat{V}\)-model category.
  Assume that
  \[
    \psh(\cat{C},\cat{V})
    =
    \Fun_{\cat{V}}(\cat{C}^{\op},\cat{V})
  \]
  admits the injective model structure and that
  \[
    \Fun_{\cat{V}}(\cat{C},\cat{M})
  \]
  admits the projective model structure.
  Then the weighted
  colimit bifunctor
  \[
    -\otimes_\cat{C} - :
    \psh(\cat{C},\cat{V})_{\inj}
    \times
    \Fun_{\cat{V}}(\cat{C},\cat{M})_{\proj}
    \longrightarrow
    \cat{M}
  \]
  is a \(\cat{V}\)-enriched left Quillen bifunctor.

  In particular, if \(G\in \Fun_{\cat{V}}(\cat{C},\cat{M})\) is projectively
  cofibrant, then
  \[
    -\otimes_\cat{C} G :
    \psh(\cat{C},\cat{V})_{\inj}
    \longrightarrow
    \cat{M}
  \]
  is a \(\cat{V}\)-enriched left Quillen functor. Dually, if
  \(F\in\psh(\cat{C},\cat{V})\) is injectively cofibrant, then
  \[
    F\otimes_\cat{C} - :
    \Fun_{\cat{V}}(\cat{C},\cat{M})_{\proj}
    \longrightarrow
    \cat{M}
  \]
  is a \(\cat{V}\)-enriched left Quillen functor.
\end{proposition}

\begin{proof}
  The proof is the \(\cat{V}\)-enriched version of the pointwise
  right-adjoint argument of~\cite[Theorem~3.2]{Gam10}. The weighted-colimit
  bifunctor is \(\cat{V}\)-enriched and left
  adjoint in each variable. We describe the right adjoint in the second
  variable. For \(F\in\psh(\cat{C},\cat{V})\) and \(X\in\cat{M}\), let
  \[
    X^{F}:\cat{C}\longrightarrow\cat{M}
  \]
  denote the enriched functor defined objectwise by
  \[
    X^{F}(c)\coloneqq X^{F(c)},
  \]
  where the power on the right is the cotensoring of \(\cat{M}\)
  over \(\cat{V}\). Its enriched functor structure is induced by the
  enriched presheaf structure of \(F\). The coend-end and tensor-cotensor
  adjunctions give a natural isomorphism
  \[
    \begin{aligned}
      \Map_{\cat{M}}^{\cat{V}}
      \bigl(F\otimes_{\cat{C}}G,X\bigr)
      &\cong
      \int_{c\in\cat{C}}
      \Map_{\cat{M}}^{\cat{V}}
      \bigl(F(c)\otimes G(c),X\bigr) \\
      &\cong
      \int_{c\in\cat{C}}
      \Map_{\cat{M}}^{\cat{V}}
      \bigl(G(c),X^{F(c)}\bigr) \\
      &\cong
      \Map_{\Fun_{\cat{V}}(\cat{C},\cat{M})}^{\cat{V}}
      \bigl(G,X^{F}\bigr).
    \end{aligned}
  \]
  Thus \(X^{F}\) is the right adjoint to the weighted colimit in its
  second variable.

  By the adjoint form of the pushout-product criterion
  \textnormal{\cite[Lemma~4.2.2]{Hov99}}, it suffices to verify the
  corresponding pullback-corner condition for this right adjoint. Let
  \[
    i:F\longrightarrow F'
  \]
  be an injective cofibration in \(\psh(\cat{C},\cat{V})\), and let
  \[
    p:X\longrightarrow Y
  \]
  be a fibration in \(\cat{M}\). The induced pullback-corner map is
  \[
    \theta_{i,p}:
    X^{F'}
    \longrightarrow
    Y^{F'}\times_{Y^{F}}X^{F}.
  \]
  Since limits in the enriched functor category are computed objectwise,
  evaluation at \(c\in\cat{C}\) yields
  \[
    \theta_{i,p}(c):
    X^{F'(c)}
    \longrightarrow
    Y^{F'(c)}
    \times_{Y^{F(c)}}
    X^{F(c)}.
  \]
  The map \(i(c):F(c)\to F'(c)\) is a cofibration in \(\cat{V}\).
  Moreover, since \(\cat{M}\) is a \(\cat{V}\)-model category, the ordinary
  tensor bifunctor
  \[
    -\otimes-:\cat{V}\times\cat{M}\longrightarrow\cat{M}
  \]
  is left Quillen in two variables. Its adjoint pullback-corner condition
  implies that \(\theta_{i,p}(c)\) is a fibration in~\(\cat{M}\), trivial
  whenever either \(i(c)\) or \(p\) is a weak equivalence.

  Projective fibrations and trivial fibrations in
  \(\Fun_{\cat{V}}(\cat{C},\cat{M})\) are detected objectwise. Hence
  \(\theta_{i,p}\) is a projective fibration, and it is a trivial
  projective fibration whenever \(i\) or \(p\) is a weak
  equivalence. Therefore, by the adjoint
  pushout-product criterion, the weighted colimit functor
  \[
    -\otimes_{\cat{C}}-:
    \psh(\cat{C},\cat{V})_{\inj}
    \times
    \Fun_{\cat{V}}(\cat{C},\cat{M})_{\proj}
    \longrightarrow
    \cat{M}
  \]
  is a \(\cat{V}\)-enriched left Quillen bifunctor.

  The final assertions follow by fixing a cofibrant variable in this
  Quillen bifunctor.
\end{proof}

\Cref{prop:weighted.colimit.quillen.bifunctor} allows us to define a
homotopy-coherent version of a weighted colimit.

\begin{definition}\label{def:homotopy.weighted.colimit}
  \textnormal{\cite[Chapter~XII, Section 2]{BK72}},
  \textnormal{\cite[Definition 13.2]{Shu06}},
  \textnormal{\cite[Definition~19.1.2]{Hir03}}.
  Let \(\cat{V}\), \(\cat{C}\), and \(\cat{M}\) satisfy the hypotheses of
  \Cref{prop:weighted.colimit.quillen.bifunctor}. For
  \(F\in \psh(\cat{C},\cat{V})\) and a \(\cat{V}\)-enriched functor
  \(G:\cat{C}\to \cat{M}\), the \emph{homotopy weighted colimit} of
  \(G\) with weight \(F\) is the left derived weighted colimit
  (\Cref{def:weighted.colimit})
  \[
    F\otimes^{\mathbb{L}}_{\cat{C}}G
    \coloneqq
    Q_{\inj}F\otimes_{\cat{C}} Q_{\proj}G
    =
    \int^{c\in \cat{C}}
    Q_{\inj}F(c)\otimes Q_{\proj}G(c),
  \]
  where
  \[
    Q_{\inj}F\to F
  \]
  is an injectively cofibrant replacement in
  \( \psh(\cat{C},\cat{V}) \), and
  \[
    Q_{\proj}G\to G
  \]
  is a projectively cofibrant replacement in
  \( \Fun_{\cat{V}}(\cat{C},\cat{M}) \).

  When \(F\) is already injectively cofibrant, we write
  \[
    F\otimes^{\mathbb{L}}_{\cat{C}}G
    \simeq
    F\otimes_{\cat{C}} Q_{\proj}G.
  \]
  In particular, in the cases \(\cat{V}=\sset\) and \(\cat{V}=\CsSet\)
  used below, injective cofibrations in presheaf categories are objectwise
  monomorphisms, so every weight \(F\in\psh(\cat{C},\cat{V})\) is
  injectively cofibrant.
\end{definition}

\begin{definition}\label{def:homotopy.limits.colimits}
  \textnormal{\cite[Definition~8.2]{Shu06}}.
  Let \(\cat{M}\) be a combinatorial model category and
  \(\cat{I}\) a small ordinary category. For diagrams
  \(X,Y:\cat{I}\to\cat{M}\), define
  \[
    \hocolim_{\cat{I}}X
    \coloneqq
    \colim_{\cat{I}}Q_{\proj}X,
    \qquad
    \holim_{\cat{I}}Y
    \coloneqq
    \lim_{\cat{I}}R_{\inj}Y,
  \]
  using the projective and injective diagram model structures,
  respectively. These constructions represent
  \(\mathbb{L}\colim_{\cat{I}}\) and \(\mathbb{R}\lim_{\cat{I}}\).
  If \(\cat{M}\) is a \(\cat{V}\)-model category, let
  \(\cat{I}_{\cat{V}}\) denote the free \(\cat{V}\)-enrichment of
  \(\cat{I}\), with
  \(\cat{I}_{\cat{V}}(i,j)=
  \coprod_{f\in\cat{I}(i,j)}\unit_{\cat{V}}\). Then the ordinary colimit is
  the weighted colimit over \(\cat{I}_{\cat{V}}\):
  \[
    \unit_{\cat{V}}\otimes_{\cat{I}_{\cat{V}}}X
    \cong
    \colim_{\cat{I}}X.
  \]
  In the cases used below every object of \(\cat{V}\) is cofibrant, and
  hence
  \[
    \unit_{\cat{V}}
    \otimes^{\mathbb{L}}_{\cat{I}_{\cat{V}}}X
    \simeq
    \hocolim_{\cat{I}}X.
  \]
  Thus ordinary homotopy colimits agree with the homotopy weighted
  colimits of \Cref{def:homotopy.weighted.colimit} in this case.
\end{definition}

\subsubsection{Grothendieck construction and Thomason's theorem}
The Grothendieck construction and Thomason's theorem allow us to replace certain
homotopy colimits by nerves of explicit ordinary categories.

\begin{definition}\label{def:grothendieck.construction}
  \textnormal{\cite[Definition~5.1.1]{Ric20}}.
  For a functor \(F:\cat{C}\to\cat{Cat}\), its \emph{Grothendieck
  construction} \(\int_{\cat{C}}F\) has objects \((c,x)\), where
  \(x\in F(c)\), and morphisms \((c_1,x_1)\to(c_2,x_2)\) given by pairs
  \((f,g)\), where \(f:c_1\to c_2\) and
  \(g:F(f)(x_1)\to x_2\). Composition is
  \[
    (f_2,g_2)\circ(f_1,g_1)
    =
    (f_2\circ f_1,g_2\circ F(f_2)(g_1)).
  \]
  If \(F\) is set-valued, its values are regarded as discrete categories,
  so a morphism is simply a map \(f:c_1\to c_2\) satisfying
  \(F(f)(x_1)=x_2\). We write \(\int F\) when \(\cat{C}\) is clear.
\end{definition}

\begin{theorem}\label{thm:thomason}
  \textnormal{\cite[Theorem~1.2]{Tho79}}.
  Let \(F:\cat{K}\to \cat{Cat}\) be a functor. Write
  \[
    \nerve(F)\coloneqq\nerve\circ F:\cat{K}\longrightarrow\sset.
  \]
  There is a natural weak
  equivalence
  \[
    \eta:\hocolim_{\cat{K}}\nerve(F)\to \nerve\left(\int_{\cat{K}}F\right)
  \]
  between the homotopy colimit of \(\nerve(F)\) and the nerve of the
  Grothendieck construction \(\int_{\cat{K}}F\).
\end{theorem}

\subsection{Descent}\label{subsec:descent}
We recall the \v{C}ech-local model structures and hypercover constructions used
to formulate descent for ordinary and enriched presheaves. From this subsection
onward, the enriching model category \(\cat{V}\) is either \(\sset\) or
\(\CsSet\).

\subsubsection{Sites and \v{C}ech descent}
\v{C}ech descent expresses the homotopy-coherent gluing of local data. We impose
it by the left Bousfield localization.

\begin{definition}\label{def:coverage.sieve}\label{def:coverage.site}
  \textnormal{\cite[Definition~C2.1.1]{Joh02}}.
  Let \(\cat{C}\) be a category. A \emph{coverage} on \(\cat{C}\)
  assigns to every object \(U\in\cat{C}\) a collection
  \(\mathrm{Cov}_{\cat{C}}(U)\) of families of morphisms \(
  \mathcal{U}=\{f_i:U_i\to U\}_{i\in I}\), called \emph{covering families},
  satisfying the following stability
  condition. If
  \[
    \mathcal{U}
    =
    \{f_i:U_i\to U\}_{i\in I}
    \in
    \mathrm{Cov}_{\cat{C}}(U)
  \]
  is a covering family and \(g:V\to U\) is any morphism, then there exists
  a covering family
  \[
    \mathcal{V}
    =
    \{h_j:V_j\to V\}_{j\in J}
    \in
    \mathrm{Cov}_{\cat{C}}(V)
  \]
  such that, for every \(j\in J\), the composite
  \(g h_j:V_j\to U\) factors through one of the morphisms
  \(f_i:U_i\to U\). Equivalently, for every \(j\in J\), there exist
  \(i\in I\) and a morphism \(k:V_j\to U_i\) making the diagram
  \[
    \begin{tikzcd}
      {V_j} & {U_i} \\
      V & U
      \arrow["k", from=1-1, to=1-2]
      \arrow["{h_j}"', from=1-1, to=2-1]
      \arrow["{f_i}", from=1-2, to=2-2]
      \arrow["g"', from=2-1, to=2-2]
    \end{tikzcd}
  \]
  commute.

  A category equipped with a coverage is called a \emph{site}. An \emph{enriched
  site} is a \(\cat{V}\)-enriched category
  \(\cat{C}\) whose underlying ordinary category \(\cat{C}_{0}\) is
  equipped with a coverage. Covering families in an enriched site are
  therefore ordinary families of morphisms in \(\cat{C}_{0}\).
\end{definition}

\begin{definition}\label{def:covering.sieve.morphism}
  Let \(\cat{V}\) be either \(\sset\) or \(\CsSet\), and let
  \(\cat{C}\) be either an ordinary site or a small \(\cat{V}\)-enriched
  category whose underlying ordinary category \(\cat{C}_0\) is a site. In
  the ordinary case, put \(\cat{C}_0=\cat{C}\). Let
  \[
    \mathcal{U}=\{f_i:U_i\to U\}_{i\in I}
  \]
  be a covering family in \(\cat{C}_0\). Write
  \[
    \Yo^0:\cat{C}_0\longrightarrow\psh(\cat{C}_0),
    \qquad
    c\longmapsto
    \Yo_c^0\coloneqq\cat{C}_0(-,c),
  \]
  for the ordinary Yoneda embedding.

  We write
  \[
    c_{\mathcal{U}}^{0}\subseteq\Yo_U^{0}
  \]
  for the ordinary sieve generated by \(\mathcal{U}\). Thus
  \(c_{\mathcal{U}}^{0}(V)\) consists of
  those morphisms \(V\to U\) that factor through some \(f_i\).

  In the ordinary case, let \(\Yo_c\) denote the \(\cat{V}\)-valued
  representable obtained from \(\Yo_c^0\) by copowering with the unit of
  \(\cat{V}\); in the enriched case, let \(\Yo_c\) denote the enriched
  representable. Define
  \[
    c_{\mathcal{U}}
    \coloneqq
    \int^{c\in\cat{C}_0}
    c_{\mathcal{U}}^{0}(c)\cdot\Yo_c,
    \qquad
    S\cdot\Yo_c\coloneqq\coprod_{s\in S}\Yo_c.
  \]
  The ordinary covering-sieve inclusion induces a morphism
  \[
    s_{\mathcal{U}}:
    c_{\mathcal{U}}\longrightarrow\Yo_U
  \]
  in \(\psh(\cat{C},\cat{V})\). We call it the \emph{covering-sieve
  morphism} associated to \(\mathcal{U}\), and put
  \[
    S_{\cat{C}}^{\check{C}}
    \coloneqq
    \{s_{\mathcal{U}}\mid
    U\in\cat{C}_0,\ \mathcal{U}\in\mathrm{Cov}_{\cat{C}_0}(U)\}.
  \]
\end{definition}

\begin{definition}\label{def:cech.local.injective.model.structure}
  Let \(\cat{C}\) and \(\cat{V}\) be as in
  \Cref{def:covering.sieve.morphism}. Suppose that the injective model
  structure (\Cref{def:injective.model.structure}) on \(\psh(\cat{C},\cat{V})\)
  and the enriched left Bousfield
  localization (\Cref{def:enriched.left.bousfield.localization})
  below exist. The \emph{\v{C}ech-local injective model
  structure} on \(\psh(\cat{C},\cat{V})\) is the \(\cat{V}\)-enriched left
  Bousfield localization
  \[
    L_{S_{\cat{C}}^{\check{C}}}
    \psh(\cat{C},\cat{V})_{\inj}.
  \]
\end{definition}

\begin{remark}
  The cofibrations in the \v{C}ech-local injective model structure
  are the injective cofibrations. Its weak equivalences are the enriched
  \(S_{\cat{C}}^{\check{C}}\)-local equivalences
  (\Cref{def:enriched.s.local.objects.equivalences}), which we call
  \emph{\v{C}ech-local weak equivalences}, and its fibrant objects are
  precisely the injectively fibrant presheaves that are enriched
  \(S_{\cat{C}}^{\check{C}}\)-local.
\end{remark}

\begin{definition}\label{def:cech.descent.condition}
  \textnormal{\cite[Definition~4.54]{Bar10}}.
  Let \(\cat{C}\) and \(\cat{V}\) be in either of the settings of
  \Cref{def:cech.local.injective.model.structure}, and let
  \(X\in\psh(\cat{C},\cat{V})\).

  We say that \(X\) satisfies \emph{\v{C}ech descent} if, for every covering
  family
  \[
    \mathcal{U}=\{U_i\to U\}_{i\in I}
  \]
  in the relevant ordinary site, the canonical map
  \[
    \mathbb{R}\Map_{\psh(\cat{C},\cat{V})}^{\cat{V}}(\Yo_{U},X)
    \longrightarrow
    \mathbb{R}\Map_{\psh(\cat{C},\cat{V})}^{\cat{V}}(c_{\mathcal{U}},X)
  \]
  is a weak equivalence in \(\cat{V}\), where the derived mapping objects
  (\Cref{def:derived.mapping.space}) are
  computed in \(\psh(\cat{C},\cat{V})\).

  When \(\cat{C}\) is ordinary, this is equivalently the requirement
  that the canonical map
  \[
    X(U)
    \longrightarrow
    \holim_{(V\to U)\in(\cat{C}/c_{\mathcal{U}}^{0})^{\op}} X(V)
  \]
  be a weak equivalence in \(\cat{V}\), where
  \(\cat{C}/c_{\mathcal{U}}^{0}\) is the full subcategory of
  \(\cat{C}/U\) spanned by the morphisms that factor through a member of
  \(\mathcal{U}\).

  Thus the fibrant objects in the \v{C}ech-local injective model structure
  are precisely the injectively fibrant \(\cat{V}\)-valued presheaves
  satisfying \v{C}ech descent.
\end{definition}

\begin{remark}\label{rem:local.equivalences.between.local.objects}
  \textnormal{\cite[Proposition~3.2.13]{Hir03}}.
  A \v{C}ech-local weak equivalence between \v{C}ech-local objects is
  detected in the original injective model structure. Thus, if \(X\to Y\)
  is a morphism between fibrant objects in the \v{C}ech-local injective
  model structure on \(\psh(\cat{C},\cat{V})\), then \(X\to Y\) is a
  \v{C}ech-local weak equivalence if and only if
  \[
    X(c)\to Y(c)
  \]
  is a weak equivalence in \(\cat{V}\) for every \(c\in\cat{C}\).
\end{remark}

\subsubsection{Hyperdescent}
Hyperdescent extends the gluing condition introduced above from ordinary covers
to
hypercovers and provides the local weak equivalences used in our stalkwise
arguments.

\begin{definition}\label{def:matching.relative.matching}\label{def:latching.object}
  \textnormal{\cite[Definition~15.2.5 and Corollary~15.2.9]{Hir03}}.
  Let \(\cat{C}\) be a site, and let
  \[
    U_{\bullet}\longrightarrow X
  \]
  be an augmented simplicial presheaf on \(\cat{C}\), where \(X\) is a constant
  simplicial presheaf.
  Equivalently, we regard
  \(U_{\bullet}\) as a simplicial object of the slice category
  \(\psh(\cat{C})_{/X}\).

  For \(n\geq 0\), the \emph{\(n\)-th relative latching presheaf} of
  \(U_{\bullet}\to X\) is the latching object of this simplicial object in
  the slice category:
  \[
    L_n(U/X)\coloneqq
    \colim_{\sigma:[n]\twoheadrightarrow[m], m<n} U_m
    \quad\in\psh(\cat{C})_{/X},
  \]
  where the colimit is taken over the category of proper surjective maps
  out of \([n]\). The degeneracy maps of \(U_{\bullet}\) determine the relative
  latching
  morphism
  \[
    L_n(U/X)\longrightarrow U_n
  \]
  in \(\psh(\cat{C})_{/X}\).

  Dually, the \emph{\(n\)-th relative matching presheaf} of
  \(U_{\bullet}\to X\) is the matching object of \(U_{\bullet}\) in the
  slice category:
  \[
    M_n(U/X)\coloneqq
    \lim_{\delta:[m]\hookrightarrow[n], m<n} U_m
    \quad\in\psh(\cat{C})_{/X},
  \]
  where the limit is taken over the category of proper injective maps into
  \([n]\). Equivalently, on underlying presheaves,
  \[
    M_n(U/X)
    \cong
    X\times_{M_nX}M_nU,
  \]
  where
  \[
    M_nU\coloneqq
    \lim_{\delta:[m]\hookrightarrow[n], m<n} U_m
  \]
  is the ordinary matching presheaf of \(U_{\bullet}\). The augmentation and
  the face maps of \(U_{\bullet}\) determine the relative matching morphism
  \[
    U_n\longrightarrow M_n(U/X)
  \]
  in \(\psh(\cat{C})_{/X}\).
\end{definition}

Recall that a morphism \(E\to B\) of presheaves on \(\cat{C}\) is \emph{locally
surjective}, also
called a \emph{local epimorphism} or \emph{generalized cover}, if for
every \(c\in\cat{C}\) and every section \(b\in B(c)\), there is a covering
family \(\{c_i\to c\}\) such that each restriction
\(b|_{c_i}\in B(c_i)\) lifts to a section of \(E(c_i)\).

\begin{definition}\label{def:hypercover}
  \textnormal{\cite[Definition~4.2]{DHI04}}.
  Let \(\cat{C}\) be a site. A \emph{hypercover} of a presheaf \(X\) on
  \(\cat{C}\) is an augmented simplicial presheaf
  \[
    U_{\bullet}\longrightarrow X
  \]
  such that each \(U_n\) is a coproduct of representable presheaves and, for
  every \(n\geq 0\), the relative matching morphism
  (\Cref{def:matching.relative.matching})
  \[
    U_n\longrightarrow M_n(U/X)
  \]
  is locally surjective. If \(p\in\cat{C}\), a hypercover of \(p\) is a
  hypercover
  \(U_{\bullet}\to\Yo_p\).
\end{definition}

\begin{remark}\label{rem:reedy.factorization.for.hypercovers}
  The preceding latching and matching objects are the usual Reedy latching
  and matching objects for the Reedy category \(\Delta^{\op}\), computed in
  the slice category \(\psh(\cat{C})_{/X}\).

  We will use the following standard Reedy extension principle. Suppose that
  an augmented simplicial presheaf \(U_{\bullet}\to X\) has been constructed
  in degrees \(<n\). Then there is a canonical morphism
  \[
    L_n(U/X)\longrightarrow M_n(U/X)
  \]
  in \(\psh(\cat{C})_{/X}\). To extend \(U_{\bullet}\to X\) to degree \(n\)
  is equivalently to choose an object \(U_n\in\psh(\cat{C})_{/X}\) and a
  factorization
  \[
    L_n(U/X)\longrightarrow U_n\longrightarrow M_n(U/X)
  \]
  of this canonical morphism. This is a special case of
  \cite[Remark~15.2.10]{Hir03} for simplicial objects.
\end{remark}

\begin{definition}\label{def:split.hypercover}
  \textnormal{\cite[Definition~4.8]{DHI04}}.
  Let \(\cat{C}\) be a site, and let
  \[
    U_{\bullet}\longrightarrow X
  \]
  be a hypercover (\Cref{def:hypercover}). We say that \(U_{\bullet}\to X\) is
  \emph{split} if, for every
  \(n\geq0\), there is a subpresheaf \(N_n\hookrightarrow U_n\), called the
  presheaf of nondegenerate \(n\)-simplices, such that
  \[
    U_n\cong
    \coprod_{\sigma:[n]\twoheadrightarrow[r]}N_r,
  \]
  where the coproduct ranges over all surjections in \(\Delta\), and the
  degeneracy maps of~$U$ are induced by composition of surjections.

  Equivalently, the relative latching morphism
  (\Cref{def:latching.object})
  \[
    L_n(U/X)\longrightarrow U_n
  \]
  identifies \(L_n(U/X)\) with the coproduct of the summands indexed by the
  nonidentity surjections \(\sigma:[n]\twoheadrightarrow[r]\), \(r<n\), so
  that, on underlying presheaves,
  \[
    U_n\cong L_nU\amalg N_n.
  \]
  Thus every simplex of \(U_{\bullet}\) is uniquely a degeneration of a
  nondegenerate simplex, and the nondegenerate simplices in each degree form
  a coproduct of representables.

\end{definition}

\begin{lemma}\label{lem:split.hypercovers.from.basis}
  Let \(\cat{C}\) be a small site, and let
  \(\cat{B}\subset\cat{C}\) be a full subcategory such that every object of
  \(\cat{C}\) admits a covering family by objects of \(\cat{B}\). Then, for
  every \(p\in\cat{C}\), the representable presheaf \(\Yo_p\) admits a split
  hypercover
  \[
    U_{\bullet}\longrightarrow \Yo_p
  \]
  such that, for every \(n\geq 0\), the presheaf \(N_n\) of nondegenerate
  \(n\)-simplices is a coproduct of representables \(\Yo_b\), with
  \(b\in\cat{B}\). Consequently each level \(U_n\) is also a coproduct of
  representables associated to objects of \(\cat{B}\).
\end{lemma}

\begin{proof}
  We first record a simple consequence of the assumption on \(\cat{B}\).
  Every presheaf \(F\) on \(\cat{C}\) admits a locally surjective morphism
  \[
    \coprod_{\alpha\in A}\Yo_{b_\alpha}\longrightarrow F,
    \qquad b_\alpha\in\cat{B}.
  \]
  Indeed, for every object \(c\in\cat{C}\) and every section
  \(x\in F(c)\), choose a covering family
  \[
    \{b_{(c,x),i}\to c\}_{i\in I(c,x)},
    \qquad b_{(c,x),i}\in\cat{B}.
  \]
  The restriction of \(x\) to \(b_{(c,x),i}\) determines, by Yoneda, a map
  \[
    \Yo_{b_{(c,x),i}}\longrightarrow F.
  \]
  Taking the coproduct over all triples \((c,x,i)\) gives the desired
  locally surjective morphism.

  Put \(X\coloneqq\Yo_p\). We construct a split hypercover
  \(U_{\bullet}\to X\) by induction on the simplicial degree.

  In degree \(0\), the relative matching presheaf is
  \[
    M_0(U/X)\cong X.
  \]
  Apply the preceding observation to \(X\) and choose a locally surjective
  morphism
  \[
    N_0=U_0\longrightarrow X,
  \]
  with \(N_0\) a coproduct of representables \(\Yo_b\), \(b\in\cat{B}\).

  Suppose that \(U_{\bullet}\to X\) has been constructed as a split
  augmented simplicial presheaf in degrees \(<n\), and that for every
  \(m<n\) the nondegenerate presheaf \(N_m\) is a coproduct of
  representables associated to objects of \(\cat{B}\). By
  \Cref{rem:reedy.factorization.for.hypercovers}, there is a canonical map
  in the slice category \(\psh(\cat{C})_{/X}\)
  \[
    L_n(U/X)\longrightarrow M_n(U/X).
  \]
  Since the partial simplicial presheaf is split, the underlying presheaf of
  \(L_n(U/X)\) is
  \[
    L_nU\cong
    \coprod_{\sigma:[n]\twoheadrightarrow[r], r<n}N_r.
  \]

  Apply the construction from the first paragraph of the proof to the presheaf
  \(M_n(U/X)\). Choose a
  locally surjective morphism
  \[
    N_n\longrightarrow M_n(U/X)
  \]
  with \(N_n\) a coproduct of representables \(\Yo_b\), \(b\in\cat{B}\).
  Define \(U_n\) in the slice category by
  \[
    U_n\coloneqq L_n(U/X)\sqcup N_n,
  \]
  and define
  \[
    U_n\longrightarrow M_n(U/X)
  \]
  to be the canonical map on the summand \(L_n(U/X)\) and the chosen
  locally surjective map on the summand \(N_n\). Thus we have a
  factorization
  \[
    L_n(U/X)\longrightarrow U_n\longrightarrow M_n(U/X).
  \]
  By the Reedy extension principle of
  \Cref{rem:reedy.factorization.for.hypercovers}, this factorization extends
  the augmented simplicial presheaf to degree \(n\). The construction makes
  the new degree split, with \(N_n\) as the presheaf of nondegenerate
  \(n\)-simplices.

  The map \(U_n\to M_n(U/X)\) is locally surjective because its restriction
  to the summand \(N_n\) is locally surjective. Hence the hypercover
  condition holds in degree \(n\).

  Induction gives a split hypercover
  \[
    U_{\bullet}\longrightarrow X=\Yo_p.
  \]
  By construction, each \(N_n\), and therefore each \(U_n\), is a coproduct
  of representables associated to objects of \(\cat{B}\).
\end{proof}

\begin{definition}\label{def:hyperdescent}
  \textnormal{\cite[Definition~4.3 and Theorem~1.1]{DHI04}}.
  Let \(\cat{C}\) and \(\cat{V}\) be as in
  \Cref{def:covering.sieve.morphism}, suppose that
  \(\psh(\cat{C},\cat{V})\) is equipped with the injective model structure,
  and let \(X\in\psh(\cat{C},\cat{V})\). We say that \(X\) satisfies
  \emph{hyperdescent} if, for every hypercover
  \[
    U_{\bullet}\longrightarrow\Yo_U
  \]
  in the relevant underlying ordinary site, interpreted in
  \(\psh(\cat{C},\cat{V})\) as in \Cref{def:hypercover}, the canonical map
  \[
    \mathbb{R}\Map_{\psh(\cat{C},\cat{V})}^{\cat{V}}(\Yo_{U},X)
    \longrightarrow
    \holim_{[n]\in\Delta}
    \mathbb{R}\Map_{\psh(\cat{C},\cat{V})}^{\cat{V}}(U_n,X)
  \]
  is a weak equivalence in \(\cat{V}\).
\end{definition}

\begin{remark}\label{rem:hyperdescent.coproduct.representables}
  Suppose that \(X\) is injectively fibrant and write
  \[
    U_n\cong
    \coprod_{\alpha\in I_n}\Yo_{u_{n,\alpha}}
  \]
  in \(\psh(\cat{C},\cat{V})\). Since all objects in the injective model
  structures used here are cofibrant
  (\Cref{rem:injective.presheaves.cofibrant}), there are natural weak
  equivalences
  \[
    \mathbb{R}\Map(\Yo_U,X)\cong X(U),
    \qquad
    \mathbb{R}\Map(U_n,X)
    \cong
    \prod_{\alpha\in I_n}\nolimits^h X(u_{n,\alpha}).
  \]
  Hence hyperdescent is equivalently the requirement that
  \[
    X(U)
    \longrightarrow
    \holim_{[n]\in\Delta}
    \prod_{\alpha\in I_n}\nolimits^h X(u_{n,\alpha})
  \]
  be a weak equivalence in \(\cat{V}\). This formulation only uses that
  the hypercover is semi-representable, namely that every \(U_n\) is a
  coproduct of representables. In \Cref{lem:split.hypercovers.from.basis} we
  construct hypercovers that are \emph{split}, which is a
  stronger condition, but splitness is not needed for the definition of
  hyperdescent.
  More general notions of hypercover need not have semi-representable
  levels; the mapping-object formulation from \Cref{def:hyperdescent} continues
  to make sense in that
  setting.
\end{remark}

\begin{remark}\label{rem:jardine.is.hyper}
  \textnormal{\cite[Corollary~2.7]{Jar87}},
  \textnormal{\cite[Theorems 1.1 and 1.2]{DHI04}}.
  For \(\cat{V}=\sset\), the hyperlocal injective model structure on
  \(\spsh(\cat{C})\) is the Jardine local model structure. Equivalently, it
  is the left Bousfield localization of the injective model structure at
  hypercovers. Its fibrant objects are precisely the injectively fibrant
  simplicial presheaves satisfying hyperdescent. Its weak equivalences are the
  local weak equivalences: the morphisms that
  induce isomorphisms on the associated sheaf of connected components and
  on all associated sheaves of homotopy groups. In particular, a local
  isomorphism between simplicially constant
  presheaves is a hyperlocal weak equivalence.

  For \(\cat{C}=\Cart\), these weak equivalences can be detected on
  ordinary stalks. More precisely, for \(X\in\spsh(\Cart)\),
  \(V\in\Cart\), and \(x\in V\), put
  \[
    X_x\coloneqq
    \colim_{\substack{x\in W\subseteq V\\ W\in\Cart}}X(W),
  \]
  where \(W\) ranges over the Cartesian open neighborhoods of \(x\).
  Then a morphism \(f:X\to Y\) is a local weak equivalence if and only
  if
  \[
    f_x:X_x\longrightarrow Y_x
  \]
  is a weak equivalence of simplicial sets for every \(V\in\Cart\) and
  every \(x\in V\); see \cite[Last paragraph of p.~48]{Jar87}.
\end{remark}

\begin{remark}\label{rem:hypercomplete.sites}
  Hyperdescent is stronger than \v{C}ech descent, i.e.,
  localization at covering sieves and localization at hypercovers need not give
  the same model structure. In the model-categorical language used here, we say
  that the
  \v{C}ech-local model structure is \emph{hypercomplete} if every \v{C}ech-local
  object satisfies hyperdescent. Equivalently, the
  \v{C}ech-local and hyperlocal model structures coincide.
\end{remark}

\section{Geometric structures}\label{sec:geometric.structures}

This section shows some of the properties of geometric structures used in
geometric functorial field theories.
We first define the sites of fiberwise embeddings and the corresponding
categories of field stacks, both without and with isotopies. We then construct
the \emph{isotopification functor}, and prove that
it is left Quillen for the \v{C}ech-local injective model structures. Finally,
we pass
to the Cartesian subsites and show that restriction induces a Quillen
equivalence. Thus, the \v{C}ech-local homotopy theory of geometric structures
may be
computed on Cartesian families.

\subsection{Field stacks}

A geometric structure on \(d\)-manifolds should vary smoothly in families,
restrict along fiberwise open embeddings, and satisfy descent. We encode these
requirements using presheaves on sites whose objects are smooth families of
\(d\)-dimensional manifolds parametrized by \(U\in\Cart\).

\begin{definition}\label{def:Man.Cart.sites}
  \textnormal{\cite[Definition~2.1.8]{GP26}}.
  We define two small sites. The objects of \(\Man\) are smooth
  manifolds, its morphisms are smooth maps, and its covering families
  are open covers. The full subcategory
  \(\Cart\subset\Man\) consists of those manifolds which are
  diffeomorphic to \(\RR^n\) for some \(n\geq 0\) and whose underlying
  set is contained in \(\RR^m\) for some \(m\geq 0\). Its morphisms are
  again smooth maps. A covering family in \(\Cart\) is an open cover
  \[
    \{U_i\longrightarrow U\}_{i\in I}
  \]
  for which every nonempty finite intersection
  \(U_{i_0}\cap\cdots\cap U_{i_k}\) is diffeomorphic to some
  \(\RR^r\).
\end{definition}

\begin{remark}\label{rem:Cart.size.convention}
  The condition on the underlying set is a size convention ensuring
  that \(\Cart\) is small. Although an object of \(\Cart\) is required
  to be set-theoretically contained in some \(\RR^m\), the resulting
  inclusion into \(\RR^m\) is not required to be smooth and is not
  included among the data of the object.
\end{remark}

For any category \( \cat{D} \), we can think of objects in the
functor category \(
\Fun(\Cart^{\op}, \cat{D}) \) as smoothly parametrized objects of \(
\cat{D} \). In particular, this approach encodes a smooth manifold~\(M\) as a
functor \( \Cart^{\op} \ra \Set \) given by the assignment \(
U \mapsto \textnormal{C}^{\infty}(U, M) \).

\begin{definition}\label{def:split.grothendieck.fibration}
  Let \( \pi:\cat{E}\to \cat{B} \) be a functor. A morphism
  \( \widetilde{f}:e'\to e \) in \( \cat{E} \) is called
  \( \pi \)-Cartesian if, for every object \( z\in\cat{E} \), every
  morphism \( h:z\to e \), and every factorization
  \[
    \pi(h)=\pi(\widetilde{f})\circ u
  \]
  in \( \cat{B} \), there exists a unique morphism
  \( \widetilde{u}:z\to e' \) such that
  \[
    \widetilde{f}\circ \widetilde{u}=h,
    \qquad
    \pi(\widetilde{u})=u.
  \]

  The functor \( \pi \) is a \emph{Grothendieck fibration} if for every
  object \( e\in\cat{E} \) and every morphism
  \( f:b'\to \pi(e) \) in \( \cat{B} \), there exists a
  \( \pi \)-Cartesian morphism
  \[
    f^{*}e\longrightarrow e
  \]
  lying over \( f \).  A \emph{cleavage} is a choice of such a Cartesian
  lift for every pair \( (f,e) \).  A cleavage is \emph{splitting} if the
  chosen lifts are compatible with identities and composition, i.e.
  \[
    \id_{\pi(e)}^{*}e=e,
    \qquad
    (f\circ g)^{*}e=g^{*}(f^{*}e)
  \]
  on the nose, with the evident identities for the chosen Cartesian
  arrows.  A Grothendieck fibration equipped with a splitting cleavage is
  called a \emph{split Grothendieck fibration}.

  If \( \cat{E} \) and \( \cat{B} \) are \( \mathrm{C}^{\infty}\Set \)-enriched
  categories, then a \( \mathrm{C}^{\infty}\Set \)-enriched Grothendieck
  fibration is a \( \mathrm{C}^{\infty}\Set \)-enriched functor
  \[
    \pi:\cat{E}\to\cat{B}
  \]
  such that, for every \( L\in\Cart \), the evaluated functor
  \[
    \pi_L:\cat{E}_L\to\cat{B}_L
  \]
  is a Grothendieck fibration, and for every morphism
  \( a:L'\to L \) in \( \Cart \), the induced structure functor
  \[
    \cat{E}_L\longrightarrow \cat{E}_{L'}
  \]
  sends \( \pi_L \)-Cartesian morphisms to \( \pi_{L'} \)-Cartesian
  morphisms.

  A \( \mathrm{C}^{\infty}\Set \)-enriched cleavage of \( \pi \) is a choice
  of cleavage for each \( \pi_L \), natural in \( L\in\Cart \).  Such a
  cleavage is \emph{splitting} if, for every \( L\in\Cart \), the cleavage of
  \( \pi_L \) is splitting.  A \( \mathrm{C}^{\infty}\Set \)-enriched
  Grothendieck fibration equipped with a splitting
  \( \mathrm{C}^{\infty}\Set \)-enriched cleavage is called a
  \emph{\( \mathrm{C}^{\infty}\Set \)-enriched split Grothendieck fibration}.
\end{definition}

\begin{definition}\label{def:fembnoniso}
  \textnormal{\cite[Definition~4.2.1]{GP26}}.
  The site of \emph{fiberwise embeddings} \(
  \fembnoniso{d} \) has as objects submersions
  \( p: T \ra U \) in \( \Man \) with \( d \)-dimensional fibers
  and \( U\in \Cart \), where we require that the underlying map of
  sets of \( p
  \) is the restriction of a projection \( \RR^{m}\times U \ra U \)
  to a subset
  \( T\subset \RR^{m}\times U \), for some \( m \in \NN\).
  Morphisms \( (p: T \ra U) \ra (q: T' \ra V) \) are pairs of morphisms
  \( ( f: T \ra T', g: U \ra V ) \) in \( \Man \) such that the diagram
  \[
    \begin{tikzcd}
      T & T' \\
      U & V
      \arrow["f", from=1-1, to=1-2]
      \arrow["p"', from=1-1, to=2-1]
      \arrow["q", from=1-2, to=2-2]
      \arrow["g"', from=2-1, to=2-2]
  \end{tikzcd}\]
  commutes and \( f \) is a fiberwise open embedding over \( g \),
  i.e.,  the map \( T \ra g^{*}T' \) is an open embedding, where
  the pullback \(
  g^{*}T' \) exists because \( q \) is a submersion. Covering
  families (\Cref{def:coverage.site}) are given
  by collections of morphisms
  \[
    \begin{tikzcd}
      {T_{\alpha}} & T \\
      {U_{\alpha}} & U
      \arrow["{i_{\alpha}}", from=1-1, to=1-2]
      \arrow[from=1-1, to=2-1]
      \arrow[from=1-2, to=2-2]
      \arrow["{j_{\alpha}}"', from=2-1, to=2-2]
  \end{tikzcd}\]
  such that both horizontal maps \( i_{\alpha} \) and \( j_{\alpha}
  \) are open
  embeddings and the collection \( \{i_{\alpha}\} \) is a covering
  family of \( T \) in the site \( \Man \). We do not require \(
  \{j_{\alpha}\} \) to be a covering family of \( U \).

  We also define the \emph{base space functor}
  \begin{align*}
    \flat: \fembnoniso{d} &\ra \Cart \\
    (p: T\ra U) & \xmapsto{\flat} U.
  \end{align*}
\end{definition}

\begin{definition}\label{def:femb}
  \textnormal{\cite[Definition~4.2.13]{GP26}}.
  The \emph{\(\CSet\)-enriched site of fiberwise
  embeddings} \(\femb{d}\) has the same
  objects as \(\fembnoniso{d}\) and given two objects \(p:T\ra U\)
  and \(q:T'\ra
  V\), the corresponding enriched mapping object \(\femb{d}(p, q)\)
  is a smooth
  set whose \(L\)-points (\(L\in \Cart\)) are given by morphisms
  \((f,g\times\id_{L})\in \fembnoniso{d}\) of
  the form
  \[
    \begin{tikzcd}
      T\times L \arrow[r,"f"] \arrow[d,"p\times
      \id_{L}"'] & T'\times
      L \arrow[d,"q\times \id_{L}"] \\
      U\times L \arrow[r,"g\times \id_{L}"'] & V\times L
    \end{tikzcd}
  \]
  for a map \(g: U \ra V\). This assignment is functorial in \(L\):
  Let \(a:L'\to L\) be a
  morphism in \(\Cart\). Define
  \[
    a^\ast:\femb{d}(p,q)_L\longrightarrow \femb{d}(p,q)_{L'}
  \]
  by sending \((f,g\times \id_L)\) to \((a^\ast f, g\times \id_{L'})\), where
  \(a^\ast f:T\times L'\to T'\times L'\) is a map making the following diagram
  commute:
  \[
    \begin{tikzcd}
      T\times L' \arrow[r,"a^\ast f"] \arrow[d,"\id_T\times a"'] &
      T'\times L' \arrow[d,"\id_{T'}\times a"] \\
      T\times L \arrow[r,"f"'] &
      T'\times L .
    \end{tikzcd}
  \]
  The identity and composition are defined in an evident way.

  The base space functor \(\flat\) is promoted to a split \(\CsSet\)-enriched
  Grothendieck fibration
  (\cite[Definition~4.2.4]{GP26})
  \[
    \flat: \femb{d}\ra \Cart,
  \]
  which maps \(p:T\ra U\) to \(U\) and an \(L\)-family \((f,
  g\times\id_{L})\) of
  morphisms to the map \(g: U \ra V\). The category \(\femb{d}\) is
  equipped with
  the same coverage as \(\fembnoniso{d}\).
\end{definition}

\begin{remark}\label{rem:base.space.functor}
  The projection condition in
  \Cref{def:fembnoniso,def:femb} is the family-level analogue of the
  size convention in \Cref{rem:Cart.size.convention}. It keeps the
  categories of families small and makes base change along morphisms
  in \(\Cart\) strictly functorial. More precisely, it supplies chosen
  cartesian lifts for the base space functor
  \[
    \flat:\femb{d}\longrightarrow\Cart,
  \]
  so that \(\flat\) is a split Grothendieck fibration. This property is used in
  the construction of the geometric bordism
  category of Grady--Pavlov~\cite{GP26}.
\end{remark}

\begin{lemma}\label{lem:femb.enriched.covering.sieve}
  Let \(d\geq 0\), let \(\mathcal{U}=\{\alpha_i:q_i\to q\}_{i\in I}\) be a
  covering family in \(\femb{d}\), and let
  \(s_{\mathcal{U}}:c_{\mathcal{U}}\to\Yo_q\) be its enriched
  covering-sieve morphism (\Cref{def:covering.sieve.morphism}). For every
  \(r\in\femb{d}\) and \(L\in\Cart\),
  the map
  \[
    c_{\mathcal{U}}(r)(L)\longrightarrow\femb{d}(r,q)(L)
  \]
  is injective. Its image consists precisely of the \(L\)-families
  \(\sigma:r\to q\) for which there are an \(i\in I\) and an
  \(L\)-family \(\widetilde\sigma:r\to q_i\) satisfying
  \(\sigma=\alpha_i\widetilde\sigma\). Consequently,
  \(c_{\mathcal{U}}\) is naturally identified with the enriched
  subpresheaf of \(\Yo_q\) consisting of the families that factor through
  a member of \(\mathcal{U}\).
\end{lemma}

\begin{proof}
  By \Cref{def:covering.sieve.morphism}, an element of
  \(c_{\mathcal{U}}(r)(L)\) is represented by a pair
  \([a:e\to q,\tau:r\to e]\), where \(a\) is an ordinary morphism in
  the sieve generated by \(\mathcal{U}\) and \(\tau\) is an \(L\)-family.
  Its image in \(\femb{d}(r,q)(L)\) is \(a\tau\). If \(a=\alpha_i h\), the
  coend relation gives
  \[
    [a,\tau]=[\alpha_i,h\tau].
  \]
  Thus, every element in the image factors through a member of
  \(\mathcal{U}\). Conversely, if an \(L\)-family
  \(\sigma:r\to q\) has the form
  \(\sigma=\alpha_i\widetilde\sigma\), then
  \([\alpha_i,\widetilde\sigma]\in c_{\mathcal{U}}(r)(L)\) maps to
  \(\sigma\).

  Suppose that \([\alpha_i,\tau_i]\) and \([\alpha_j,\tau_j]\) have the same
  image,
  and write
  \[
    \alpha_i\tau_i=\alpha_j\tau_j\eqqcolon\sigma.
  \]
  Write \(r:R\to V\), \(q_i:T_i\to U_i\),
  \(q_j:T_j\to U_j\), and \(q:T\to U\), and let \(g_i:V\to U_i\),
  \(g_j:V\to U_j\), and \(g:V\to U\) be the base maps of
  \(\tau_i\), \(\tau_j\), and \(\sigma\), respectively. These maps are
  independent of the parameter in \(L\). Pull back \(q_i\), \(q_j\),
  and \(q\) along \(g_i\), \(g_j\), and \(g\). The families
  \(\tau_i\), \(\tau_j\), and \(\sigma\) may then be regarded as
  base-preserving families, and the morphisms \(\alpha_i\) and \(\alpha_j\)
  induce ordinary fiberwise-open embeddings into \(g^*q\).

  Let \(e\) be the restriction of \(g^*q\) to the open subset
  \[
    \bigcup_{\ell\in L}\operatorname{im}(\sigma_\ell)
  \]
  of its total space. This is an object of \(\femb{d}\). By its
  definition, \(\sigma\) determines an \(L\)-family
  \(\widetilde\sigma:r\to e\). Since
  \(\sigma_\ell=\alpha_i(\tau_i)_\ell=\alpha_j(\tau_j)_\ell\) for every
  \(\ell\in L\), the displayed open subset is contained in the images
  of both pulled-back cover members. Its inclusion into \(g^*q\)
  therefore factors through each of the two induced open embeddings.
  Composing these factorizations with the pullback morphisms to
  \(q_i\) and \(q_j\) gives ordinary morphisms
  \(v_i:e\to q_i\) and \(v_j:e\to q_j\) satisfying
  \[
    \tau_i=v_i\widetilde\sigma,
    \qquad
    \tau_j=v_j\widetilde\sigma,
    \qquad
    \alpha_iv_i=\alpha_jv_j.
  \]
  The coend relations now give
  \[
    [\alpha_i,\tau_i]
    =[\alpha_iv_i,\widetilde\sigma]
    =[\alpha_jv_j,\widetilde\sigma]
    =[\alpha_j,\tau_j].
  \]
  Hence the covering-sieve morphism is injective.
\end{proof}

\begin{definition}\label{def:geometric.structures.non.isotopies}
  \textnormal{\cite[Definition~5.1.1]{GP26}.}
  Recall the site \( \fembnoniso{d} \) of \(d\)-dimensional fiberwise
  embeddings without isotopies; see \Cref{def:fembnoniso}. A
  \emph{\(d\)-dimensional field stack without isotopies} is a simplicial
  presheaf on \( \fembnoniso{d} \).

  We write
  \[
    \fieldnoniso{d}
    \coloneqq
    \spsh(\fembnoniso{d})_{\inj,\Cech}
  \]
  for the category \(\spsh(\fembnoniso{d})\) equipped with the
  \v{C}ech-local injective model structure of
  \Cref{def:cech.local.injective.model.structure}.

  We will refer to objects of \( \fieldnoniso{d} \) as
  \emph{geometric structures without isotopies}.
\end{definition}

\begin{definition}\label{def:geometric.structures}
  \textnormal{\cite[Definition~5.1.3]{GP26}.}
  Recall the \( \CsSet \)-enriched site \( \femb{d} \) of
  \(d\)-dimensional fiberwise embeddings with isotopies; see
  \Cref{def:femb}. A \emph{\(d\)-dimensional field stack with isotopies}
  is a \( \CsSet \)-enriched presheaf on \( \femb{d} \).

  We write
  \[
    \field{d}
    \coloneqq
    \psh(\femb{d},\CsSet)_{\inj,\Cech}
  \]
  for the category of \(\CsSet\)-enriched presheaves on
  \(\femb{d}\), equipped with the enriched \v{C}ech-local injective model
  structure
  of \Cref{def:cech.local.injective.model.structure}. This enriched
  localization exists by
  \Cref{thm:enriched.left.bousfield.localization.concrete}.

  We will refer to objects of \( \field{d} \) as
  \emph{geometric structures with isotopies}.
\end{definition}

The general constructions in \Cref{subsec:isotopification} apply to arbitrary
objects of \(\fieldnoniso{d}\)
and \(\field{d}\). The following example is the geometric structure used in the
Riemannian \(\sigma\)-model application.

\begin{example}\label{ex:riemannian.structure}
  For a smooth manifold \(M\), let \(\RM\in\fieldnoniso{1}\) be the
  discrete sheaf whose value on \(p:T\to U\) consists of a fiberwise
  Riemannian metric and orientation on \(T\), together with a smooth map
  \(T\to M\). Its presheaf structure is given by pullback. We call \(\RM\)
  the geometric structure of oriented Riemannian \(1\)-dimensional
  \(\sigma\)-models with target \(M\); see
  \cite[Definition~3.1.1]{KP26} for the detailed formulation.
\end{example}

\subsection{Isotopification}\label{subsec:isotopification}

In this subsection we introduce the \emph{isotopification} functor, which
passes from field stacks without isotopies to field stacks with isotopies. We
then show that this functor is left Quillen for the relevant local injective
model structures.

\begin{definition}\label{def:useful.functors}
  Consider the canonical inclusion of the site without
  isotopies into the
  \( \CsSet \)-enriched site
  \[
    \iota:\fembnoniso{d}\longrightarrow \femb{d}.
  \]

  Write
  \[
    \iota_{!}^{\infty}:\psh(\fembnoniso{d}, \CsSet)\longrightarrow
    \psh(\femb{d}, \CsSet)
  \]
  for the left Kan
  extension along
  \( \iota \). It is
  left adjoint to the restriction functor
  \[
    \iota^{*}_{\infty}:\psh(\femb{d}, \CsSet)\longrightarrow
    \psh(\fembnoniso{d}, \CsSet).
  \]
\end{definition}

\begin{proposition}\label{prop:lambda.left.quillen}
  The functor of \Cref{def:constant.smooth.direction}
  \[
    \lambda^{\infty}:
    \spsh(\fembnoniso{d})
    \longrightarrow
    \psh(\fembnoniso{d},\CsSet)
  \]
  is left Quillen for the \v{C}ech-local injective model structures.
  Moreover, its adjunction with objectwise evaluation at \(\RR^0\) is a
  Quillen equivalence.
\end{proposition}

\begin{proof}
  The functor is given by viewing a simplicial presheaf as a
  \(\CsSet\)-valued presheaf that is constant in the smooth direction. In
  particular, it is defined objectwise. Since injective cofibrations and
  injective trivial cofibrations are defined objectwise, it follows that
  \(\lambda^{\infty}\) is left Quillen for the global injective model
  structures.

  It remains to check compatibility with the left Bousfield localizations. In
  this case, the localized injective model structure is obtained by localizing
  at the \v{C}ech covering-sieve morphisms.
  Since \(\lambda^\infty\) does not change the underlying objects or
  morphisms and only regards simplicial presheaves as presheaves valued in
  \(\CsSet\), it sends each such localizing map to the corresponding
  localizing map in the target category. Hence it sends the
  chosen localization set to weak equivalences in the localized target model
  structure.

  Therefore the global Quillen adjunction descends to the localized
  injective model structures. Thus, the functor is left Quillen for
  the localized injective model structures.

  For the Quillen-equivalence assertion, first note that the global
  adjunction is an objectwise Quillen equivalence by
  \Cref{def:constant.smooth.direction}. Moreover, the target
  localization set is precisely the image under \(\lambda^\infty\) of the
  source localization set, as observed above. The invariance of
  left Bousfield localization under a Quillen equivalence
  \cite[Theorem~3.3.20]{Hir03} therefore shows that the localized
  adjunction is a Quillen equivalence as well.
\end{proof}

\begin{definition}\label{def:isotopification.functor}
  Let
  \[
    (-)_{\RR^0}:\field{d}\longrightarrow\fieldnoniso{d}
  \]
  be the functor that sends \(X\in\field{d}\) to the simplicial presheaf
  \[
    X_{\RR^0}(p)=X(p)_{\RR^0}.
  \]
  Following \cite{GP26}, we define the isotopification functor
  \[
    \mathfrak{I}_{d}:\fieldnoniso{d}\longrightarrow\field{d}
  \]
  to be the left adjoint of \((-)_{\RR^0}\). Equivalently, the functor
  \(\mathfrak{I}_{d}\) is computed by
  \[
    \mathfrak{I}_{d}F \coloneqq \iota^\infty_!\lambda^\infty F.
  \]
\end{definition}

\begin{lemma}\label{lem:LKE.raw.triples}
  Let \(F\in \fieldnoniso{d}\) be a geometric structure without isotopies.
  Let \(q\in \femb{d}\), \(L\in \Cart\), and \(k\geq 0\). Then
  \[
    \bigl(\mathfrak{I}_{d}F(q)\bigr)_{L,k}
  \]
  is naturally identified with the quotient of
  \[
    \coprod_{p\in \fembnoniso{d}}
    F(p)_k\times \femb{d}(q,\iota(p))_L
  \]
  by the equivalence relation generated by
  \[
    (p_1,F(h)(x),e)
    \sim
    (p_2,x,(\iota(h)\times \id_L)\circ e)
  \]
  for every morphism \(h:p_1\to p_2\) in \(\fembnoniso{d}\),
  \(x\in F(p_2)_k\), and
  \(e\in \femb{d}(q,\iota(p_1))_L\).
\end{lemma}

\begin{proof}
  Since the functor \( \iota_!^\infty \) is given by enriched left Kan
  extension along the inclusion
  \[
    \iota:\fembnoniso{d}\hookrightarrow \femb{d},
  \]
  the presheaf
  \[
    \mathfrak{I}_{d}F=\iota_!^\infty(\lambda^\infty F)
  \]
  admits a pointwise description by the coend formula. Fix
  \(q\in \femb{d}\), \(L\in \Cart\), and \(k\geq 0\). Since evaluation at
  \(L\) and \(k\) preserves colimits, we obtain

  \begin{equation}\label{eq:pointwise.formula.for.isotopification}
    \bigl(\mathfrak{I}_{d}F(q)\bigr)_{L,k}
    \cong
    \left(
      \int^{p\in \fembnoniso{d}}
      (\lambda^\infty F)(p)\times \femb{d}(q,\iota(p))
    \right)_{L,k}.
  \end{equation}
  By definition of \( \lambda^\infty \), the smooth simplicial presheaf
  \( (\lambda^\infty F)(p) \) is constant in the smooth direction, while
  \( \femb{d}(q,\iota(p)) \) is regarded as constant in the simplicial
  direction. Hence
  \[
    \bigl(\mathfrak{I}_{d}F(q)\bigr)_{L,k}
    \cong
    \int^{p\in \fembnoniso{d}}
    F(p)_k\times \femb{d}(q,\iota(p))_L.
  \]

  Writing this coend as a coequalizer, we obtain
  \[
    \bigl(\mathfrak{I}_{d}F(q)\bigr)_{L,k}
    \cong
    \operatorname{coeq}\left(
      \coprod_{h:p_1\to p_2}
      F(p_2)_k\times \femb{d}(q,\iota(p_1))_L
      \rightrightarrows
      \coprod_{p\in \fembnoniso{d}}
      F(p)_k\times \femb{d}(q,\iota(p))_L
    \right),
  \]
  where the two arrows are given on the summand indexed by
  \(h:p_1\to p_2\) by
  \[
    d_0(x,e)=(F(h)(x),e),
    \qquad
    d_1(x,e)=\bigl(x,(\iota(h)\times \id_L)\circ e\bigr).
  \]
  Equivalently,
  \[
    \bigl(\mathfrak{I}_{d}F(q)\bigr)_{L,k}
    \cong
    \left(
      \coprod_{p\in \fembnoniso{d}}
      F(p)_k\times \femb{d}(q,\iota(p))_L
    \right)\Big/{\sim},
  \]
  where the equivalence relation is generated by the displayed relation.
\end{proof}

Thus an \((L,k)\)-simplex of \(\mathfrak{I}_{d}F(q)\) is represented by a
triple \((p,x,e)\), where \(p\in \fembnoniso{d}\),
\(x\in F(p)_k\), and
\(e\in \femb{d}(q,\iota(p))_L\) is an \(L\)-parameterized family of
fiberwise open embeddings of \(q\) into \(p\). Geometrically, such a
representative is an \(F\)-structure in simplicial degree \(k\) on some
target family \(p\), together with an \(L\)-family of embeddings of \(q\)
into \(p\), modulo the operation of pulling the geometric structure back
along morphisms in \(\fembnoniso{d}\). In the proof of
\Cref{prop:iota.left.quillen}, the same coend will be rewritten as a colimit
over a Grothendieck construction, and the triples \((p,x,e)\) appearing here
will be identified with the objects of the corresponding category of
elements.

\begin{proposition}\label{prop:reduced.category.discrete}
  Let \(F\in \fieldnoniso{d}\), \(q\in \femb{d}\),
  \(L\in \Cart\), and \(k\geq 0\). Then the set of
  \((L,k)\)-simplices of \(\mathfrak{I}_{d}F(q)\) is naturally identified
  with the set of isomorphism classes of triples
  \[
    (p,x,e),
  \]
  where \(p\in \fembnoniso{d}\), \(x\in F(p)_k\), and
  \[
    e\in \femb{d}(q,\iota(p))_L
  \]
  The map \(e\) is a base-preserving \(L\)-parameterized family of
  fiberwise open embeddings whose joint image is \(p\).

  Two such triples \((p,x,e)\) and \((p',x',e')\) are isomorphic if and only
  if there exists a necessarily unique isomorphism \(h:p\to p'\) in
  \(\fembnoniso{d}\) such that
  \[
    F(h)(x')=x,
    \qquad
    e'=(\iota(h)\times \id_L)\circ e.
  \]
  More generally, the same description of connected components by
  isomorphism classes of reduced triples holds with \(F(p)_k\) replaced by
  \(X(p)_{L,k}\), for every
  \(X\in\psh(\fembnoniso{d},\CsSet)\). The reduction depends only on the
  embedding datum \(e\).
\end{proposition}

\begin{proof}
  Fix \(q\in \femb{d}\), \(L\in \Cart\), \(k\geq 0\), and
  \(X\in \psh(\fembnoniso{d},\CsSet)\). Let
  \(\mathcal{G}_{X,q,L,k}\) be the category whose objects are triples
  \((p,x,e)\) with
  \[
    p\in \fembnoniso{d},\qquad
    x\in X(p)_{L,k},\qquad
    e\in \femb{d}(q,\iota(p))_L,
  \]
  Its morphisms
  \[
    (p_1,x_1,e_1)\longrightarrow (p_2,x_2,e_2)
  \]
  are morphisms \(h:p_1\to p_2\) in \(\fembnoniso{d}\) satisfying
  \[
    X(h)(x_2)=x_1,
    \qquad
    e_2=(\iota(h)\times \id_L)\circ e_1.
  \]
  Let \(\mathcal{G}^{\mathrm{red}}_{X,q,L,k}\) be the full subcategory on
  the triples for which \(e\) is base-preserving and has joint image equal
  to \(p\). We first prove that its inclusion into
  \(\mathcal{G}_{X,q,L,k}\) induces a bijection on connected components and
  that \(\mathcal{G}^{\mathrm{red}}_{X,q,L,k}\) is a groupoid with at most
  one morphism between any two objects.

  First, every representative may be chosen so that the map~\(e\) is
  base-preserving. Let \((p,x,e)\) be an object of
  \(\mathcal{G}_{X,q,L,k}\), and write
  \[
    e=(f,g\times \id_L):q\times \id_L\longrightarrow p\times \id_L.
  \]
  Since morphisms in \(\femb{d}\) are fiberwise open embeddings over maps of
  base spaces, the map \(e\) canonically factors through the pullback of \(p\)
  along the base map \(g\times \id_L\). In other words, if
  \(p'\coloneqq g^{*}p\), then there is a factorization
  \[
    q\times \id_L
    \xrightarrow{\ \widetilde{e}\ }
    p'\times \id_L
    \xrightarrow{\ \iota(h)\times \id_L\ }
    p\times \id_L,
  \]
  where \(\widetilde{e}\) lies over the identity on the base and
  \(h:p'\to p\) is the canonical morphism. Passing from \((p,x,e)\) to
  \((p',X(h)(x),\widetilde{e})\) gives a morphism
  \[
    (p',X(h)(x),\widetilde{e})\longrightarrow (p,x,e)
  \]
  in \(\mathcal{G}_{X,q,L,k}\). Thus every connected component contains an
  object whose embedding family is base-preserving.

  Second, after making the representative base-preserving, one may further
  replace the target by the joint image of the family of embeddings. Suppose
  therefore that the map
  \[
    e:q\times \id_L\longrightarrow p\times \id_L
  \]
  lies over the identity on the base. For each \(\ell\in L\), let
  \(e_\ell:q\to p\) denote the corresponding fiberwise open embedding. Define
  \[
    p^\circ\coloneqq \operatorname{im}(\pi_1\circ e)\subset p,
  \]
  where \(\pi_1:p\times \id_L\to p\) is the projection. Equivalently,
  \[
    p^\circ=\bigcup_{\ell\in L}\operatorname{im}(e_\ell).
  \]
  Since each \(e_\ell\) is an open embedding, \(p^\circ\) is an open subobject
  of \(p\), hence again an object of \(\fembnoniso{d}\). The map \(e\) factors
  through \(p^\circ\times \id_L\) as
  \[
    q\times \id_L
    \xrightarrow{\ e^\circ\ }
    p^\circ\times \id_L
    \xrightarrow{\ \iota(j)\times \id_L\ }
    p\times \id_L,
  \]
  where \(j:p^\circ\hookrightarrow p\) is the open inclusion. Passing from
  \((p,x,e)\) to \((p^\circ,X(j)(x),e^\circ)\) gives a morphism
  \[
    (p^\circ,X(j)(x),e^\circ)\longrightarrow (p,x,e)
  \]
  in \(\mathcal{G}_{X,q,L,k}\). By construction, the new representative has
  the property that its target is exactly the union of the images of the
  family \(\{e_\ell\}_{\ell\in L}\).

  Applying the two reductions successively defines a functor
  \[
    \nu:\mathcal{G}_{X,q,L,k}\longrightarrow
    \mathcal{G}^{\mathrm{red}}_{X,q,L,k}.
  \]
  On objects, \(\nu\) sends a triple \((p,x,e)\) to its reduced representative
  \[
    (\overline{p},\overline{x},\overline{e}),
    \qquad
    \overline{x}\coloneqq X(r)(x),
  \]
  where \(r:\overline{p}\to p\) is the composite reduction morphism.

  We now define \(\nu\) on morphisms. Let
  \[
    a:(p_1,x_1,e_1)\longrightarrow(p_2,x_2,e_2)
  \]
  be a morphism in \(\mathcal{G}_{X,q,L,k}\). For \(i=1,2\), write
  \[
    r_i:\overline{p}_i\longrightarrow p_i
  \]
  for the composite reduction morphism, and write
  \[
    (\overline{p}_i,\overline{x}_i,\overline{e}_i)
  \]
  for the reduced representative of \((p_i,x_i,e_i)\). Thus
  \[
    \overline{x}_i=X(r_i)(x_i),
    \qquad
    e_i=(\iota(r_i)\times\id_L)\circ \overline{e}_i.
  \]

  The identity
  \[
    e_2=(\iota(a)\times\id_L)\circ e_1
  \]
  implies that \(\iota(a)\times\id_L\) sends the joint image of
  \(\overline{e}_1\) into the joint image of \(\overline{e}_2\). Since
  \(\overline{p}_i\) is defined as this joint image after making the embedding
  base-preserving, there is a unique induced morphism
  \[
    \overline{a}:\overline{p}_1\longrightarrow\overline{p}_2
  \]
  such that
  \[
    r_2\overline{a}=a r_1.
  \]

  It remains to check that \(\overline{a}\) is a morphism of triples, i.e., that
  it
  satisfies the condition on the \(X\)-coordinates. Since \(X\) is
  contravariant, we have
  \[
    X(\overline{a})(\overline{x}_2)
    =
    X(\overline{a})\bigl(X(r_2)(x_2)\bigr)
    =
    X(r_1)\bigl(X(a)(x_2)\bigr)
    =
    X(r_1)(x_1)
    =
    \overline{x}_1.
  \]
  Hence \(\overline{a}\) defines a morphism
  \[
    (\overline{p}_1,\overline{x}_1,\overline{e}_1)
    \longrightarrow
    (\overline{p}_2,\overline{x}_2,\overline{e}_2)
  \]
  in \(\mathcal{G}^{\mathrm{red}}_{X,q,L,k}\). The uniqueness of the induced
  morphism \(\overline{a}\) makes this assignment
  compatible with identities and composition, so \(\nu\) is a functor.

  If \(J:\mathcal{G}^{\mathrm{red}}_{X,q,L,k}\hookrightarrow
  \mathcal{G}_{X,q,L,k}\) denotes the inclusion, then \(\nu J=\id\). Moreover,
  the morphisms produced by the two reduction steps give a natural
  transformation
  \[
    J\nu\longrightarrow \id_{\mathcal{G}_{X,q,L,k}}.
  \]
  Therefore \(J\) induces a bijection on connected components: every object of
  \(\mathcal{G}_{X,q,L,k}\) is connected to a reduced object, and any zigzag
  between reduced objects may be sent by \(\nu\) to a zigzag entirely inside
  \(\mathcal{G}^{\mathrm{red}}_{X,q,L,k}\).

  It remains to prove the stated property of the reduced category. Let
  \[
    h:(p_1,x_1,e_1)\longrightarrow (p_2,x_2,e_2)
  \]
  be a morphism in \(\mathcal{G}^{\mathrm{red}}_{X,q,L,k}\). Since both
  embedding families are base-preserving, the morphism \(h:p_1\to p_2\) lies
  over the identity on the common base. Since
  \[
    e_2=(\iota(h)\times \id_L)\circ e_1
  \]
  and \(e_2\) is jointly surjective, every point of \(p_2\) lies in the image
  of \(h\). Thus \(h\) is surjective on total spaces. But \(h\) is already a
  fiberwise open embedding over the fixed base, so it is an isomorphism.
  Therefore every morphism in
  \(\mathcal{G}^{\mathrm{red}}_{X,q,L,k}\) is an isomorphism.

  We now show that between any two objects there is at most one morphism. Let
  \[
    h_1,h_2:(p_1,x_1,e_1)\longrightarrow (p_2,x_2,e_2)
  \]
  be two morphisms in \(\mathcal{G}^{\mathrm{red}}_{X,q,L,k}\). By definition,
  \[
    e_2=(\iota(h_1)\times \id_L)\circ e_1
    =(\iota(h_2)\times \id_L)\circ e_1.
  \]
  Let \(y_1\in p_1\). Since \(e_1\) is jointly surjective, there exist
  \(\ell\in L\) and \(y\in q\) such that \(e_{1,\ell}(y)=y_1\). Then
  \[
    h_1(y_1)
    =h_1(e_{1,\ell}(y))
    =e_{2,\ell}(y)
    =h_2(e_{1,\ell}(y))
    =h_2(y_1).
  \]
  Since this holds for every \(y_1\in p_1\), we conclude that \(h_1=h_2\).

  In particular, every automorphism is the identity. Since
  \(\mathcal{G}^{\mathrm{red}}_{X,q,L,k}\) is a groupoid with at most one
  morphism between any two objects, it is equivalent to the discrete category
  of its isomorphism classes. Explicitly, if
  \[
    \pi_0\bigl(\mathcal{G}^{\mathrm{red}}_{X,q,L,k}\bigr)
  \]
  denotes the set of isomorphism classes of objects, then the quotient functor
  \[
    Q:
    \mathcal{G}^{\mathrm{red}}_{X,q,L,k}\longrightarrow
    \pi_0\bigl(\mathcal{G}^{\mathrm{red}}_{X,q,L,k}\bigr)
  \]
  is essentially surjective by construction, and it is fully faithful because
  for any two objects there is exactly one morphism between them if they are
  isomorphic, and none otherwise.
  Now take \(X=\lambda^\infty F\). By
  \Cref{lem:LKE.raw.triples},
  \(\bigl(\mathfrak{I}_{d}F(q)\bigr)_{L,k}\) is the set of connected
  components of \(\mathcal{G}_{\lambda^\infty F,q,L,k}\), and
  \[
    (\lambda^\infty F)(p)_{L,k}=F(p)_k.
  \]
  The preceding argument therefore identifies this set with the isomorphism
  classes of the triples in the statement. The description of an isomorphism
  between two such triples follows directly from the definition of the
  morphisms in \(\mathcal{G}_{\lambda^\infty F,q,L,k}\) and their uniqueness
  in the reduced category.

  Finally, both reductions are defined using only \(e\): the first replaces
  the target by the pullback determined by the base map of \(e\), and the
  second replaces that pullback by the joint image of the resulting
  base-preserving family. The \(x\)-coordinate is only transported along the
  resulting morphism. Hence the reduction is independent of \(x\) and of
  \(k\), and the same argument applies to arbitrary \(X\).
\end{proof}

\begin{proposition}\label{prop:iota.left.quillen}
  The functor
  \[
    \iota_{!}^{\infty}:
    \psh(\fembnoniso{d},\CsSet)\longrightarrow
    \psh(\femb{d},\CsSet)
  \]
  is left Quillen for the \v{C}ech-local injective model structures.
\end{proposition}

\begin{proof}
  Equip \(\CsSet=\spsh(\Cart)\) first with the injective model
  structure before localization at \(I_{\RR}\)
  (\Cref{prop:smooth.model.as.localization}), and equip the two presheaf
  categories with the resulting injective model structures. We first prove
  that \(\iota_{!}^{\infty}\) is left Quillen for these model structures.

  Let \(X\in \psh(\fembnoniso{d},\CsSet)\), and fix
  \(q\in \femb{d}\), \(L\in \Cart\), and \(k\geq 0\). By the pointwise coend
  formula for left Kan extension,
  \[
    \bigl((\iota_{!}^{\infty}X)(q)\bigr)_{L,k}
    \cong
    \int^{p\in \fembnoniso{d}}
    X(p)_{L,k}\times \femb{d}(q,\iota(p))_L.
  \]
  Let \(\mathcal{E}_{q,L}\) be the category whose objects are pairs
  \((p,e)\), where
  \( e\in \femb{d}(q,\iota(p))_L \), and a morphism
  \[
    (p_{1},e_{1}) \longrightarrow (p_{2}, e_{2})
  \]
  is a morphism \( h: p_1\to p_2 \) in \( \fembnoniso{d} \) such that the
  diagram
  \[
    \begin{tikzcd}
      q & {\iota(p_1)} \\
      & {\iota(p_2)}
      \arrow["{e_1}", from=1-1, to=1-2]
      \arrow["{e_2}"', from=1-1, to=2-2]
      \arrow["{\iota(h)}", from=1-2, to=2-2]
    \end{tikzcd}
  \]
  commutes. Write
  \[
    \mathcal{E}^{\mathrm{red}}_{q,L}
    \subseteq \mathcal{E}_{q,L}
  \]
  for the full subcategory on the pairs \((p,e)\) for which \(e\) is
  base-preserving and has joint image equal to \(p\). Write
  \[
    G_{X,q,L,k}: \mathcal{E}_{q,L}^{\op}\longrightarrow \Set
  \]
  for the functor defined on objects by
  \[
    G_{X,q,L,k}(p,e)\coloneqq X(p)_{L,k}.
  \]
  By the standard description of a coend as a colimit over the opposite of the
  Grothendieck construction, there is a natural isomorphism
  \[
    \bigl((\iota_{!}^{\infty}X)(q)\bigr)_{L,k}
    \cong
    \colim_{\mathcal{E}_{q,L}^{\op}} G_{X,q,L,k}.
  \]

  Let
  \[
    \mathcal{G}_{X,q,L,k}\coloneqq \int G_{X,q,L,k}
  \]
  denote the Grothendieck construction (\Cref{def:grothendieck.construction})
  of \(G_{X,q,L,k}\). By the set-valued
  special case of Thomason's theorem (\Cref{thm:thomason}), the colimit above
  may be identified with
  the set of connected components of the nerve of \(\mathcal{G}_{X,q,L,k}\):
  \[
    \colim_{\mathcal{E}_{q,L}^{\op}} G_{X,q,L,k}
    \cong
    \pi_0\bigl(\nerve(\mathcal{G}_{X,q,L,k})\bigr)
    \cong
    \pi_0(\mathcal{G}_{X,q,L,k}).
  \]
  Unwinding the definitions, \(\mathcal{G}_{X,q,L,k}\) is the category whose
  objects are triples
  \[
    (p,x,e),
    \qquad
    p\in \fembnoniso{d},\ x\in X(p)_{L,k},\ e\in \femb{d}(q,\iota(p))_L,
  \]
  and whose morphisms
  \[
    (p_1,x_1,e_1)\longrightarrow (p_2,x_2,e_2)
  \]
  are morphisms \( h: p_1\to p_2 \) in \( \fembnoniso{d} \) such that
  \[
    X(h)(x_2)=x_1,
    \qquad
    e_2=(\iota(h)\times \id_L)\circ e_1.
  \]
  Thus
  \[
    \bigl((\iota_{!}^{\infty}X)(q)\bigr)_{L,k}
    \cong
    \pi_0(\mathcal{G}_{X,q,L,k}).
  \]
  For \(X=\lambda^\infty F\), the objects of
  \(\mathcal{G}_{X,q,L,k}\) are precisely the raw triples considered in
  \Cref{lem:LKE.raw.triples}.

  Let
  \[
    \mathcal{G}^{\mathrm{red}}_{X,q,L,k}
  \]
  denote the full subcategory of \(\mathcal{G}_{X,q,L,k}\) on the triples for
  which \(e\) is base-preserving and has joint image equal to \(p\). By
  \Cref{prop:reduced.category.discrete}, the inclusion
  \[
    \mathcal{G}^{\mathrm{red}}_{X,q,L,k}
    \hookrightarrow
    \mathcal{G}_{X,q,L,k}
  \]
  induces a bijection on connected components, and
  \(\mathcal{G}^{\mathrm{red}}_{X,q,L,k}\) is equivalent to the discrete
  category of its isomorphism classes.

  Choose once and for all a skeleton
  \[
    S_{q,L}\subset \mathcal{E}^{\mathrm{red}}_{q,L}.
  \]
  We claim that every connected component of \(\mathcal{G}_{X,q,L,k}\) contains
  a unique object of the form
  \[
    (p_s,x_s,e_s),
    \qquad
    (p_s,e_s)\in S_{q,L},\ x_s\in X(p_s)_{L,k}.
  \]
  Existence follows from \Cref{prop:reduced.category.discrete}: every raw
  triple is connected to a reduced triple. Replacing the reduced embedding
  datum by its unique representative in the chosen skeleton \(S_{q,L}\) gives
  a representative of the displayed form.

  For uniqueness, suppose
  \[
    (p_s,x_s,e_s),\ (p_t,x_t,e_t)
  \]
  are two such objects lying in the same connected component of
  \(\mathcal{G}_{X,q,L,k}\). Since the inclusion
  \[
    \mathcal{G}^{\mathrm{red}}_{X,q,L,k}
    \hookrightarrow
    \mathcal{G}_{X,q,L,k}
  \]
  induces a bijection on connected components, these two objects lie in the
  same connected component of \(\mathcal{G}^{\mathrm{red}}_{X,q,L,k}\). Since
  \(\mathcal{G}^{\mathrm{red}}_{X,q,L,k}\) is equivalent to a discrete
  category, there is an isomorphism between them in the reduced triple
  category.

  In particular, \((p_s,e_s)\) and \((p_t,e_t)\) are isomorphic in
  \(\mathcal{E}^{\mathrm{red}}_{q,L}\). Since \(S_{q,L}\) is a skeleton, this
  implies \(s=t\). Thus both objects lie over the same reduced embedding datum
  \((p_s,e_s)\). Any morphism between them is then given by an endomorphism
  \(h:p_s\to p_s\) satisfying
  \[
    e_s=(\iota(h)\times \id_L)\circ e_s.
  \]
  By joint surjectivity of \(e_s\), this forces \(h=\id_{p_s}\). Hence the
  condition \(X(h)(x_t)=x_s\) becomes \(x_t=x_s\). Therefore the representative
  is unique.

  Consequently there is a natural bijection
  \[
    \bigl((\iota_{!}^{\infty}X)(q)\bigr)_{L,k}
    \cong
    \coprod_{(p_s,e_s)\in S_{q,L}} X(p_s)_{L,k}.
  \]
  This construction is compatible with the simplicial operators, since the
  reduction procedure and the skeleton \(S_{q,L}\) depend only on the embedding
  datum \(e\), not on the simplex \(x\) or the degree \(k\). Hence we obtain an
  isomorphism of simplicial sets
  \[
    (\iota_{!}^{\infty}X)(q)_L
    \cong
    \coprod_{(p_s,e_s)\in S_{q,L}} X(p_s)_L.
  \]

  It follows that for fixed \(q\in \femb{d}\) and \(L\in \Cart\), the functor
  \[
    X\longmapsto (\iota_{!}^{\infty}X)(q)_L
  \]
  is a coproduct of evaluation functors
  \[
    X\longmapsto X(p_s)_L.
  \]
  Coproducts in simplicial sets preserve monomorphisms and weak equivalences.
  Hence \(\iota_{!}^{\infty}\) preserves objectwise cofibrations and objectwise
  trivial cofibrations, so it is a left Quillen functor for the global
  injective model structures with values in injective model structure on
  \(\CsSet\).

  We next pass to the injective model structures with values in the
  \(\RR\)-local injective
  model structure on \(\CsSet\). By
  \Cref{prop:smooth.model.as.localization}, these are obtained from the
  preceding injective model structures by left Bousfield localization at the
  maps
  \[
    (\Yo_L\otimes\RR)\otimes\Yo_p
    \xrightarrow{(\id_{\Yo_L}\otimes{!})\otimes\id_{\Yo_p}}
    \Yo_L\otimes\Yo_p,
    \qquad
    p\in\mathcal{C},\quad L\in\Cart,
  \]
  where \(\mathcal{C}\) is either \(\fembnoniso{d}\) or
  \(\femb{d}\).

  Enriched left Kan extension preserves tensors and sends representables to
  representables. Therefore
  \[
    \iota_{!}^{\infty}
    \left(
      (\Yo_L\otimes\RR)\otimes\Yo_p
      \longrightarrow
      \Yo_L\otimes\Yo_p
    \right)
    \cong
    \left((\Yo_L\otimes\RR)\otimes\Yo_{\iota(p)}
      \longrightarrow
    \Yo_L\otimes\Yo_{\iota(p)}\right) .
  \]
  This is one of the corresponding localizing maps in the target category.
  The localization criterion~\cite[Proposition~3.3.18(1)]{Hir03} therefore shows
  that
  \(\iota_{!}^{\infty}\) is left Quillen for the global injective model
  structures with values in the \(\RR\)-local injective model structure on
  \(\CsSet\).

  We now pass to the localized injective model structures. By the universal
  property of left Bousfield localization
  (\Cref{def:enriched.left.bousfield.localization}), it is enough to show that
  the left
  derived functor (\Cref{def:derived.functors}) of \(\iota_{!}^{\infty}\) sends
  the localizing morphisms to
  weak equivalences in the target model category.

  Since \(\iota_{!}^{\infty}\) is left Quillen for the global injective model
  structures, its left derived functor is computed by applying
  \(\iota_{!}^{\infty}\) to cofibrant objects, or equivalently by composing
  \(\iota_{!}^{\infty}\) with a cofibrant replacement functor. But in the
  injective model structure all objects are cofibrant
  (\Cref{rem:injective.presheaves.cofibrant}). Hence the left derived
  functor of \(\iota_{!}^{\infty}\) is computed by \(\iota_{!}^{\infty}\)
  itself.

  The localized injective model structures are obtained by left Bousfield
  localization at the \v{C}ech maps
  \[
    c_{\mathcal{U}}\longrightarrow \Yo_p
  \]
  associated to covering families
  \(\mathcal{U}=\{\alpha_i:p_i\to p\}_{i\in I}\). The sites
  \(\fembnoniso{d}\) and \(\femb{d}\) have the same objects and the same
  covering families. Moreover, \(\iota_{!}^{\infty}\) sends representables to
  the corresponding representables, and it sends the sieve generated by
  \(\mathcal{U}\) in \(\fembnoniso{d}\) to the sieve generated by the same
  family in \(\femb{d}\). Hence the
  ordinary covering-sieve map \(s_{\mathcal{U}}^{0}\) is the same on both sides.
  By \Cref{def:covering.sieve.morphism}, cocontinuity of
  \(\iota_{!}^{\infty}\), and the enriched Yoneda lemma,
  \[
    \begin{aligned}
      \iota_{!}^{\infty}c_{\mathcal{U}}
      &\cong
      \iota_{!}^{\infty}
      \left(
        \int^{r\in\fembnoniso{d}}
        c_{\mathcal{U}}^{0}(r)\cdot\Yo_r
      \right)\\
      &\cong
      \int^{r\in\fembnoniso{d}}
      c_{\mathcal{U}}^{0}(r)\cdot\Yo_{\iota(r)}
      \cong
      c_{\mathcal{U}},
    \end{aligned}
  \]
  and
  \[
    \iota_{!}^{\infty}\Yo_p\cong\Yo_{\iota(p)}.
  \]
  Under these identifications,
  \(\iota_{!}^{\infty}(s_{\mathcal{U}})\) is the enriched covering-sieve
  morphism associated to the same covering family \(\mathcal{U}\). It is
  therefore a \v{C}ech-local weak equivalence in the target model structure.

  Therefore the global Quillen adjunction descends to the localized injective
  model structures by~\cite[Proposition~3.3.18(1)]{Hir03}. Hence
  \(\iota_{!}^{\infty}\) is left Quillen for the localized
  injective model structures.
\end{proof}

\begin{corollary}\label{cor:isotop.is.left.quillen}
  The isotopification functor
  \[
    \mathfrak{I}_d=\iota_{!}^{\infty}\lambda^\infty:
    \spsh(\fembnoniso{d})\longrightarrow
    \psh(\femb{d},\CsSet)
  \]
  is left Quillen for the \v{C}ech-local model structure.
\end{corollary}

\begin{proof}
  By \Cref{prop:iota.left.quillen}, \(\iota_{!}^{\infty}\) is left Quillen and
  by \Cref{prop:lambda.left.quillen}, \(\lambda^{\infty}\) is left
  Quillen. Therefore their composite is left Quillen.
\end{proof}

\subsection{The Cartesian subsite}

Having incorporated isotopies, we now restrict to Cartesian \(d\)-dimensional
families and prove that this smaller subsite retains the local homotopical
information carried by the full site.

\begin{definition}\label{def:fembcart}
  \textnormal{\cite[Definition~4.10]{GP24}}.
  The \emph{(Cartesian) site} \( \fembcart{d} \) is the full subcategory of
  \( \femb{d} \) (\Cref{def:femb}) on objects isomorphic to a projection
  \( \RR^{d}\times U \ra U \), with the coverage induced from \( \femb{d}
  \). We define the Cartesian site \( \fembcartnoniso{d} \) as the full
  subcategory of \( \fembnoniso{d} \) (\Cref{def:fembnoniso}) in the same way.
\end{definition}

\begin{lemma}\label{lem:femb.cover.by.fembcart}
  Every object of \( \femb{d} \) admits a
  covering family by objects of
  \( \fembcart{d} \). More precisely, for every \( p:T\to U \in
  \femb{d} \), there exists a covering
  family
  \[
    \mathcal{U}=\{(u_i:T_i\to U_i) \to (p:T\to U)\}_{i\in I}
  \]
  in \( \femb{d} \) such that each \( u_i \) lies in \( \fembcart{d}
  \). The same statement holds for \( \fembnoniso{d} \) and \(
  \fembcartnoniso{d} \).
\end{lemma}

\begin{proof}
  Let \( p:T\to U \) be an object of \( \femb{d} \). Since \( p \) is a
  submersion with \( d \)-dimensional fibers and \( U \in \Cart \),
  every point
  \( x\in T \) admits an open neighborhood \( V_x\subseteq T \) and an open
  neighborhood \( U_x\subseteq U \) of \( p(x) \) such that
  \[
    u_x\coloneqq p\restriction_{V_x}:V_x\longrightarrow U_x
  \]
  is isomorphic, over \( U_x \), to the projection
  \[
    \mathbb{R}^{d}\times U_x\longrightarrow U_x.
  \]
  The maps
  \[
    u_x\longrightarrow p
  \]
  are fiberwise open embeddings. Since the open subsets \(V_x\subset T\) cover
  \(T\), these maps form a covering family of \(p\) by objects of
  \(\fembcart{d}\). The statement without isotopies is proved in the same way.
\end{proof}

\begin{lemma}\label{lem:femb.enough.points.hypercomplete}
  The \v{C}ech-local injective model
  structures on
  \[
    \psh(\femb{d},\CsSet),
    \qquad
    \psh(\fembcart{d},\CsSet),
    \qquad
    \psh(\fembnoniso{d},\sset),
    \qquad
    \psh(\fembcartnoniso{d},\sset)
  \]
  are hypercomplete (\Cref{rem:hypercomplete.sites}).
\end{lemma}

\begin{proof}
  The argument presented in \cite[Theorem~5.8~and~Theorem~5.9]{GM26} applies
  verbatim to all sites above.
\end{proof}

The following three statements generalize
\cite[Proposition~3.1.1]{GP23}.

\begin{lemma}\label{lem:femb.stalk.local.trivialization}
  Let \(d\geq0\), and recall the enriched site \(\femb{d}\) from
  \Cref{def:femb}. For \(p:T\to U\in\femb{d}\), let \(\Open(p)\) be the
  category of ordinary morphisms \(v=(v_T,v_U):r\to p\) whose total-space
  and base-space components are open embeddings, with morphisms the
  commutative triangles over \(p\). For \(x\in T\), let \(\Open(p)_x\) be
  the full subcategory on the morphisms for which \(x\in v_T(T_r)\). For
  \(X\in\psh(\femb{d},\CsSet)\), define
  \[
    X_{(p,x)}
    \coloneqq
    \colim_{(v:r\to p)\in\Open(p)_x^{\op}}X(r).
  \]
  The category \(\Open(p)_x^{\op}\) is filtered. After choosing a local
  trivialization of \(p\) at \(x\), write
  \(p^{p,x}_{\delta,\varepsilon}\) for the restriction of \(p\)
  corresponding to the Cartesian submersion
  \[
    B^d_\delta\times B^n_\varepsilon\longrightarrow B^n_\varepsilon.
  \]
  Then there is a natural isomorphism
  \[
    X_{(p,x)}
    \cong
    \colim_{\delta,\varepsilon\to0}
    X\bigl(p^{p,x}_{\delta,\varepsilon}\bigr).
  \]
  The same assertion holds for \(\fembcart{d}\), \(\fembnoniso{d}\), and
  \(\fembcartnoniso{d}\).
\end{lemma}

\begin{proof}
  The category \(\Open(p)_x^{\op}\) is filtered because the fiber product
  over \(p\) of two open morphisms containing \(x\) contains a smaller open
  morphism containing \(x\) which maps to both.

  Choose open neighborhoods \(x\in V_x\subseteq T\) and
  \(p(x)\in U_x\subseteq U\), an integer \(n\geq0\), positive real numbers
  \(\delta_0,\varepsilon_0>0\), and an isomorphism
  \[
    \phi_x:
    \bigl(p\restriction_{V_x}:V_x\to U_x\bigr)
    \xrightarrow{\ \cong\ }
    \bigl(\pr_2:B^d_{\delta_0}\times B^n_{\varepsilon_0}
    \to B^n_{\varepsilon_0}\bigr)
  \]
  sending \(x\) to \((0,0)\). Write \(\phi_{x,T}\) and \(\phi_{x,U}\) for
  the total-space and base-space components of \(\phi_x\). Let \(I_x\) be
  the poset of pairs \((\delta,\varepsilon)\) with
  \(0<\delta<\delta_0\) and \(0<\varepsilon<\varepsilon_0\), ordered
  coordinatewise. For each \((\delta,\varepsilon)\in I_x\), let
  \(B^{p,x}_{\delta,\varepsilon}\subseteq V_x\) be the inverse image of
  \(B^d_\delta\times B^n_\varepsilon\) under \(\phi_{x,T}\), let
  \[
    U^{p,x}_\varepsilon\coloneqq\phi_{x,U}^{-1}(B^n_\varepsilon),
  \]
  and let
  \[
    p^{p,x}_{\delta,\varepsilon}:
    B^{p,x}_{\delta,\varepsilon}\longrightarrow U^{p,x}_\varepsilon
  \]
  be the restricted submersion. The assignment
  \[
    (\delta,\varepsilon)\longmapsto
    (p^{p,x}_{\delta,\varepsilon}\to p)
  \]
  defines an initial functor \(I_x\to\Open(p)_x\). Indeed, for every
  \(v:r\to p\) in \(\Open(p)_x\), sufficiently small \(\delta\) and
  \(\varepsilon\) give a morphism
  \[
    p^{p,x}_{\delta,\varepsilon}\longrightarrow r
  \]
  over \(p\). Any two such choices admit a common choice with smaller fiber
  and base radii. Hence
  \[
    I_x^{\op}\longrightarrow\Open(p)_x^{\op}
  \]
  is final, and therefore
  \begin{equation}\label{eq:femb.fiberwise.stalk.local.trivialization}
    X_{(p,x)}
    \cong
    \colim_{(\delta,\varepsilon)\in I_x^{\op}}
    X(p^{p,x}_{\delta,\varepsilon}).
  \end{equation}
  The other cases are identical.
\end{proof}

\begin{proposition}\label{prop:femb.stalk.left.quillen}
  Let \(d\geq0\). Equip \(\psh(\femb{d},\CsSet)\) with the
  \v{C}ech-local injective model structure of
  \Cref{def:cech.local.injective.model.structure}, formed using the enriched
  covering-sieve morphisms of \Cref{def:covering.sieve.morphism}. For
  \(p:T\to U\in\femb{d}\) and \(x\in T\), let
  \[
    (-)_{(p,x)}:
    \psh(\femb{d},\CsSet)_{\inj,\Cech}\longrightarrow\CsSet
  \]
  be the fiberwise stalk functor defined by
  \[
    X_{(p,x)}
    =
    \colim_{(v:r\to p)\in\Open(p)_x^{\op}}X(r),
  \]
  where \(\Open(p)_x\) is defined in
  \Cref{lem:femb.stalk.local.trivialization}. Then \((- )_{(p,x)}\) is left
  Quillen. The analogous statement holds for \(\fembcart{d}\) and, with
  \(\CsSet\) replaced by \(\sset\), for \(\fembnoniso{d}\) and
  \(\fembcartnoniso{d}\), using their corresponding \v{C}ech-local injective
  model structures.
\end{proposition}

\begin{proof}
  We prove the statement for \(\femb{d}\). The other cases are analogous,
  replacing \(\CsSet\) by \(\sset\) in the unenriched cases.

  Fix \(p:T\to U\in\femb{d}\). Equip \(\Open(p)\) with the coverage induced
  from \(\fembnoniso{d}\). The identity of \(p\) is a terminal object of
  \(\Open(p)\). Let
  \[
    \dom_p:\Open(p)\longrightarrow\femb{d},
    \qquad (v:r\to p)\longmapsto r.
  \]

  For \(x\in T\), define
  \[
    x^*(Z)\coloneqq
    \colim_{(v:r\to p)\in\Open(p)_x^{\op}}Z(v:r\to p)
  \]
  for \(Z\in\psh(\Open(p),\CsSet)\). By definition,
  \begin{equation}\label{eq:femb.fiberwise.stalk}
    X_{(p,x)}
    =x^*(\dom_p^*X)
    =\colim_{(v:r\to p)\in\Open(p)_x^{\op}}X(r).
  \end{equation}

  By \Cref{lem:femb.stalk.local.trivialization}, the category
  \(\Open(p)_x^{\op}\) is filtered.

  Before localization, restriction along \(\dom_p\) preserves
  objectwise cofibrations and objectwise weak equivalences. Filtered colimits in
  \(\CsSet\) preserve monomorphisms and weak equivalences, since
  \(\Sing^\infty\) detects weak equivalences
  (\Cref{prop:Cartesian.model.str.smooth.simp.sets}) and commutes with filtered
  colimits. Hence the composite
  \[
    (-)_{(p,x)}=x^*\dom_p^*:
    \psh(\femb{d},\CsSet)_{\inj,\Cech}
    \longrightarrow
    \CsSet
  \]
  is left Quillen for the global injective model structures.

  Let \(\mathcal{U}=\{\alpha_i:q_i\to q\}_{i\in I}\) be a covering family in
  \(\femb{d}\). Applying \((- )_{(p,x)}\) to the associated enriched
  covering-sieve morphism of \Cref{def:covering.sieve.morphism} gives
  \begin{equation}\label{eq:dom.restricted.covering.sieve.source.stalk}
    \colim_{(v:r\to p)\in\Open(p)_x^{\op}}c_{\mathcal{U}}(r)
    \longrightarrow
    \colim_{(v:r\to p)\in\Open(p)_x^{\op}}\femb{d}(r,q)
  \end{equation}
  in \(\CsSet\). Fix \(L\in\Cart\) and \(\ell\in L\). Taking the stalk at
  \(\ell\) gives
  \begin{equation}\label{eq:dom.restricted.covering.sieve.double.stalk}
    \begin{aligned}
      &\colim_{\ell\in K\subseteq L}
      \colim_{(v:r\to p)\in\Open(p)_x^{\op}}
      c_{\mathcal{U}}(r)(K)\\
      &\hspace{25mm}\longrightarrow
      \colim_{\ell\in K\subseteq L}
      \colim_{(v:r\to p)\in\Open(p)_x^{\op}}
      \femb{d}(r,q)(K),
    \end{aligned}
  \end{equation}
  where \(K\) ranges over Cartesian open neighborhoods of \(\ell\) in \(L\).
  We claim that \Cref{eq:dom.restricted.covering.sieve.double.stalk} is a
  bijection.

  For surjectivity, an element on the right is represented by a pair
  consisting of an open morphism \(v:r\to p\) and a \(K\)-family
  \(\sigma:r\to q\), where \(K\subseteq L\) is a Cartesian neighborhood of
  \(\ell\). Since \(v_T\) is an open embedding and \(x\) belongs to its
  image, there is a unique \(y\in T_r\) such that \(v_T(y)=x\). We then
  evaluate the family \(\sigma\) at \((y,\ell)\).

  Write \(r:S\to V\), \(q_i:T_i\to U_i\), and \(q:T\to U\). For every
  \(i\in I\), write
  \[
    \alpha_i=(\alpha_{i,T},\alpha_{i,U}),
    \qquad
    \alpha_{i,T}:T_i\longrightarrow T,
    \qquad
    \alpha_{i,U}:U_i\longrightarrow U.
  \]
  Write the total-space map of the family and its base-space component as
  \[
    \widetilde{\sigma}_T:S\times K\longrightarrow T\times K,
    \qquad
    \sigma_U:V\longrightarrow U.
  \]
  By the definition of a \(K\)-family, \(\widetilde{\sigma}_T\) is a map
  over \(K\). Hence there is a map \(\sigma_T:S\times K\to T\) such that
  \[
    \widetilde{\sigma}_T(s,k)=(\sigma_T(s,k),k).
  \]
  The map \(\sigma_U\) is independent of the parameter in \(K\).

  The open subsets \(\alpha_{i,T}(T_i)\) cover \(T\), so choose \(i\in I\) such
  that \(\sigma_T(y,\ell)\) belongs to \(\alpha_{i,T}(T_i)\). It follows from
  the commutative square defining
  \(\sigma\) that \(\sigma_U(r(y))\) belongs to \(\alpha_{i,U}(U_i)\). Choose
  local coordinates around \(y\) in which the submersion \(r:S\to V\) is a
  projection. Since the images of \(\alpha_{i,T}\) and \(\alpha_{i,U}\) are open
  and
  \(\sigma_T\) and \(\sigma_U\) are continuous, there are a smaller open
  morphism \(r'\to r\) whose total-space image contains \(y\) and a smaller
  Cartesian neighborhood \(K'\subseteq K\) containing \(\ell\) such that
  \[
    \sigma_T(T_{r'}\times K')
    \subseteq \alpha_{i,T}(T_i),
    \qquad
    \sigma_U(U_{r'})\subseteq \alpha_{i,U}(U_i).
  \]
  Therefore the restricted family has a unique factorization
  \[
    \sigma|_{r'\times K'}
    =\alpha_i\widetilde\sigma
  \]
  through a \(K'\)-family \(\widetilde\sigma:r'\to q_i\). By
  \Cref{lem:femb.enriched.covering.sieve}, the restricted family defines an
  element of \(c_{\mathcal{U}}(r')(K')\), which represents a preimage of the
  original element in
  \Cref{eq:dom.restricted.covering.sieve.double.stalk}.

  For injectivity, \Cref{lem:femb.enriched.covering.sieve} shows that
  \[
    c_{\mathcal{U}}(r)(K)\longrightarrow\femb{d}(r,q)(K)
  \]
  is injective for every \(r\) and \(K\). The two indexing categories in
  \Cref{eq:dom.restricted.covering.sieve.double.stalk} are filtered, hence so
  is their product. Since filtered colimits of sets preserve monomorphisms,
  the map in \Cref{eq:dom.restricted.covering.sieve.double.stalk} is
  injective.

  Consequently,
  \Cref{eq:dom.restricted.covering.sieve.source.stalk} is a stalkwise
  isomorphism in the \(\Cart\)-variable. By
  \cite[Proposition~12.5]{PavlovSmoothOka}, it is a weak equivalence in
  \(\CsSet\). Thus, \((- )_{(p,x)}\) is left Quillen for the \v{C}ech-local
  injective model
  structure. Since all objects are cofibrant, its ordinary value computes its
  left derived value. In particular, every \v{C}ech-local weak equivalence
  induces a weak equivalence on every fiberwise stalk.
\end{proof}

\begin{proposition}\label{prop:femb.stalks.detect.local.we}
  Let \(d\geq0\), and equip \(\psh(\femb{d},\CsSet)\) with the
  \v{C}ech-local injective model structure of
  \Cref{def:cech.local.injective.model.structure}. For \(p:T\to U\in\femb{d}\)
  and \(x\in T\), let \((- )_{(p,x)}\) be the fiberwise stalk functor of
  \Cref{prop:femb.stalk.left.quillen}. Then the family of fiberwise stalk
  functors
  \[
    \bigl\{(-)_{(p,x)}\bigr\}_{p:T\to U\in\femb{d},\,x\in T}
  \]
  is jointly conservative on
  \(\psh(\femb{d},\CsSet)_{\inj,\Cech}\). Equivalently, a morphism
  \(f:X\to Y\) is a \v{C}ech-local weak equivalence if and only if
  \[
    f_{(p,x)}:X_{(p,x)}\longrightarrow Y_{(p,x)}
  \]
  is a weak equivalence in \(\CsSet\) for every \(p:T\to U\) and every
  \(x\in T\). The analogous statement holds for \(\fembcart{d}\) and, with
  \(\CsSet\) replaced by \(\sset\), for \(\fembnoniso{d}\) and
  \(\fembcartnoniso{d}\).
\end{proposition}

\begin{proof}
  By \Cref{prop:femb.stalk.left.quillen}, every \v{C}ech-local weak
  equivalence induces a weak equivalence on every fiberwise stalk.

  Conversely, suppose that \(f:X\to Y\) induces a weak equivalence on every
  fiberwise stalk. Choose functorial \v{C}ech-fibrant replacements
  \[
    X\xrightarrow{\sim}aX,
    \qquad
    Y\xrightarrow{\sim}aY,
  \]
  and let \(af:aX\to aY\) be induced. Since the fiberwise stalk functors are
  left
  Quillen and all objects are cofibrant, the \(2\)-out-of-\(3\) property shows
  that \((af)_{(p,x)}\) is a weak equivalence for every \(p\) and \(x\).

  Fix \(p:T\to U\), and define simplicial presheaves \(A\) and \(B\) on
  \(\Open(p)\) by
  \[
    A(v:r\to p)=\Sing^\infty aX(r),
    \qquad
    B(v:r\to p)=\Sing^\infty aY(r).
  \]
  They are objectwise Kan and satisfy \v{C}ech descent
  (\Cref{def:cech.descent.condition}), because a covering family
  in \(\Open(p)\) is the same covering family in \(\femb{d}\), and
  \(\Sing^\infty\) preserves the corresponding derived homotopy limits.
  Moreover, since \(\Sing^\infty\) commutes with filtered colimits,
  \[
    A_x\cong\Sing^\infty(aX)_{(p,x)},
    \qquad
    B_x\cong\Sing^\infty(aY)_{(p,x)}.
  \]
  Thus \(A\to B\) is a weak equivalence on every stalk.

  The total-space images of the objects of \(\Open(p)\) form a basis for
  the ordinary topology of \(T\): every point of every open subset of \(T\)
  has a smaller neighborhood belonging to \(\Open(p)\).
  Consequently, point stalks on \(\Open(p)\) detect local weak equivalences
  (\Cref{rem:jardine.is.hyper}).
  Indeed, filtered colimits commute with homotopy groups of Kan complexes,
  sheafification does not change stalks, and points of \(T\) detect
  isomorphisms of sheaves. Since \(T\) is finite-dimensional, the
  \v{C}ech-local and hyperlocal model structures agree by
  \cite[Theorems~5.8--5.9]{GM26}. Hence \(A\to B\) is a \v{C}ech-local weak
  equivalence. As \(A\) and \(B\) are objectwise Kan and \v{C}ech-local, this
  map is an objectwise weak equivalence. Evaluating at the terminal object
  \(\id_p\in\Open(p)\) shows that
  \[
    \Sing^\infty af(p):\Sing^\infty aX(p)\longrightarrow
    \Sing^\infty aY(p)
  \]
  is a weak equivalence. Thus \(af(p)\) is a weak equivalence in \(\CsSet\).
  Since \(p\) was arbitrary, \(af\) is an objectwise weak equivalence, and
  therefore \(f\) is
  a \v{C}ech-local weak equivalence by \(2\)-out-of-\(3\). Hence the family
  \[
    \bigl\{(-)_{(p,x)}\bigr\}_{p\in\femb{d},\,x\in T}
  \]
  is jointly conservative. The other cases are analogous.
\end{proof}

\begin{lemma}\label{lem:hypercover.replacement.Cartesian}
  Let \(\mathcal{C}\) denote either \(\femb{d}\) or
  \(\fembnoniso{d}\), and let \(\mathcal{C}_{\mathrm{Cart}}\) denote the
  corresponding Cartesian subsite, namely \(\fembcart{d}\) or
  \(\fembcartnoniso{d}\), respectively. Let \(\psh(\mathcal{C})\) denote
  \(\psh(\femb{d},\CsSet)\) or \(\spsh(\fembnoniso{d})\), respectively.

  For every \(p\in\mathcal{C}\), there is an augmented simplicial object
  \(U_{\bullet}\to\Yo_{p}\) such that each level has the form
  \[
    U_{n}
    =
    \coprod_{\alpha\in I_{n}}\Yo_{u_{n,\alpha}},
    \qquad
    u_{n,\alpha}\in\mathcal{C}_{\mathrm{Cart}}.
  \]
  When \(\mathcal{C}=\fembnoniso{d}\), this augmented simplicial object is
  a split hypercover (\Cref{def:split.hypercover}). When
  \(\mathcal{C}=\femb{d}\), it is obtained by applying
  \(\iota_{!}^{\infty}\lambda^{\infty}\) degreewise to the corresponding
  ordinary split hypercover. In either case, the augmentation induces a
  \v{C}ech-local weak equivalence
  \[
    \hocolim_{[n]\in\Delta^{\op}}U_{n}
    \xrightarrow{\sim}
    \Yo_{p}
  \]
  in \(\psh(\mathcal{C})\).
\end{lemma}

\begin{proof}
  We first construct the split hypercover. Suppose first that
  \(\mathcal{C}=\fembnoniso{d}\) and
  \(\mathcal{C}_{\mathrm{Cart}}=\fembcartnoniso{d}\). By
  \Cref{lem:femb.cover.by.fembcart}, every object of \(\mathcal{C}\) admits
  a covering family by objects of \(\mathcal{C}_{\mathrm{Cart}}\). Applying
  \Cref{lem:split.hypercovers.from.basis} to
  \[
    \mathcal{C}_{\mathrm{Cart}}\subset\mathcal{C}
  \]
  gives a split hypercover
  \[
    U_\bullet\longrightarrow\Yo_p
  \]
  such that each nondegenerate level \(N_n\) is a coproduct of
  representables associated to objects of \(\mathcal{C}_{\mathrm{Cart}}\).

  Since the hypercover is split,
  \[
    U_n\cong
    \coprod_{\sigma:[n]\twoheadrightarrow[r]}N_r .
  \]
  Writing
  \[
    N_r=\coprod_{\beta\in J_r}\Yo_{v_{r,\beta}},
    \qquad v_{r,\beta}\in\mathcal{C}_{\mathrm{Cart}},
  \]
  we obtain
  \[
    U_n\cong
    \coprod_{\substack{\sigma:[n]\twoheadrightarrow[r]\\ \beta\in J_r}}
    \Yo_{v_{r,\beta}} .
  \]
  Thus \(U_n\) has the required form
  \[
    U_n=\coprod_{\alpha\in I_n}\Yo_{u_{n,\alpha}},
    \qquad
    u_{n,\alpha}\in\mathcal{C}_{\mathrm{Cart}},
  \]
  where \(I_n\) is the indexing set of pairs
  \((\sigma,\beta)\) appearing in the preceding coproduct.

  It remains to prove that any such hypercover gives a hyperlocal weak
  equivalence. Let \(E\) be a fibrant \v{C}ech-local object. By
  \Cref{lem:femb.enough.points.hypercomplete}, the \v{C}ech-local model
  structure is hypercomplete, so \(E\) satisfies hyperdescent
  (\Cref{def:hyperdescent}). Hence the
  hypercover \(U_\bullet\to\Yo_p\) induces a weak equivalence
  \[
    \RR\Map_{\spsh(\fembnoniso{d})}(\Yo_p,E)
    \xrightarrow{\sim}
    \holim_{[n]\in\Delta}
    \RR\Map_{\spsh(\fembnoniso{d})}(U_n,E).
  \]
  By the simplicial Yoneda lemma,
  \[
    \RR\Map_{\spsh(\fembnoniso{d})}(\Yo_p,E)
    \simeq E(p).
  \]
  Therefore
  \[
    E(p)
    \xrightarrow{\sim}
    \holim_{[n]\in\Delta}
    \RR\Map_{\spsh(\fembnoniso{d})}(U_n,E).
  \]

  On the other hand, by the universal property of homotopy colimits
  (\Cref{def:homotopy.limits.colimits}),
  \[
    \RR\Map_{\spsh(\fembnoniso{d})}
    \left(
      \hocolim_{[n]\in\Delta^{\op}}U_n,
      E
    \right)
    \simeq
    \holim_{[n]\in\Delta}
    \RR\Map_{\spsh(\fembnoniso{d})}(U_n,E).
  \]
  Combining these equivalences, the augmentation
  \begin{equation}\label{eq:unenriched.aug}
    \hocolim_{[n]\in\Delta^{\op}}U_n
    \longrightarrow
    \Yo_p
  \end{equation}
  induces a weak equivalence on derived mapping objects into every fibrant
  \v{C}ech-local object \(E\). Hence it is a \v{C}ech-local weak
  equivalence.

  For \(\mathcal{C}=\femb{d}\), regard \(p\) as an object of
  \(\fembnoniso{d}\). Since the isotopification functor
  \(\mathfrak{I}_d\) preserves coproducts and sends ordinary representables to
  the corresponding enriched representables, we have an isomorphism
  \[
    \mathfrak{I}_d (U_n)
    \cong
    \coprod_{\alpha\in I_n}\Yo_{\iota(u_{n,\alpha})},
    \qquad
    \iota(u_{n,\alpha})\in\fembcart{d}.
  \]
  Hence every level of \(\mathfrak{I}_d(U_\bullet)\) has the form asserted in
  the statement. The ordinary augmentation is a \v{C}ech-local weak equivalence
  by the
  first part of the proof. Because the source and target of the ordinary
  augmentation are cofibrant, the left Quillen functor \(\mathfrak{I}_d\)
  sends the ordinary augmentation to a \v{C}ech-local weak equivalence.
  Therefore the enriched augmentation is the required
  \v{C}ech-local weak equivalence.
\end{proof}

\begin{definition}\label{def:Cartesian.subsite.restriction.functors}
  Let
  \[
    q:\fembcartnoniso{d}\hookrightarrow\fembnoniso{d}
  \]
  denote the canonical inclusion of the Cartesian subsite in the site without
  isotopies. Restriction along \(q\) defines a functor
  \[
    q^{*}:
    \spsh(\fembnoniso{d})
    \longrightarrow
    \spsh(\fembcartnoniso{d}).
  \]

  Similarly, let
  \[
    \mathfrak{q}:\fembcart{d}\hookrightarrow\femb{d}
  \]
  denote the canonical inclusion of the enriched Cartesian subsite in the
  enriched site with isotopies. Restriction along \(\mathfrak{q}\) defines
  a functor
  \[
    \mathfrak{q}^{*}:
    \psh(\femb{d},\CsSet)
    \longrightarrow
    \psh(\fembcart{d},\CsSet).
  \]

  We use the same notation for the induced functors between the corresponding
  \v{C}ech-local injective model structures whenever these restriction
  functors are regarded as Quillen functors.
\end{definition}

\begin{lemma}\label{lem:qen.comparison.on.cech.local}
  Let \(E_{1},E_{2}\in\psh(\femb{d},\CsSet)\) be \v{C}ech-local fibrant
  objects, \(f:E_{1}\to E_{2}\) a morphism between them and let
  \(\mathfrak{q}^{*}, q^{*}\) be functors from
  \Cref{def:Cartesian.subsite.restriction.functors}. Then
  \[
    f \text{ is a \v{C}ech-local weak equivalence}
    \Longleftrightarrow
    \mathfrak{q}^{*}f \text{ is a \v{C}ech-local weak equivalence}.
  \]
  The same statement holds for \(q^{*}\).
\end{lemma}

\begin{proof}
  We prove the statement for \(\mathfrak{q}^{*}\); the proof for
  \(q^{*}\) is analogous.

  The forward implication is immediate, since a local weak equivalence
  between \v{C}ech-local objects is detected objectwise
  (\Cref{rem:local.equivalences.between.local.objects}), and restriction
  is computed objectwise. We prove the converse.

  Assume that
  \(\mathfrak{q}^{*}f:\mathfrak{q}^{*}E_{1}\to
  \mathfrak{q}^{*}E_{2}\) is a \v{C}ech-local weak equivalence. Fix
  \(p\in\femb{d}\). By
  \Cref{lem:hypercover.replacement.Cartesian}, there is an augmented
  simplicial object \(U_{\bullet}\to\Yo_{p}\) such that
  \[
    U_{n}
    =
    \coprod_{\alpha\in I_{n}}\Yo_{u_{n,\alpha}},
    \qquad
    u_{n,\alpha}\in\fembcart{d},
  \]
  and the augmentation induces a \v{C}ech-local weak equivalence
  \[
    \hocolim_{[n]\in\Delta^{\op}}U_n
    \xrightarrow{\sim}\Yo_p.
  \]

  We write
  \[
    E_i(U_n)\coloneqq \RR\Map_{\psh(\femb{d},\CsSet)}(U_n,E_i).
  \]
  Since all objects are cofibrant in the injective model structure
  (\Cref{rem:injective.presheaves.cofibrant}) and \(E_i\) is
  fibrant, this is computed by the homotopy product
  (\Cref{rem:hyperdescent.coproduct.representables})
  \[
    E_i(U_n)\simeq
    \prod\nolimits^{h}_{\alpha\in I_n}E_i(u_{n,\alpha}),
  \]
  where \(\prod\nolimits^h\) denotes a homotopy product.

  Since \(E_{1}\) and \(E_{2}\) are fibrant \v{C}ech-local objects,
  the preceding local weak equivalence, the derived Yoneda lemma, and the
  universal property of the homotopy colimit give, for \(i=1,2\), weak
  equivalences
  \[
    E_{i}(p)
    \xrightarrow{\sim}
    \holim_{[n]\in\Delta}E_{i}(U_{n}).
  \]
  We obtain a commutative square
  \[
    \begin{tikzcd}
      E_{1}(p) \arrow[r,"\sim"] \arrow[d,"f(p)"']
      & \displaystyle\holim_{[n]\in\Delta}E_{1}(U_{n}) \arrow[d] \\
      E_{2}(p) \arrow[r,"\sim"']
      & \displaystyle\holim_{[n]\in\Delta}E_{2}(U_{n}).
    \end{tikzcd}
  \]

  Every representable summand of every \(U_{n}\) lies in \(\fembcart{d}\).
  Moreover, since covers in \(\fembcart{d}\) are covers in \(\femb{d}\),
  the restrictions \(\mathfrak{q}^{*}E_{1}\) and
  \(\mathfrak{q}^{*}E_{2}\) are \v{C}ech-local objects on
  \(\fembcart{d}\). Hence \(\mathfrak{q}^{*}f\) is a local weak
  equivalence between \v{C}ech-local objects, so it is objectwise on
  \(\fembcart{d}\). In particular, for every representable summand \(u\)
  of every \(U_{n}\), the map \(E_{1}(u)\to E_{2}(u)\) is a weak
  equivalence.

  Consequently the induced maps \(\prod\nolimits^{h}_{\alpha\in
  I_n}E_1(u_{n,\alpha})\to \prod\nolimits^{h}_{\alpha\in I_n}E_2(u_{n,\alpha})\)
  are weak
  equivalences for all \(n\).
  Hence the induced
  map
  \[
    \holim_{[n]\in\Delta}E_{1}(U_{n})
    \longrightarrow
    \holim_{[n]\in\Delta}E_{2}(U_{n})
  \]
  is a weak equivalence. By the 2-out-of-3 property, the map
  \(f(p):E_{1}(p)\to E_{2}(p)\) is a weak equivalence.

  Since \(p\in\femb{d}\) was arbitrary, \(f\) is an objectwise weak equivalence
  on
  \(\femb{d}\). Therefore \(f\) is a weak equivalence in the underlying
  injective model structure, hence also a local weak equivalence in the
  \v{C}ech-local injective model structure.
\end{proof}

\begin{proposition}\label{prop:infty.dense.inclusion.of.sites}
  The canonical inclusion
  \(q:\fembcartnoniso{d}\hookrightarrow\fembnoniso{d}\) induces a left
  Quillen equivalence
  \[
    q^{*}:
    \spsh(\fembnoniso{d})_{\inj,\Cech}
    \xrightarrow{\we_{Q}}
    \spsh(\fembcartnoniso{d})_{\inj,\Cech}
  \]
  for the \v{C}ech-local injective model structures.

  An analogous statement holds for the enriched case: the
  canonical inclusion of enriched sites
  \(\mathfrak{q}:\fembcart{d}\hookrightarrow\femb{d}\) induces a left
  Quillen equivalence
  \[
    \mathfrak{q}^{*}:
    \psh(\femb{d},\CsSet)_{\inj,\Cech}
    \xrightarrow{\we_{Q}}
    \psh(\fembcart{d},\CsSet)_{\inj,\Cech}.
  \]
\end{proposition}

\begin{proof}
  Since \(q^{*}\) and \(\mathfrak{q}^{*}\) are restrictions along full
  inclusions, they preserve objectwise monomorphisms and objectwise weak
  equivalences. Hence both functors are left Quillen for the global
  injective model structures. Write \(q_{*}\) and \(\mathfrak{q}_{*}\) for their
  right
  adjoints, given by right Kan extension.

  We first show that the global Quillen adjunction
  \(\mathfrak{q}^*\dashv\mathfrak{q}_*\) descends to the \v{C}ech-local model
  structures. Let
  \[
    s_{\mathcal{U}}:c_{\mathcal{U}}\longrightarrow\Yo_p
  \]
  be a covering-sieve map in \(\psh(\femb{d},\CsSet)\). Fix
  \(u:T_u\to U_u\in\fembcart{d}\) and \(x\in T_u\). By
  \Cref{lem:femb.stalk.local.trivialization}, the stalk of a
  presheaf on \(\femb{d}\) at \((u,x)\), and the stalk of its restriction to
  \(\fembcart{d}\) at \((u,x)\), are computed by the colimits over the
  shrinking Cartesian submersion-chart neighborhoods of \(x\). Consequently,
  there is a natural isomorphism
  \[
    (\mathfrak{q}^*s_{\mathcal{U}})_{(u,x)}
    \cong
    (s_{\mathcal{U}})_{(u,x)}.
  \]
  The covering-sieve map \(s_{\mathcal{U}}\) is a \v{C}ech-local weak
  equivalence by definition. Therefore
  \Cref{prop:femb.stalk.left.quillen}, applied to
  \(\femb{d}\), shows that \((s_{\mathcal{U}})_{(u,x)}\), and hence
  \((\mathfrak{q}^*s_{\mathcal{U}})_{(u,x)}\), is a weak equivalence in
  \(\CsSet\). Since this holds for every \(u\in\fembcart{d}\) and every
  \(x\in T_u\), \Cref{prop:femb.stalks.detect.local.we}, applied to
  \(\fembcart{d}\), shows that
  \[
    \mathfrak{q}^*s_{\mathcal{U}}:
    \mathfrak{q}^*c_{\mathcal{U}}\longrightarrow\mathfrak{q}^*\Yo_p
  \]
  is a \v{C}ech-local weak equivalence. The localization criterion therefore
  shows that \(\mathfrak{q}^*\dashv\mathfrak{q}_*\) is a Quillen adjunction for
  the \v{C}ech-local injective model structures. The proof that
  \(q^*\dashv q_*\) descends to the unenriched \v{C}ech-local injective model
  structures is analogous, using the unenriched assertion of
  \Cref{prop:femb.stalks.detect.local.we}.

  It remains to show that these Quillen adjunctions are Quillen
  equivalences. We give the argument for \(\mathfrak{q}^{*}\); the argument
  for \(q^{*}\) is the same. Since \(\mathfrak{q}\) is a full inclusion,
  \(\mathfrak{q}^{*}\mathfrak{q}_{*}F\cong F\) naturally for every
  \(F\in\psh(\fembcart{d},\CsSet)\). Since all objects are cofibrant, it
  remains to verify the Quillen-equivalence criterion
  (\Cref{def:quillen.equivalence.criterion}) for arbitrary
  \(E\in\psh(\femb{d},\CsSet)\) and \v{C}ech-fibrant
  \(F\in\psh(\fembcart{d},\CsSet)\):
  \[
    E\xrightarrow{\sim}\mathfrak{q}_{*}F
    \quad\Longleftrightarrow\quad
    \mathfrak{q}^{*}E\xrightarrow{\sim}F.
  \]

  The forward implication follows immediately by applying the left
  Quillen functor \(\mathfrak{q}^{*}\) and using
  \(\mathfrak{q}^{*}\mathfrak{q}_{*}F\cong F\). Conversely, suppose
  \(\eta:\mathfrak{q}^{*}E\to F\) is a local weak equivalence. Choose a
  \v{C}ech-fibrant replacement \(i:E\to aE\), with \(i\) a trivial
  cofibration. Since \(\mathfrak{q}^{*}\) is left Quillen,
  \(\mathfrak{q}^{*}i\) is a trivial cofibration. Since \(F\) is
  \v{C}ech-fibrant, the square
  \[
    \begin{tikzcd}
      \mathfrak{q}^{*}E
      \arrow[r,"\eta"]
      \arrow[d,"\mathfrak{q}^{*}i"']
      &
      F
      \arrow[d]
      \\
      \mathfrak{q}^{*}aE
      \arrow[r]
      &
      *
    \end{tikzcd}
  \]
  admits a lift \(v:\mathfrak{q}^{*}aE\to F\) with
  \(v\circ \mathfrak{q}^{*}i=\eta\). By \(2\)-out-of-\(3\), \(v\) is a
  local weak equivalence.

  Let \(g:aE\to \mathfrak{q}_{*}F\) be the adjoint transpose of \(v\). Since
  \(\mathfrak{q}^{*}\dashv \mathfrak{q}_{*}\) is a Quillen adjunction and \(F\)
  is \v{C}ech-fibrant, \(\mathfrak{q}_{*}F\) is \v{C}ech-fibrant, hence
  \v{C}ech-local. The object \(aE\) is \v{C}ech-fibrant as well, and
  \(\mathfrak{q}^{*}g\cong v\). Therefore
  \Cref{lem:qen.comparison.on.cech.local} implies that
  \(g:aE\to \mathfrak{q}_{*}F\) is a local weak equivalence.

  The adjoint transpose of \(\eta\) is the composite
  \(E\xrightarrow{i}aE\xrightarrow{g}\mathfrak{q}_{*}F\). Since both \(i\) and
  \(g\) are local weak equivalences, \(E\to\mathfrak{q}_{*}F\) is a local weak
  equivalence by \(2\)-out-of-\(3\). This proves the
  Quillen-equivalence criterion, so
  \[
    \mathfrak{q}^{*}:
    \psh(\femb{d},\CsSet)_{\inj,\Cech}
    \longrightarrow
    \psh(\fembcart{d},\CsSet)_{\inj,\Cech}
  \]
  is a left Quillen equivalence.
\end{proof}

\begin{remark}\label{rem:qen.reflects.local.weak.equivalences}
  The same argument shows that
  \[
    \mathfrak{q}^{*}:
    \psh(\femb{d},\CsSet)_{\inj,\Cech}
    \longrightarrow
    \psh(\fembcart{d},\CsSet)_{\inj,\Cech}
  \]
  reflects local weak equivalences. The same holds for
  \(q^{*}\).

  Indeed, let \(f:E_{1}\to E_{2}\) be a morphism such that
  \(\mathfrak{q}^{*}f\) is a local weak equivalence. Choose
  \v{C}ech-fibrant replacements \(E_{i}\xrightarrow{\sim}aE_{i}\), for
  \(i=1,2\), and let \(af:aE_{1}\to aE_{2}\) be the induced morphism.
  Since \(\mathfrak{q}^{*}\) is left Quillen, the maps
  \(\mathfrak{q}^{*}E_{i}\to \mathfrak{q}^{*}aE_{i}\) are local
  weak equivalences. Hence \(\mathfrak{q}^{*}(af)\) is a local weak
  equivalence by \(2\)-out-of-\(3\).

  The objects \(aE_{1}\) and \(aE_{2}\) are \v{C}ech-fibrant, hence
  \v{C}ech-local. Therefore \Cref{lem:qen.comparison.on.cech.local}
  implies that \(af:aE_{1}\to aE_{2}\) is a local weak equivalence.
  Applying \(2\)-out-of-\(3\) to the square
  \[
    \begin{tikzcd}
      E_{1} \arrow[r,"f"] \arrow[d,"\sim"']
      & E_{2} \arrow[d,"\sim"] \\
      aE_{1} \arrow[r,"af"']
      & aE_{2}
    \end{tikzcd}
  \]
  shows that \(f\) is a local weak equivalence.
\end{remark}

\begin{remark}\label{rem:GEmb}
  There is a more general formalism in which \( \femb{d} \) is replaced by a
  \emph{fibered geometric site} \( \gemb{d} \)
  \cite[Definition~4.2.5]{GP26}, with
  \( \femb{d} \)
  as a special case. In that setting, one can study field theories
  and geometric
  structures defined on more general objects, such as
  supermanifolds or derived
  manifolds. In this paper, we are interested only in the study of field
  theories defined on ordinary manifolds, so we restrict our
  computations to \( \femb{d} \)
  throughout to avoid importing the full \( \gemb{} \)--apparatus.
\end{remark}

\section{Geometric functorial field theories}\label{sec:geom.fft}

In this section we connect the homotopy theory of geometric structures from
\Cref{sec:geometric.structures} with moduli spaces of geometric functorial
field theories. We first construct the \emph{Cartesian realization} of a
geometric
structure as an \(\mathrm{O}(d)\)-equivariant simplicial presheaf on
\(\Cart\), and then use it, together with the results of Grady--Pavlov
\cite{GP23,GP26} to obtain the reduction
theorem for moduli spaces of functorial field theories. We give an example of
how to use the theorem in the form of the final corollary which specializes to
the case of the one-dimensional Riemannian structure.

\subsection{Cartesian realization of geometric
structures}\label{subsec:Cartesian.realization}

The full fiberwise-embedding sites \(\fembnoniso{d}\) and \(\femb{d}\) are
inconvenient for explicit calculations. We therefore restrict first to Cartesian
families
and then replace the enriched Cartesian site by the homotopically equivalent
site \(\Cart\times\mathbf{B}\Sing\mathrm{O}(d)\). The resulting geometric
structures are \(\mathrm{O}(d)\)-equivariant simplicial presheaves
on \(\Cart\).

\begin{definition}\label{def:sing.GLd.Od}
  Let
  \[
    \GL(d)\coloneqq
    \{A\in\operatorname{Mat}_{d\times d}(\RR)\mid \det(A)\neq 0\},
    \qquad
    \mathrm{O}(d)\coloneqq
    \{A\in\GL(d)\mid A^{\mathsf{T}}A=I_d\},
  \]
  be topological groups equipped with their usual topologies. Their singular
  complexes
  \(\Sing\GL(d)\) and \(\Sing\mathrm{O}(d)\) are simplicial groups under
  pointwise multiplication; explicitly,
  \[
    \bigl(\Sing\GL(d)\bigr)_n
    =\Top\bigl(|\Delta^n|,\GL(d)\bigr),
    \qquad
    \bigl(\Sing\mathrm{O}(d)\bigr)_n
    =\Top\bigl(|\Delta^n|,\mathrm{O}(d)\bigr).
  \]
  The inclusion \(\mathrm{O}(d)\subseteq\GL(d)\) induces a morphism of
  simplicial groups
  \[
    \Sing\mathrm{O}(d)\longrightarrow\Sing\GL(d).
  \]
  This morphism is a weak equivalence, since the polar-decomposition
  homotopy
  \[
    A\longmapsto A(A^{\mathsf{T}}A)^{-t/2},\qquad 0\leq t\leq 1,
  \]
  is a strong deformation retraction of \(\GL(d)\) onto \(\mathrm{O}(d)\).
\end{definition}

\begin{definition}\label{def:Od.equivariant.presheaves}
  Let \(\mathbf{B}\Sing\mathrm{O}(d)\) be the one-object simplicial
  category whose simplicial endomorphism group is \(\Sing\mathrm{O}(d)\).
  We define the category of
  \emph{\(\mathrm{O}(d)\)-equivariant simplicial presheaves on \(\Cart\)} by
  \[
    \spsh(\Cart)^{\mathrm{O}(d)}
    \coloneqq
    \Fun_{\sset}\bigl(\mathbf{B}\Sing\mathrm{O}(d),\spsh(\Cart)\bigr).
  \]
  Similarly, the category of \emph{\(\mathrm{O}(d)\)-equivariant smooth
  simplicial presheaves on \(\Cart\)} is
  \[
    \psh(\Cart,\CsSet)^{\mathrm{O}(d)}
    \coloneqq
    \Fun_{\sset}\bigl(\mathbf{B}\Sing\mathrm{O}(d),
    \psh(\Cart,\CsSet)\bigr).
  \]
  Thus the superscript \(\mathrm{O}(d)\) denotes an action of the simplicial
  group \(\Sing\mathrm{O}(d)\), rather than an action of the underlying
  discrete group.
\end{definition}

Explicitly, an object of \(\spsh(\Cart)^{\mathrm{O}(d)}\) is a
simplicial presheaf \(X\) together with an action
\[
  \Sing\mathrm{O}(d)\times X\longrightarrow X
\]
that is natural in \(\Cart\), and its morphisms are the equivariant natural
transformations. Whenever either underlying presheaf category is equipped
with a combinatorial simplicial model structure, the corresponding equivariant
functor category carries the projective and injective model structures of
\Cref{thm:enriched.injective.projective.model.structures}. In both cases, weak
equivalences are detected by the forgetful functor to the underlying presheaf
category.

\begin{definition}\label{def:equivariant.lambda}
  Define
  \[
    \lambda^{\infty}_{\mathrm{O}(d)}:
    \spsh(\Cart)^{\mathrm{O}(d)}
    \longrightarrow
    \psh(\Cart,\CsSet)^{\mathrm{O}(d)}
  \]
  by applying the functor \(\lambda^{\infty}\) of
  \Cref{prop:lambda.left.quillen} objectwise to the underlying simplicial
  presheaf and to the \(\Sing\mathrm{O}(d)\)-action maps.
  Equivalently,
  \[
    \lambda^{\infty}_{\mathrm{O}(d)}
    =
    \Fun_{\sset}\bigl(\mathbf{B}\Sing\mathrm{O}(d),\lambda^{\infty}\bigr).
  \]
\end{definition}

\begin{proposition}\label{prop:equivariant.lambda.left.quillen}
  The functor \(\lambda^{\infty}_{\mathrm{O}(d)}\) of
  \Cref{def:equivariant.lambda} is left Quillen for the \v{C}ech-local
  injective model structures. Moreover, its adjunction with objectwise
  evaluation at \(\RR^0\) is a Quillen equivalence.
\end{proposition}

\begin{proof}
  The proof of \Cref{prop:lambda.left.quillen} applies objectwise in
  \(\Fun_{\sset}(\mathbf{B}\Sing\mathrm{O}(d),-)\).
\end{proof}

\begin{remark}\label{rem:presheaf.categories.for.Cartesian.realization}
  In what follows, we will repeatedly apply the general notions recalled in
  \Cref{def:tensored.cotensored,def:tensor.cotensor.presheaves,def:homotopy.weighted.colimit}
  to the simplicial presheaf categories
  \[
    \spsh(\fembnoniso{d}), \qquad \spsh(\Cart), \qquad
    \spsh(\fembcartnoniso{d}).
  \]
  We equip them with local injective model structures. In particular, for
  \(K\in\sset\) and \(F\) an object of one of the above categories, the tensor
  and cotensor are computed objectwise by
  \[
    (K\otimes F)(p)=K\times F(p), \qquad
    (F^{K})(p)=\Map_{\sset}\bigl(K,F(p)\bigr),
  \]
  where \(p\) is an object of \(\fembnoniso{d}\), \(\fembcartnoniso{d}\), or
  \(\Cart\). Moreover, limits and colimits in the above categories are
  computed objectwise (\Cref{rem:presheaves.complete.cocomplete}), and
  homotopy weighted colimits in these categories are given by
  the derived coend of \Cref{def:homotopy.weighted.colimit}.
\end{remark}

\begin{definition}\label{def:rho.equivalence}
  For \(U\in\Cart\), write
  \[
    q_U\coloneqq(\RR^d\times U\longrightarrow U).
  \]
  Define the enriched functor
  \[
    \rho_d:\Cart\times \mathbf{B}\Sing\mathrm{O}(d)
    \longrightarrow \fembcart{d}
  \]
  by
  \[
    \rho_d(U)=q_U,
    \qquad
    \rho_d(g,A)(x,u)=(Ax,g(u))
  \]
  for \(g:U\to V\) and \(A\in\mathrm{O}(d)\). The higher
  simplices are induced by the continuous action of \(\mathrm{O}(d)\) on
  \(\RR^d\). Under the canonical
  identification
  \[
    \psh(\Cart\times\mathbf{B}\Sing\mathrm{O}(d),\CsSet)
    \cong
    \psh(\Cart,\CsSet)^{\mathrm{O}(d)},
  \]
  we equip \(\Cart\times\mathbf{B}\Sing\mathrm{O}(d)\) with the coverage
  in which a family
  \(\{(g_i,A_i):U_i\to U\}_{i\in I}\) is covering if and only if
  \(\{g_i:U_i\to U\}_{i\in I}\) is a covering family in \(\Cart\).
  Restriction along \(\rho_d\) defines a functor
  \[
    \rho^{*}_{\infty,d} : \psh(\fembcart{d},\CsSet)\longrightarrow
    \psh(\Cart,\CsSet)^{\mathrm{O}(d)}
  \]
  given concretely by
  \[
    (\rho^{*}_{\infty,d} F)(U)=F(q_U),
  \]
  with \(\Sing\mathrm{O}(d)\)-action induced by the fiberwise orthogonal
  linear maps \((x,u)\mapsto(Ax,u)\).
\end{definition}

\begin{proposition}\label{prop:rho.on.representables}
  The functor \(\rho_d\) is essentially surjective and homotopically
  fully faithful. More precisely, for every \(U,V\in\Cart\), the induced map
  of mapping spaces
  \[
    \Cart(U,V)\times\Sing\mathrm{O}(d)
    \longrightarrow
    \Map_{\fembcart{d}}(q_U,q_V)
  \]
  is a weak equivalence. Equivalently, for every \(U\in\Cart\), the
  canonical morphism
  \[
    \Yo_U^{\Cart\times\mathbf{B}\Sing\mathrm{O}(d)}
    \longrightarrow
    \rho_{\infty,d}^{*}\Yo_{q_U}^{\fembcart{d}}
  \]
  is an objectwise weak equivalence.
\end{proposition}

\begin{proof}
  Essential surjectivity follows from \Cref{def:fembcart}: every object
  of \(\fembcart{d}\) is isomorphic to \(q_U\) for some \(U\in\Cart\).

  Fix \(U,V,L\in\Cart\) and a base map \(g:U\to V\). An \(L\)-family
  of fiberwise open embeddings \(q_U\to q_V\) over \(g\) is represented by
  a smooth map
  \[
    F:\RR^d\times U\times L\longrightarrow\RR^d
  \]
  such that \(F(-,u,\ell):\RR^d\to\RR^d\) is an open embedding for
  every \((u,\ell)\in U\times L\). For \(s\in[0,1]\), define
  \[
    H_s(F)(x,u,\ell)=
    \begin{cases}
      \displaystyle
      sF(0,u,\ell)+\frac{F(sx,u,\ell)-F(0,u,\ell)}{s},&s>0,\\[1.2ex]
      D_xF(0,u,\ell)x,&s=0.
    \end{cases}
  \]
  This is smooth at \(s=0\), since
  \[
    \frac{F(sx,u,\ell)-F(0,u,\ell)}{s}
    =
    \int_0^1D_xF(rsx,u,\ell)x\,dr.
  \]
  For every \(s\), the map \(H_s(F)(-,u,\ell)\) is an open embedding
  for every \((u,\ell)\in U\times L\). Thus \(H\) is a smooth
  deformation onto the space of linear families
  \[
    x\longmapsto A(u,\ell)x,
    \qquad A:U\times L\longrightarrow\GL(d).
  \]
  Here \(A(u,\ell)=D_xF(0,u,\ell)\) is invertible because
  \(F(-,u,\ell)\) is an embedding between manifolds of the same dimension.
  Polar decomposition gives a deformation of such a family through
  invertible linear families,
  \[
    A_t(u,\ell)
    =A(u,\ell)\bigl(A(u,\ell)^{\mathsf{T}}A(u,\ell)\bigr)^{-t/2},
    \qquad 0\leq t\leq1,
  \]
  from \(A\) to the orthogonal family \(A_1:U\times L\to\mathrm{O}(d)\).
  Choose a point \(u_0\in U\) and a smooth contraction \(h_t:U\to U\) from
  the identity to the constant map with value \(u_0\). Suppressing the fixed
  base map \(g\), define
  \[
    i_L:\mathrm{C}^\infty(L,\mathrm{O}(d))
    \longrightarrow
    \left\{
      \begin{array}{c|c}
        F:\RR^d\times U\times L\to\RR^d
        &F(-,u,\ell)\text{ is an open embedding for every }(u,\ell)
      \end{array}
    \right\}
  \]
  by
  \[
    i_L(B)(x,u,\ell)=B(\ell)x.
  \]
  The preceding construction defines
  \[
    r_L(F)(\ell)=A_1(u_0,\ell)\in\mathrm{O}(d).
  \]
  Indeed, after the rescaling and polar-decomposition homotopies, the
  homotopy
  \[
    (x,u,\ell)\longmapsto A_1(h_t(u),\ell)x
  \]
  ends at \(i_Lr_L(F)\). The resulting homotopy from \(F\) to
  \(i_Lr_L(F)\) is constant on the image of \(i_L\): rescaling fixes linear
  maps, polar decomposition fixes orthogonal matrices, and contraction of
  \(U\) does not change a family independent of \(u\). Thus
  \(r_Li_L=\id\), and these homotopies exhibit the image of \(i_L\) as a
  deformation retract, naturally in \(L\).

  Taking \(L=\Delta^n_{\mathrm{e}}\) gives a homotopy equivalence
  between the mapping simplicial set over the fixed base map \(g\) and the
  smooth singular complex of \(\mathrm{O}(d)\). The latter is weakly
  equivalent to the ordinary singular complex \(\Sing\mathrm{O}(d)\) by
  \cite[Theorem~4.15]{Bun22}. Since the same argument applies to every base
  map \(g:U\to V\), the displayed map in the statement is a weak
  equivalence.
\end{proof}

\begin{proposition}\label{prop:rho.Cech.local.Quillen.equivalence}
  Restriction along the functor \(\rho_d\) of \Cref{def:rho.equivalence}
  is the left adjoint in a Quillen equivalence
  \[
    \rho_{\infty,d}^{*}:
    \psh(\fembcart{d},\CsSet)_{\inj,\Cech}
    \rightleftarrows
    \psh(\Cart,\CsSet)^{\mathrm{O}(d)}_{\inj,\Cech}
    :\rho_{\infty,d,*},
  \]
  where \(\rho_{\infty,d,*}\) is an enriched right
  Kan extension.
\end{proposition}

\begin{proof}
  Put
  \[
    \mathcal{A}\coloneqq
    \Cart\times\mathbf{B}\Sing\mathrm{O}(d),
    \qquad
    \mathcal{B}\coloneqq\fembcart{d}.
  \]
  By \Cref{prop:rho.on.representables}, the functor
  \(\rho_d:\mathcal{A}\to\mathcal{B}\) is essentially surjective and
  induces weak equivalences on mapping objects. Hence it is a
  \(\CsSet\)-weak equivalence in the sense of
  \cite[Definition~2.3(iv)]{GM20}, and
  \cite[Proposition~2.4]{GM20} gives a Quillen equivalence on the global
  enriched presheaf categories. The global projective and injective model
  structures have the same weak equivalences, and restriction along
  \(\rho_d\) preserves objectwise monomorphisms and objectwise weak
  equivalences. Therefore
  \[
    \rho^{*}_{\infty,d}:
    \psh(\mathcal{B},\CsSet)_{\inj}
    \rightleftarrows
    \psh(\mathcal{A},\CsSet)_{\inj}
    :\rho_{\infty,d,*}
  \]
  is a Quillen equivalence for the global injective model structures.

  It remains to compare the \v{C}ech localizations. Let
  \(\mathcal{U}=\{u_i\to u\}_{i\in I}\) be a covering family in
  \(\mathcal{B}\). We show directly that
  \[
    \rho_{\infty,d}^{*}s_{\mathcal{U}}:
    \rho_{\infty,d}^{*}c_{\mathcal{U}}
    \longrightarrow
    \rho_{\infty,d}^{*}\Yo_u
  \]
  is a \v{C}ech-local weak equivalence. By
  \Cref{lem:femb.enriched.covering.sieve}, its evaluation at
  \(V,L\in\Cart\) is the inclusion of those \(L\)-families
  \(\sigma:q_V\to u\) which factor through one of the morphisms
  \(u_i\to u\).

  Let \(X\in\psh(\mathcal{B},\CsSet)\), let \(V\in\Cart\), and let \(v\in V\).
  There is a natural map
  \[
    \bigl(\rho_{\infty,d}^{*}X\bigr)_v
    =
    \colim_{v\in V'\subseteq V}X(q_{V'})
    \longrightarrow
    X_{(q_V,(0,v))},
  \]
  where the right-hand side is the fiberwise stalk of
  \Cref{prop:femb.stalk.left.quillen}. Indeed, every base restriction
  \(q_{V'}\to q_V\) is an open neighborhood
  of \((0,v)\) in \(q_V\). We claim that this map is a weak equivalence in
  \(\CsSet\).

  After choosing coordinates around \(v\), the balls \(B_\varepsilon(v)\)
  form a cofinal system of neighborhoods. Hence
  \[
    \bigl(\rho_{\infty,d}^{*}X\bigr)_v
    \cong
    \colim_{\varepsilon\to0}X(q_{B_\varepsilon(v)}).
  \]
  By
  \Cref{lem:femb.stalk.local.trivialization}, the fiberwise stalk on the right
  is computed by
  \[
    \colim_{\delta,\varepsilon\to0}
    X\bigl(B_\delta^d(0)\times B_\varepsilon(v)
    \longrightarrow B_\varepsilon(v)\bigr).
  \]
  For every \(\delta\) and \(\varepsilon\), this Cartesian submersion is
  isomorphic to \(q_{B_\varepsilon(v)}\). Under such an isomorphism, its
  inclusion into \(q_{B_\varepsilon(v)}\) is represented by a fiberwise open
  embedding \(\RR^d\to\RR^d\) over \(\id_{B_\varepsilon(v)}\). By
  the deformation in the proof of \Cref{prop:rho.on.representables}, this
  fiberwise open embedding is isotopic to an orthogonal linear isomorphism.
  Applying the enriched functor \(X\) to this homotopy gives a homotopy from
  the induced restriction map
  \[
    X(q_{B_\varepsilon(v)})
    \longrightarrow
    X\bigl(B_\delta^d(0)\times B_\varepsilon(v)
    \longrightarrow B_\varepsilon(v)\bigr)
  \]
  to an isomorphism. After applying \(\Sing^\infty\), it is therefore a
  simplicial homotopy equivalence, and hence a weak equivalence in \(\CsSet\)
  by \Cref{prop:Cartesian.model.str.smooth.simp.sets}. Filtered colimits
  preserve weak equivalences in \(\CsSet\) and the claim follows.

  Applying the claim to \(X=c_{\mathcal{U}}\) and \(X=\Yo_u\) gives a
  commutative square
  \[
    \begin{tikzcd}
      \bigl(\rho_{\infty,d}^{*}c_{\mathcal{U}}\bigr)_v
      \ar[r]
      \ar[d,"\sim"']
      &
      \bigl(\rho_{\infty,d}^{*}\Yo_u\bigr)_v
      \ar[d,"\sim"]
      \\
      (c_{\mathcal{U}})_{(q_V,(0,v))}
      \ar[r,"\sim"']
      &
      (\Yo_u)_{(q_V,(0,v))}.
    \end{tikzcd}
  \]
  The bottom map is a weak equivalence by
  \Cref{prop:femb.stalk.left.quillen}, since \(s_{\mathcal{U}}\) is a
  \v{C}ech-local weak equivalence by definition. Hence the top map is a weak
  equivalence by the two-out-of-three property. This holds for every
  \(V\in\Cart\) and every \(v\in V\). Weak equivalences in \(\psh(\mathcal{A},
  \CsSet)\) are
  detected after forgetting the \(\mathrm{O}(d)\)-action, and
  \(\fembcart{0}\cong\Cart\). Therefore
  \Cref{prop:femb.stalks.detect.local.we}, applied with \(d=0\), shows that
  \(\rho_{\infty,d}^{*}s_{\mathcal{U}}\) is a \v{C}ech-local weak
  equivalence. Thus restriction sends every
  \(\mathcal{B}\)-side covering-sieve morphism to an
  \(\mathcal{A}\)-side \v{C}ech-local weak equivalence, and the localization
  criterion gives a Quillen adjunction for the \v{C}ech-local injective model
  structures.

  The
  preceding argument showed that localization at the \(\mathcal{A}\)-side
  covering-sieve morphisms inverts
  \[
    \rho_{\infty,d}^{*}s_{\mathcal{U}}
  \]
  for every \(\mathcal{B}\)-side covering family \(\mathcal{U}\). It remains
  to show that localization at these restricted \(\mathcal{B}\)-side
  morphisms inverts every \(\mathcal{A}\)-side covering-sieve morphism.

  Let
  \(\mathcal{V}=\{(g_i,A_i):U_i\to U\}_{i\in I}\) be a covering family in
  \(\mathcal{A}\). Then \(\rho_d\mathcal{V}\) is a covering family in
  \(\mathcal{B}\). Consider an \(L\)-family \(\sigma:q_W\to q_U\), and
  denote its base map by \(g:W\to U\) and its fiber map by
  \[
    F:\RR^d\times W\times L\longrightarrow\RR^d.
  \]
  The family \(\sigma\) factors through
  \[
    \rho_d(g_i,A_i):q_{U_i}\longrightarrow q_U
  \]
  if and only if \(g\) factors through \(g_i\). Indeed, any such
  factorization induces a map \(h:W\to U_i\) satisfying \(g=g_i h\).
  Conversely, given such an \(h\), the smooth map
  \[
    (x,w,\ell)\longmapsto A_i^{-1}F(x,w,\ell),
  \]
  together with \(h\), defines an \(L\)-family \(q_W\to q_{U_i}\) whose
  composite with \(\rho_d(g_i,A_i)\) is \(\sigma\).

  Every stage of the deformation constructed in the proof of
  \Cref{prop:rho.on.representables} leaves the base map \(g\) unchanged. It
  therefore restricts to the families whose base maps factor through some
  \(g_i\), and gives a commutative square
  \[
    \begin{tikzcd}
      c_{\mathcal{V}}^{\mathcal{A}}
      \ar[r,"s_{\mathcal{V}}"]
      \ar[d,"\sim"']
      &
      \Yo_U^{\mathcal{A}}
      \ar[d,"\sim"]
      \\
      \rho_{\infty,d}^{*}c_{\rho_d\mathcal{V}}^{\mathcal{B}}
      \ar[r,"\rho_{\infty,d}^{*}s_{\rho_d\mathcal{V}}"']
      &
      \rho_{\infty,d}^{*}\Yo_{q_U}^{\mathcal{B}},
    \end{tikzcd}
  \]
  whose vertical maps are objectwise weak equivalences.

  Now consider this square after localizing the \(\mathcal{A}\)-side
  presheaf category at the restrictions of the \(\mathcal{B}\)-side
  covering-sieve morphisms. The bottom horizontal map is a weak equivalence
  by definition, and the vertical maps remain weak equivalences because they
  are objectwise weak equivalences. Hence the two-out-of-three property
  implies that \(s_{\mathcal{V}}\) is a weak equivalence. Thus localization
  at the restricted \(\mathcal{B}\)-side morphisms inverts every
  \(\mathcal{A}\)-side covering-sieve morphism.

  The two collections of morphisms therefore determine the same
  localization. Since every object is cofibrant, restriction computes its
  left derived functor on the localization morphisms. Consequently,
  \cite[Theorem~3.3.20]{Hir03} gives the asserted Quillen equivalence after
  \v{C}ech localization.
\end{proof}

\begin{definition}\label{def:Cartesian.coefficient}
  An \(\mathrm{O}(d)\)-equivariant \emph{Cartesian coefficient} is a
  simplicially enriched functor
  \[
    \mathcal{K}_d:
    \fembcartnoniso{d}
    \longrightarrow
    \spsh(\Cart)^{\mathrm{O}(d)}
  \]
  together with a zigzag of natural \(\mathrm{O}(d)\)-equivariant
  objectwise weak equivalences
  \[
    \lambda^\infty_{\mathrm{O}(d)}\mathcal{K}_d(u)
    \simeq
    \rho_{\infty,d}^{*}\mathfrak{q}^{*}\Yo_{\iota(u)}
  \]
  in \(\psh(\Cart,\CsSet)^{\mathrm{O}(d)}\), for every
  \(u\in\fembcartnoniso{d}\), where \(\Yo_{\iota(u)}\) denotes the enriched
  representable on \(\femb{d}\) and \(\mathfrak{q}^{*}\) is the restriction
  functor of \Cref{def:Cartesian.subsite.restriction.functors}. We fix such a
  Cartesian coefficient \(\mathcal{K}_d\).
\end{definition}

\begin{remark}\label{rem:Cartesian.coefficient.intuition}
  Cartesian coefficients always exist. Indeed, take \(\mathcal{K}_d\) to
  be a functorial model for
  \[
    u\longmapsto
    \RR\ev_{\RR^0}\left(
      \rho_{\infty,d}^{*}\mathfrak{q}^{*}\Yo_{\iota(u)}
    \right).
  \]
  Here \(\RR\ev_{\RR^0}\) denotes the right derived functor of the
  objectwise evaluation functor of
  \Cref{def:constant.smooth.direction}. The derived counit of the Quillen
  equivalence in \Cref{prop:equivariant.lambda.left.quillen} gives the natural
  zigzag of objectwise weak equivalences required in
  \Cref{def:Cartesian.coefficient}.

  The diagram on the right-hand side of the comparison in
  \Cref{def:Cartesian.coefficient} is canonical, but it is often inconvenient
  for explicit calculations. A Cartesian coefficient replaces it by a
  simpler simplicial-presheaf model before passing to smooth simplicial sets.
  For example, in dimension \(1\) one may take
  \[
    \mathcal{K}_1(u)=\mathcal{O}_u\times\Yo_{\flat u},
  \]
  where \(\mathcal{O}_u\) is the set of fiberwise orientations of \(u\). We use
  this model in \Cref{cor:one-dimensional.Riemannian.classification} and give
  more details in~\cite{KP26}.
\end{remark}

Both \(\spsh(\fembcartnoniso{d})\) and \(\spsh(\Cart)\) are simplicial
presheaf categories. In particular, they are simplicially enriched, tensored,
and cotensored, and admit all small limits and colimits. The same holds for
the equivariant category \({\spsh(\Cart)}^{\mathrm{O}(d)}\). Thus, we can use
the language introduced in \Cref{sec:background}, in particular,
\Cref{def:homotopy.weighted.colimit}.

\begin{definition}\label{def:Cartesian.realization.functor}
  Let
  \[
    Q_{\proj}\mathcal{K}_d
    \longrightarrow
    \mathcal{K}_d
  \]
  be a projectively cofibrant replacement in the model category
  \(\Fun(\fembcartnoniso{d},\spsh(\Cart)^{\mathrm{O}(d)})\).
  We define the \emph{(derived) Cartesian realization functor} as
  \[
    \Cc_d(F)
    \coloneqq
    F\otimes_{\fembcartnoniso{d}}Q_{\proj}\mathcal{K}_d.
  \]
\end{definition}

\begin{proposition}\label{prop:Cartesian.realization.Quillen.adjunction}
  There is a simplicially enriched Quillen adjunction
  \begin{equation}\label{eq:Cartesian.realization.adjunction}
    \Cc_d : \spsh(\fembcartnoniso{d})_{\mathrm{inj},\Cech} \rightleftarrows
    \spsh(\Cart)_{\mathrm{inj},\Cech}^{\mathrm{O}(d)} : \mathfrak{R}_{d}
  \end{equation}
  for the \v{C}ech-local injective model structures. In particular, \(\Cc_d\)
  is obtained by deriving the coefficient diagram \(\mathcal{K}_d\) in its
  weighted-colimit presentation.
\end{proposition}

\begin{proof}
  Equip
  \(\Fun(\fembcartnoniso{d}, {\spsh(\Cart)}^{\mathrm{O}(d)})\) with a
  projective model structure and set
  \[
    K \coloneqq Q_{\proj}\mathcal{K}_d
    \in \Fun(\fembcartnoniso{d},\spsh(\Cart)^{\mathrm{O}(d)}).
  \]
  By definition, \(K\) is projectively cofibrant. Since
  \(\spsh(\fembcartnoniso{d})\) and \(\spsh(\Cart)^{\mathrm{O}(d)}\) are
  simplicial model categories, by
  \Cref{prop:weighted.colimit.quillen.bifunctor}, the weighted
  colimit bifunctor
  \[
    - \otimes_{\fembcartnoniso{d}} - :
    \spsh(\fembcartnoniso{d})
    \times
    \Fun(\fembcartnoniso{d},\spsh(\Cart)^{\mathrm{O}(d)})
    \longrightarrow
    \spsh(\Cart)^{\mathrm{O}(d)}
  \]
  is a simplicial left Quillen bifunctor. Therefore, after fixing the
  projectively cofibrant diagram \(K\), the functor
  \[
    F \longmapsto F\otimes_{\fembcartnoniso{d}} K
  \]
  is a simplicial left Quillen functor. By
  \Cref{def:Cartesian.realization.functor}, this is exactly the functor
  \(\Cc_d\).

  Since weighted colimits are simplicially enriched left adjoints, the
  corresponding simplicially enriched right adjoint is given objectwise by
  \[
    (\mathfrak{R}_d G)(u)
    \coloneqq
    \Map_{\spsh(\Cart)^{\mathrm{O}(d)}}
    \bigl(\Cc_d(\Yo_u), G\bigr).
  \]
  Hence \(\Cc_d \dashv \mathfrak{R}_d\) is a simplicially enriched
  Quillen adjunction for the global injective model structures. Since
  \[
    K \longrightarrow \mathcal{K}_d
  \]
  is a projectively cofibrant replacement and every object of
  \(\spsh(\fembcartnoniso{d})\) is cofibrant in the injective model structure,
  \(\Cc_d\) computes the left derived weighted colimit with coefficient
  \(\mathcal{K}_d\).

  It remains to check that the Quillen adjunction
  \Cref{eq:Cartesian.realization.adjunction} descends to the \v{C}ech-local
  injective model
  structures. By the left-adjoint localization criterion
  \cite[Proposition~3.3.18(1)]{Hir03}, it is enough to show that \(\Cc_d\)
  sends each covering-sieve morphism to a \v{C}ech-local weak equivalence.
  Let
  \[
    \mathcal{U}=\{u_i\longrightarrow u\}_{i\in I}
  \]
  be a covering family in \(\fembcartnoniso{d}\), with covering-sieve
  morphism
  \[
    s_{\mathcal{U}}:c_{\mathcal{U}}\longrightarrow\Yo_u.
  \]
  By the projectively cofibrant replacement \(K\to\mathcal{K}_d\), the
  defining comparison for \(\mathcal{K}_d\) in
  \Cref{def:Cartesian.coefficient}, and the enriched co-Yoneda lemma
  (\Cref{def:weighted.colimit}), applying
  \(\lambda^{\infty}_{\mathrm{O}(d)}\) to
  \(\Cc_d(s_{\mathcal{U}})\) gives, up to a zigzag of natural weak
  equivalences, the restriction along \(\rho_d\) of the enriched
  covering-sieve morphism associated to the same family \(\mathcal{U}\) in
  \(\fembcart{d}\). The comparison passes through the weighted colimits
  because \(K\) is projectively cofibrant and all objects are cofibrant in the
  injective model structures. By
  \Cref{prop:rho.Cech.local.Quillen.equivalence}, this restricted enriched
  covering-sieve morphism is a \v{C}ech-local weak equivalence. Hence
  \[
    \lambda^{\infty}_{\mathrm{O}(d)}\Cc_d(s_{\mathcal{U}})
  \]
  is a \v{C}ech-local weak equivalence. Since
  \(\lambda^{\infty}_{\mathrm{O}(d)}\) is a Quillen equivalence by
  \Cref{prop:equivariant.lambda.left.quillen} and all objects are cofibrant, it
  reflects
  weak equivalences. Therefore \(\Cc_d(s_{\mathcal{U}})\) is a
  \v{C}ech-local weak equivalence. The localization criterion now gives the
  asserted Quillen adjunction.
\end{proof}

\begin{proposition}\label{prop:Cartesian.realization.comparison}
  With notation as in
  \Cref{def:useful.functors,prop:lambda.left.quillen,def:equivariant.lambda,def:Cartesian.subsite.restriction.functors,def:Cartesian.coefficient,def:Cartesian.realization.functor,def:rho.equivalence},
  there is a natural weak equivalence of homotopy-cocontinuous
  functors
  \[
    \lambda^{\infty}_{\mathrm{O}(d)}
    \circ \Cc_d\circ q^{*}
    \simeq
    \rho_{\infty,d}^{*}\circ \mathfrak{q}^{*}\circ
    \iota_{!}^{\infty}\circ\lambda^{\infty}
    :
    \spsh(\fembnoniso{d})
    \longrightarrow
    \psh(\Cart,\CsSet)^{\mathrm{O}(d)}.
  \]
\end{proposition}

\begin{proof}
  Set
  \[
    \mathbf{A}
    \coloneqq
    \lambda^{\infty}_{\mathrm{O}(d)}\circ\Cc_d\circ q^{*},
    \qquad
    \mathbf{B}
    \coloneqq
    \rho_{\infty,d}^{*}\circ \mathfrak{q}^{*}\circ
    \iota_{!}^{\infty}\circ\lambda^{\infty}.
  \]
  By
  \Cref{prop:equivariant.lambda.left.quillen,prop:infty.dense.inclusion.of.sites,prop:Cartesian.realization.Quillen.adjunction},
  \(\mathbf{A}\) is a composition of left Quillen functors for the
  \v{C}ech-local injective model structure. By
  \Cref{prop:lambda.left.quillen,prop:iota.left.quillen,prop:infty.dense.inclusion.of.sites,prop:rho.Cech.local.Quillen.equivalence},
  \(\mathbf{B}\) is also left Quillen for the \v{C}ech-local injective model
  structure. In particular, \(\mathbf{A}\) and \(\mathbf{B}\) are left-derived
  homotopy cocontinuous functors.

  We first compare \(\mathbf{A}\) and \(\mathbf{B}\) on Cartesian
  representables. Let \(u\in\fembcartnoniso{d}\), and put
  \(K=Q_{\proj}\mathcal{K}_d\). Since \(q\) is a full inclusion,
  \(q^{*}\Yo_u=\Yo_u\). The enriched co-Yoneda lemma
  (\Cref{def:weighted.colimit}) gives
  \[
    \Cc_dq^{*}(\Yo_u)
    =
    \Yo_u\otimes_{\fembcartnoniso{d}}K
    \cong
    K(u).
  \]
  Consequently, the projectively cofibrant replacement
  \(K\to\mathcal{K}_d\) and the defining comparison in
  \Cref{def:Cartesian.coefficient} give natural weak equivalences
  \[
    \begin{aligned}
      \mathbf{A}(\Yo_u)
      &\cong
      \lambda^{\infty}_{\mathrm{O}(d)}K(u)\\
      &\simeq
      \lambda^{\infty}_{\mathrm{O}(d)}\mathcal{K}_d(u)\\
      &\simeq
      \rho_{\infty,d}^{*}\mathfrak{q}^{*}\Yo_{\iota(u)}.
    \end{aligned}
  \]
  On the other hand, the pointwise description of isotopification gives
  a natural objectwise weak equivalence
  \[
    \iota_{!}^{\infty}\lambda^{\infty}(\Yo_u)
    \longrightarrow
    \Yo_{\iota(u)}.
  \]
  Applying \(\rho_{\infty,d}^{*}\mathfrak{q}^{*}\) gives
  \[
    \mathbf{B}(\Yo_u)
    \simeq
    \rho_{\infty,d}^{*}\mathfrak{q}^{*}\Yo_{\iota(u)}.
  \]
  Hence
  \[
    \mathbf{A}(\Yo_u)\simeq\mathbf{B}(\Yo_u)
  \]
  naturally in \(u\in\fembcartnoniso{d}\).

  Now let \(p\in\fembnoniso{d}\). Choose a split hypercover
  \(U_{\bullet}\to\Yo_p\) such that
  \[
    U_n=\coprod_{i\in I_n}\Yo_{u_{n,i}},
    \qquad
    u_{n,i}\in\fembcartnoniso{d}.
  \]
  By \Cref{lem:hypercover.replacement.Cartesian}, the augmentation induces
  a local weak equivalence
  \[
    \hocolim_{[n]\in\Delta^{\op}}U_n
    \xrightarrow{\sim}
    \Yo_p
  \]
  in \(\spsh(\fembnoniso{d})\). Applying \(\mathbf{A}\) and
  \(\mathbf{B}\), and using homotopy-cocontinuity, gives weak equivalences
  \[
    \mathbf{A}(\Yo_p)
    \simeq
    \hocolim_{[n]\in\Delta^{\op}}\mathbf{A}(U_n),
    \qquad
    \mathbf{B}(\Yo_p)
    \simeq
    \hocolim_{[n]\in\Delta^{\op}}\mathbf{B}(U_n).
  \]
  Since each \(U_n\) is a coproduct of Cartesian representables, the
  comparison already constructed for \(\Yo_u\), with
  \(u\in\fembcartnoniso{d}\), gives
  \[
    \mathbf{A}(U_n)\simeq \mathbf{B}(U_n)
  \]
  for every \(n\).  Therefore
  \[
    \mathbf{A}(\Yo_p)\simeq \mathbf{B}(\Yo_p)
  \]
  naturally in \(p\in\fembnoniso{d}\).

  Since \(\mathbf{A}\) and \(\mathbf{B}\) are homotopy-cocontinuous, they
  preserve homotopy weighted colimits, and hence, by the co-Yoneda lemma
  (\Cref{def:weighted.colimit}), we have
  \[
    \mathbf{A}(F)
    \simeq
    F\otimes^{\mathbb{L}}_{\fembnoniso{d}}(\mathbf{A}\circ\Yo)
    \simeq
    F\otimes^{\mathbb{L}}_{\fembnoniso{d}}(\mathbf{B}\circ\Yo)
    \simeq
    \mathbf{B}(F).
  \]
  Therefore
  \[
    \lambda^{\infty}_{\mathrm{O}(d)}\circ \Cc_d\circ q^{*}
    \simeq
    \rho_{\infty,d}^{*}\circ \mathfrak{q}^{*}\circ
  \iota_{!}^{\infty}\circ\lambda^{\infty}.\qedhere\]
\end{proof}

\subsection{Moduli spaces of geometric functorial field
theories}\label{subsec:moduli.spaces.functorial.field.theories}

We are now ready to define moduli spaces of geometric functorial field theories
and present them in a more tractable form. Doing so requires several
constructions from~\cite{GP26}. Although the underlying theory is substantial,
it is needed only in the proof of our result. The final statement can be
formulated and applied using only the material developed in the present paper.

We use \(\GCat_{\infty,d}^{\otimes}\) for the
\(\CsSet\)-enriched model category of geometric symmetric monoidal
\((\infty,d)\)-categories with isotopies of
\cite[Definition~3.1.8(2)]{GP26}. Its model structure is the \emph{globular}
\(\CsSet\)-enriched left Bousfield localization of the injective model
structure on
\(\psh(\Cart\times\Gamma\times\Delta^{\times d},\CsSet)\)
at the morphisms imposing descent, the symmetric monoidal condition, the
Segal and completeness conditions, and globularity; see
\cite[Proposition~3.1.7]{GP26}. For \(F\in\field{d}\), let
\(\Bord_d^F\in\GCat_{\infty,d}^{\otimes}\) denote the bordism category
with isotopies of \cite[Definition~5.3.16]{GP26}.

\begin{definition}\label{def:functorial.field.theory}
  \textnormal{\cite[Definition~5.3.18]{GP26}.}
  Let \(F\in\field{d}\) and let
  \(\cat{V}\in\GCat_{\infty,d}^{\otimes}\).
  The \emph{smooth simplicial set of \(d\)-dimensional functorial
    field theories valued in \(\cat{V}\) with geometric structure
  \(F\)} is the derived \(\CsSet\)-enriched mapping object
  (\Cref{def:derived.mapping.space})
  \begin{equation}\label{eq:fft.object}
    \FFT_{d,\cat{V}}^{F}
    \coloneqq
    \RR\Map_{\GCat_{\infty,d}^{\otimes}}
    \bigl(\Bord_{d}^{F},\cat{V}\bigr)
    \in\CsSet.
  \end{equation}

  If \(\cat{V}\) is fibrant, then, since every object of
  \(\GCat_{\infty,d}^{\otimes}\) is cofibrant,
  \[
    \FFT_{d,\cat{V}}^{F}
    \simeq
    \Map_{\GCat_{\infty,d}^{\otimes}}
    \bigl(\Bord_{d}^{F},\cat{V}\bigr).
  \]
  In this case a functorial field theory is a morphism
  \[
    Z:\Bord_{d}^{F}\longrightarrow\cat{V}
  \]
  in \(\GCat_{\infty,d}^{\otimes}\).
\end{definition}

\begin{remark}
  One must distinguish the smooth simplicial set of
  functorial field theories from the geometric symmetric monoidal
  \((\infty,d)\)-category of functorial field theories. The construction of the
  latter is
  more subtle because the globular model structure is not compatible with
  the naive internal-hom construction. Grady--Pavlov construct the
  globular functor object in~\cite[Definition~3.2.9]{GP26} and prove in
  \cite[Proposition~3.2.10]{GP26} that it satisfies the expected derived
  tensor--hom adjunction.

  Moreover, the full functor object also retains the external \(\Cart\)- and
  \(\Gamma\)-directions, encoding geometric families of field theories and
  their symmetric monoidal structure. These refinements are easy to obtain from
  our simpler model. Concretely, we can obtain the \(\Cart\)- and
  \(\Gamma\)-directions by replacing
  \(\Bord_d^F\) with
  \begin{equation}\label{eq:more.general.target.for.moduli.space}
    \Yo_{(U,\langle n\rangle,([0],\ldots,[0]))}
    \otimes\Bord_d^F
  \end{equation}
  in \Cref{eq:fft.object}, where the representable presheaf is the external
  product of
  the representables associated to \(U\) and \(\langle n\rangle\). The
  external \(\Cart\)-direction records geometric families of field theories
  and should not be confused with the smooth direction in the
  \(\CsSet\)-enrichment, which records isotopies, while the
  \(\Gamma\)-direction retains the symmetric monoidal structure. As already
  noticed by Lurie~\cite[Remark~2.4.7(a)]{Lur09}, transformations between fully
  extended field theories are invertible, so the geometric symmetric monoidal
  \((\infty,d)\)-category of functorial field theories is already a geometric
  symmetric monoidal \((\infty, 0)\)-category, and hence no additional
  information relevant to the present moduli problem is obtained by allowing
  nontrivial representables in the \(\Delta^{\times d}\)-directions.

  Thus, we use only the smooth simplicial set of
  \Cref{def:functorial.field.theory}.
\end{remark}

\begin{definition}\label{def:core.of.V}
  Let \(\cat{V}\in
  \GCat^{\otimes}_{\infty,d}\) be a geometric
  symmetric monoidal \((\infty,d)\)-category with isotopies.
  Let
  \[
    \iota: \Cart \hookrightarrow \Cart\times
    \Gamma\times\Delta^{\times d},
    \qquad
    U\longmapsto
    \bigl(U,\langle 1 \rangle,([0],\ldots,[0])\bigr)
  \]
  be the inclusion at \(\langle 1 \rangle\) in the \(\Gamma\)-direction
  and at \([0]\) in each of the \(d\) \(\Delta\)-directions, and write
  \[
    \ev\coloneqq \iota^{*}:
    \psh(\Cart\times\Gamma\times\Delta^{\times d},\CsSet)
    \longrightarrow
    \psh(\Cart,\CsSet)
  \]
  for the corresponding restriction functor. The \emph{invertible
  part} (or \emph{core})
  of \(\cat{V}\) is the object
  \[
    \cat{V}^{\times} \coloneqq \RR\ev(\cat{V})
    \ \in\ \psh(\Cart,\CsSet),
  \]
  i.e., the value on \(\cat{V}\) of the \emph{right derived}
  functor of \(\ev\).
  Equivalently, choose a fibrant replacement \(\cat{V}\to
  R\cat{V}\) in
  \(\GCat^{\otimes}_{\infty,d}\) and set
  \(\cat{V}^{\times}\coloneqq \ev(R\cat{V})\).
\end{definition}

\begin{definition}\label{def:field.theory.field.stack}
  \textnormal{\cite[Definition~4.1.3]{GP22}}.
  Let \(\cat{V}\) be a fibrant object in
  \(\GCat^{\otimes}_{\infty,d}\) and
  \(\Bord_{d}^{\Yo_{p}}\) an \emph{embedded bordism category with isotopies}.
  Define the \(\CsSet\)-enriched smooth simplicial presheaf
  \[
    \FFT_d(\cat{V}):\femb{d}^{\op}\ra\CsSet,
    \qquad
    \FFT_d(\cat{V})(p)
    \coloneqq
    \CsSet(\Bord_{d}^{\Yo_{p}},\cat{V}),
  \]
  where \(\Yo_{p}\) is a representable presheaf in
  \(\psh(\femb{d},\CsSet)\).
\end{definition}

\begin{remark}\label{rem:bordism.field.theory.Quillen.adjunction}
  Equivalently, one can define
  \[
    \FFT_d:\GCat_{\infty,d}^{\otimes}\ra\field{d},
    \qquad
    \cat{V}\longmapsto
    \bigl(p\longmapsto
    \CsSet(\Bord_d^{\Yo_p},\cat{V})\bigr),
  \]
  to be the \(\CsSet\)-enriched right adjoint to the bordism
  category functor
  \[
    \Bord_d^{(-)}:\field{d}\ra\GCat_{\infty,d}^{\otimes}.
  \]
  Grady--Pavlov~\cite[Theorem~6.0.2]{GP23} show that
  \(\Bord_d^{(-)}\) is a \(\CsSet\)-enriched left Quillen functor
  for these model structures. In particular, if \(\cat{V}\) is fibrant,
  then \(\FFT_d(\cat{V})\) is fibrant in \(\field{d}\).
\end{remark}

\begin{remark}\label{rem:equivariant.core}
  By
  \Cref{def:Cartesian.subsite.restriction.functors,def:rho.equivalence},
  the restriction
  \[
    \rho_{\infty,d}^{*}\mathfrak{q}^{*}\FFT_d(\cat{V})
    \in\psh(\Cart,\CsSet)^{\mathrm{O}(d)}
  \]
  carries the \(\mathrm{O}(d)\)-action induced by the fiberwise orthogonal
  maps
  \[
    \RR^d\times U\longrightarrow\RR^d\times U,
    \qquad
    (x,u)\longmapsto(Ax,u).
  \]
  When \(d=1\), this is the strict
  \(\mathrm{O}(1)\cong\ZZ/2\)-action induced by reflection in the
  \(\RR\)-factor.
\end{remark}

\begin{theorem}\label{thm:moduli.space.reduction}
  Recall \(\lambda^{\infty}_{\mathrm{O}(d)}\) from
  \Cref{def:equivariant.lambda}, the restriction functors \(q^{*}\) and
  \(\mathfrak{q}^{*}\) from
  \Cref{def:Cartesian.subsite.restriction.functors},
  \(\rho_{\infty,d}^{*}\) from \Cref{def:rho.equivalence},
  \(\Cc_d\) from \Cref{def:Cartesian.realization.functor}, and
  \(\FFT_{d,\cat{V}}^{F}\) and \(\FFT_d(\cat{V})\) from
  \Cref{def:functorial.field.theory,def:field.theory.field.stack},
  respectively. Let \(\cat{V}\in
  \GCat^{\otimes}_{\infty,d}\) be a geometric
  symmetric monoidal \((\infty,d)\)-category with isotopies and assume that
  \(\cat{V}\) is fibrant. Let \(F\in \fieldnoniso{d}\) be a geometric structure
  (\Cref{def:geometric.structures.non.isotopies}) and \(
  \mathfrak{I}_{d}F \) its isotopification
  (\Cref{def:isotopification.functor}).
  Then there is the following weak equivalence between the space of \( d
  \)-dimensional functorial field theories with geometric structure \(
  F \) and a smooth simplicial set of maps between \(\mathrm{O}(d)\)-equivariant
  smooth simplicial presheaves on \(\Cart\):
  \begin{equation*}
    \FFT_{d,\cat{V}}^{\mathfrak{I}_{d}F}\we
    \CsSet\bigl(\lambda^{\infty}_{\mathrm{O}(d)}\Cc_{d}q^{*}(F),\rho^{*}_{\infty,d}\mathfrak{q}^{*}\FFT_{d}(\cat{V})\bigr).
  \end{equation*}
\end{theorem}

\begin{proof}
  We have the following zig-zag of natural weak equivalences, with the right
  column indicating the ambient category for arguments
  \begin{align*}
    \FFT_{d,\cat{V}}^{\mathfrak{I}_{d}F}
    &\stackrel{\mathrm{def}}{=}
    \CsSet\bigl(\Bord_{d}^{\mathfrak{I}_{d}F},\cat{V}\bigr)
    &&\text{\(\GCat_{\infty,d}^{\otimes}\)} \\[1ex]
    &\stackrel{(2)}{\we}
    \CsSet\bigl(\iota_{!}^{\infty}\lambda^{\infty}(F)(-)
    \otimes_{\femb{d}}\Bord_{d}^{\Yo_{(-)}},\cat{V}\bigr)
    &&\text{\(\GCat_{\infty,d}^{\otimes}\)} \\[1ex]
    &\stackrel{(3)}{\cong}
    \CsSet\bigl(\iota_{!}^{\infty}\lambda^{\infty}(F)(-),
    \CsSet\bigl(\Bord_{d}^{\Yo_{(-)}},\cat{V}\bigr)\bigr)
    &&\text{\(\psh(\femb{d},\CsSet)\)} \\[1ex]
    &\stackrel{\mathrm{def}}{=}
    \CsSet\bigl(\iota_{!}^{\infty}\lambda^{\infty}(F),
    \FFT_{d}(\cat{V})\bigr)
    &&\text{\(\psh(\femb{d},\CsSet)\)} \\[1ex]
    &\xrightarrow[(5)]{\we}
    \CsSet\bigl(
      \rho^{*}_{\infty,d}\mathfrak{q}^{*}
      \iota_{!}^{\infty}\lambda^{\infty}(F),
    \rho^{*}_{\infty,d}\mathfrak{q}^{*}\FFT_{d}(\cat{V})\bigr)
    &&\text{\({\psh(\Cart,\CsSet)}^{\mathrm{O}(d)}\)} \\[1ex]
    &\xleftarrow[(6)]{\we}
    \CsSet\bigl(
      \lambda^{\infty}_{\mathrm{O}(d)}\Cc_{d}q^{*}(F),
    \rho^{*}_{\infty,d}\mathfrak{q}^{*}\FFT_{d}(\cat{V})\bigr)
    &&\text{\({\psh(\Cart,\CsSet)}^{\mathrm{O}(d)}\)}
  \end{align*}

  The first line is the
  definition of the space of functorial field
  theories with geometric structure from
  \Cref{def:functorial.field.theory}. The second line follows
  from~\cite[Theorem~6.1.3]{GP26},
  which states that the bordism category functor
  \[
    \Bord_{d}:\field{d}\to \GCat_{\infty,d}^{\otimes},
    \qquad
    F\mapsto \Bord_{d}^{F},
  \]
  is a \( \CsSet \)-enriched left Quillen functor and preserves weak
  equivalences. Applying this to the isotopification of \( F \), and
  unwinding the definition of \( \mathfrak{I}_{d}F \), identifies
  \( \Bord^{\mathfrak{I}_{d}F}_{d} \) with the weighted colimit
  \[
    \iota^{\infty}_{!}\lambda^{\infty}(F)(-)
    \otimes_{\femb{d}}
    \Bord^{\Yo(-)}_{d}.
  \]

  The third line is the enriched tensor--hom adjunction for
  this weighted colimit. Namely, mapping out of the weighted colimit is the
  same as mapping the weight into the smooth simplicial presheaf
  \[
    p\longmapsto
    \CsSet(\Bord^{\Yo_{p}}_{d},\cat{V}).
  \]
  This presheaf is precisely
  \( \FFT_{d}(\cat{V}) \) from \Cref{def:field.theory.field.stack}, which
  gives the fourth line.

  For the fifth line, we pass from smooth simplicial presheaves on \( \femb{d}
  \) to
  \( \mathrm{O}(d) \)-equivariant smooth simplicial presheaves on
  \( \Cart \).
  This is done by applying the composition of the left Quillen equivalences
  \( \mathfrak{q}^{*} \) and \( \rho_{\infty,d}^{*} \), from
  \Cref{prop:infty.dense.inclusion.of.sites} and
  \Cref{prop:rho.Cech.local.Quillen.equivalence}. These restriction functors
  are also right Quillen by~\cite[Proposition~3.3.13]{GP22}. Hence they
  preserve the fibrancy of \(\FFT_{d}(\cat{V})\), which follows from
  \Cref{rem:bordism.field.theory.Quillen.adjunction}. Thus both mapping objects
  compute the corresponding derived mapping objects, and the Quillen
  equivalences induce the displayed weak equivalence.

  The sixth line follows from
  \Cref{prop:Cartesian.realization.comparison}. It identifies the restriction of
  \( \iota^{\infty}_{!}\lambda^{\infty}(F) \) to equivariant presheaves on
  \( \Cart \) with
  \[
    \lambda^{\infty}_{\mathrm{O}(d)}
    \Cc_{d}q^{*}(F).
  \]
  Since this is a weak equivalence between cofibrant objects and the target
  of the mapping object is fibrant, it induces a weak equivalence on mapping
  objects.
\end{proof}

\begin{remark}\label{rem:gch.is.for.noneq}
  If \(\cat{V}\) has duals, the geometric framed cobordism hypothesis
  \cite[Theorem~4.1.8]{GP22} gives, after forgetting the
  \(\mathrm{O}(d)\)-action, a weak equivalence
  \[
    \ev: \rho_{\infty,d}^{*}\mathfrak{q}^{*}\FFT_d(\cat{V})
    \xrightarrow{\sim}
    \cat{V}^{\times}
  \]
  in \(\psh(\Cart,\CsSet)\). If \(\cat{V}\) does not have duals, the same
  statement holds with \(\cat{V}^{\times}\) replaced by the core of the fully
  dualizable part of \(\cat{V}\)~\cite[Definition~2.3.21]{Lur09}. The strict
  action on its source
  determines a homotopy-coherent action on \(\cat{V}^{\times}\), but this
  need not produce a strict action for the specific choice of
  \(\cat{V}^{\times}\) from \Cref{def:core.of.V}. Consequently, the ordinary
  core cannot in general
  replace the target in \Cref{thm:moduli.space.reduction}, whose mapping
  object is formed in the category of strict
  \(\mathrm{O}(d)\)-equivariant presheaves.

  The one-dimensional Riemannian case considered below avoids this issue.
  The computation of the source in~\cite{KP26} presents it as an object with a
  free
  \(\mathrm{O}(1)\cong\mathbb{Z}/2\)-action. The free--forgetful adjunction
  therefore allows us to discard the action before applying the geometric
  framed cobordism hypothesis. This yields a description in terms of the
  ordinary core \(\cat{V}^{\times}\), without choosing a strict action on
  it.
\end{remark}

\begin{corollary}\label{cor:one-dimensional.Riemannian.classification}
  Let \(\cat{V}\in\GCat_{\infty,1}^{\otimes}\) be a geometric
  symmetric monoidal \((\infty,1)\)-category with isotopies, and assume
  that \(\cat{V}\) is fibrant and has duals; see \cite[Section~2.3]{GP22}. Let
  \(\RM\in\fieldnoniso{1}\) be the oriented Riemannian geometric
  structure of \Cref{ex:riemannian.structure}, and let
  \(\rpathgermy\) be the corresponding Riemannian path category functor
  constructed in \cite[Definition~4.1.6]{KP26}.
  Then there is a natural weak equivalence
  \[
    \FFT_{1,\cat{V}}^{\mathfrak{I}_{1}\RM}
    \we
    \CsSet\bigl(
      \lambda^{\infty}\nerve(\rpathgermy),
      \cat{V}^{\times}
    \bigr),
  \]
  where \(\cat{V}^{\times}\) is the core of
  \Cref{def:core.of.V}, with its \(\mathrm{O}(1)\)-action forgotten.
\end{corollary}

\begin{proof}
  Apply \Cref{thm:moduli.space.reduction} with \(d=1\) and
  \(F=\RM\). The resulting mapping object can be further simplified as follows:
  \begin{align*}
    &\CsSet\bigl(\lambda^{\infty}_{\mathrm{O}(1)}
    \Cc_{1}q^{*}(\RM),\rho^{*}_{\infty,1}\mathfrak{q}^{*}\FFT_{1}(\cat{V})\bigr)
    &&\text{\(\psh(\Cart,\CsSet)^{\ZZ/2}\)} \\[1ex]
    &\we
    \CsSet\bigl(\lambda^{\infty}_{\mathrm{O}(1)}(\ZZ/2\times
    \hocolim_{\CM}\Yo_{U}),\rho^{*}_{\infty,1}\mathfrak{q}^{*}\FFT_{1}(\cat{V})\bigr)
    &&\text{\(\psh(\Cart,\CsSet)^{\ZZ/2}\)} \\[1ex]
    &\cong
    \CsSet\bigl(\lambda^{\infty}\hocolim_{\CM}\Yo_{U},
    \rho^{*}_{\infty,1}\mathfrak{q}^{*}\FFT_{1}(\cat{V})\bigr)
    &&\text{\(\psh(\Cart,\CsSet)\)} \\[1ex]
    &\we
    \CsSet\bigl(\lambda^{\infty}\nerve(\rpathgermy),
    \rho^{*}_{\infty,1}\mathfrak{q}^{*}\FFT_{1}(\cat{V})\bigr)
    &&\text{\(\psh(\Cart,\CsSet)\)} \\[1ex]
    &\we
    \CsSet\bigl(\lambda^{\infty}\nerve(\rpathgermy),
    \cat{V}^{\times}\bigr)
    &&\text{\(\psh(\Cart,\CsSet)\)}
  \end{align*}

  The second and fourth lines use
  \cite[Proposition~3.4.7]{KP26} and \cite[Theorem~4.2.1]{KP26},
  respectively. The middle isomorphism is the free--forgetful adjunction
  for \(\ZZ/2\)-equivariant objects; in the last two lines, the action
  on \(\cat{V}^{\times}\) is forgotten. Finally, the last line is an application
  of the one-dimensional geometric framed cobordism hypothesis as explained in
  \Cref{rem:gch.is.for.noneq}.
\end{proof}

\def\doi#1{\href{https://doi.org/#1}{doi:#1}}
\def\arXiv#1{\href{https://arxiv.org/abs/#1}{arXiv:#1}}
\def\gen#1{\href{http://gen.lib.rus.ec/book/index.php?md5=#1}{PDF}}
\def\jstor#1{\href{https://www.jstor.org/stable/#1}{JSTOR:#1}}
\def\numdam#1{\href{http://www.numdam.org/item/#1}{numdam:#1}}


\end{document}